\documentclass{amsart}
\usepackage[left=3 cm,top=3 cm,right= 3 cm,bottom=4cm]{geometry}

\usepackage[utf8]{inputenc}
\usepackage[english]{babel}
\usepackage{amsmath, amssymb, amsthm, mathtools, mathrsfs, tikz, ifthen, lmodern, xcolor, booktabs, cancel, graphicx, nicefrac, changepage, enumerate, dsfont,listings, float, enumitem, booktabs, hyperref,nameref,url}
\usepackage{subcaption}
\usepackage{todonotes}
\DeclareMathOperator*{\argmax}{arg\,max}

\DeclareMathOperator*{\esssup}{ess\,sup}
\DeclareMathOperator*{\essinf}{ess\,inf}

\DeclareMathOperator*{\Interior}{int}
\DeclareMathOperator*{\supp}{supp}

\makeatletter
\def\namedlabel#1#2{\begingroup
   \def\@currentlabel{#2}%
   \phantomsection\label{#1}\endgroup
}
\makeatother

\theoremstyle{plain}
\newtheorem{theorem}{Theorem}[section]
\newtheorem{lemma}[theorem]{Lemma}

\newtheorem{proposition}[theorem]{Proposition}
\newtheorem{remark}[theorem]{Remark}

\theoremstyle{definition}
\newtheorem{definition}[theorem]{Definition}
\newtheorem{assumption}[theorem]{Assumption}

\title[Jump Signals in Optimal Investment]{How much should we care about what others know?\\Jump signals in optimal investment under relative performance concerns}

\author{Peter Bank and Gemma Sedrakjan\\Technische Universit\"at Berlin}
\date{}
\thanks{PB gratefully acknowledges financial support by Deutsche Forschungsgemeinschaft through CRC/TRR 388 ``Rough Analysis, Stochastic
Dynamics and Related Fields”, Projects B03 and B05. GS gratefully acknowledges financial support by Deutsche Forschungsgemeinschaft through IRTG 2544 ``Stochastic
Analysis in Interaction”. We are also grateful to two anonymous referees for their helpful and constructive comments and suggestions.}

\begin{document}

\begin{abstract}
We present a multi-agent and mean-field formulation of a game between investors who receive private signals informing their investment decisions and who interact through relative performance concerns. A key tool in our model is a Poisson random measure which drives jumps in both market prices and signal processes and thus captures common and idiosyncratic noise. Upon receiving a jump signal, an investor evaluates not only the signal's implications for stock price movements but also its implications for the signals received by her peers and for their subsequent investment decisions. A crucial aspect of this assessment is the distribution of investor types in the economy. These types determine their risk aversion, performance concerns, and the quality and quantity of their signals. We demonstrate how these factors are reflected in the corresponding HJB equations, characterizing an agent's optimal response to her peers' signal-based strategies. The existence of equilibria in both the multi-agent and mean-field game is established using Schauder’s Fixed Point Theorem under suitable conditions on investor characteristics, particularly their signal processes. Finally, we present numerical case studies that illustrate these equilibria from a financial-economic perspective. This allows us to address questions such as how much investors should care about the information known by their peers.  
\end{abstract}

\maketitle

\begin{description}
\item[Mathematical Subject Classification (2020)] 91G10, 91A06, 91A16, 91B06 
\item[Keywords] Optimal investment, jump signals, relative performance concern, idiosyncratic information
\end{description}

\tableofcontents

\section{Introduction}
Information is a key driver of investment decisions, especially when shaped by idiosyncratic signals about impending stock price shocks. When combined with relative performance concerns---the competitive drive to measure one's wealth against the average---this creates a compelling dynamic: investors must not only interpret their own signals but also account for their peers’ information flows and resulting actions, leading to questions such as: How much should we care about what others know?

In this paper, we propose a simple and tractable framework that allows us to address such questions. We combine the works of \cite{BK22} and \cite{LZ19} to obtain a financial market model featuring a heterogeneous population of agents who make strategic decisions based on dynamically generated idiosyncratic signals on imminent price shocks. Their goal is to maximize their own expected utility from a ratio between their terminal wealth and a measure of their peers' average wealth.

We start with a finite number of investors who trade between their individual risk-free bank accounts that pay a constant interest rate and individual risky assets modeled by exponential Lévy processes. These processes are driven by an idiosyncratic Brownian motion, a common Brownian motion and a Poisson random measure generating both common and idiosyncratic jump noise. Such an asset specialization reflects real-world observations: investors often concentrate on stocks they are familiar with (cf.\ \cite{Merton1987}), are geographically bound to (cf.\ \cite{CovalMoskowitz1999}), or specialize in certain industries (cf.\ \cite{KSZ2005}). That said, the possibility of one common stock market for all investors is a special case of this setting.

The key modeling feature is the presence of \emph{short-lived idiosyncratic jump signals}: at (random) times the market is subject to a price shock and, at the same instant, some or all investors may receive idiosyncratic signals that carry imperfect information about that imminent price jump. Investors who receive a signal can react instantly by adjusting their fraction of wealth invested in the stock right \emph{before} the price shock materializes. So, intuitively, the timeline of events is as follows: first the investor receives a signal about a price jump, allowing her to adjust her position right before the actual price jump is realized. In other words, the investor reacts to a signal on an impending price shock before the latter materializes.
In \cite{BK22}, the authors develop a rigorous theory for these kinds of jump signals and signal-driven strategies in the context of a single investor's optimal investment problem. Formally, strategies should be measurable with respect to a so-called Meyer-$\sigma$-field that enriches the predictable field by the signal; we refer the reader to \cite{Len80,BB19} and \cite{BK22} for the precise definitions and constructions.

Importantly, in the present paper the information provided by signals may differ between investors: for example, some investors may receive more frequent but less reliable signals than others, or different investors may observe different signal-to-noise ratios. To capture investor heterogeneity, we therefore allow differences both in the information structure (signal arrival intensity and signal quality) and in economic parameters such as risk aversion and performance-concern coefficients. Empirical evidence from \cite{Koijen2014} supports the presence of heterogeneity in risk aversion among fund managers and motivates consideration of agents who differ in their attitudes toward absolute and relative wealth. Each investor's characteristics are collected in her \emph{type vector}, whose distribution is assumed to be common knowledge; this assumption is partially justified by the public observability of past returns and other performance indicators (cf.\ \cite{Koijen2014}), which allow market participants to infer, at least approximately, peers' types. 

Strategic interactions among investors arise from various motivations. For individual investors, these interactions may stem from aspirations to accumulate more wealth than their peers (cf.\ \cite{Abel90,RelativeWealthConcerns}). For fund managers, interactions may reflect career concerns or efforts to attract inflows (cf.\ \cite{BasakMakarov2015}). We model these interactions through relative performance concerns: investors adopt power utility where a concern parameter mediates the trade-off between absolute terminal wealth and a measure of peers' geometric mean terminal wealth. Optimal investment under such relative performance criteria has been studied in \cite{FDR11,EspinosaTouzi2015,BasakMakarov2015,Bielagk,LZ19}, among others.

The relative performance criteria lead to an interdependence of the investors' target functionals and thus introduce a stochastic game of optimal investment which calls for the search for Nash equilibria. To address this, we focus on a class of \textit{signal-driven} strategies, where investment decisions depend only on the signal (if any) received at each time. This choice is motivated by the work of \cite{LZ19}, where the authors construct the unique equilibrium where investors hold time-independent constant positions. In our setting, this corresponds to what we call signal-driven strategies where agents employ time-independent policies reacting to the jump signals received, i.e., they continuously hold a baseline position in the stock, and react with instantaneous adjustments to jump signals.
The existence of a Nash equilibrium with such signal-driven strategies is established in two steps. First, we solve a single investor's best response problem when her peers follow fixed signal-driven strategies. Using the approach in \cite{BK22} we formulate a Hamilton-Jacobi-Bellman (HJB) equation and prove the corresponding verification result which characterizes best response strategies as solutions to one-dimensional concave optimization problems and identifies them to also be signal-driven. This allows us to introduce, in a second step, a fixed-point mapping that encodes the so-called \emph{consistency condition}: each strategy must coincide with its best response given the peers' strategies. Under suitable assumptions on the investors' types, we use Schauder's fixed point theorem to establish the existence of an equilibrium.

The focus on signal-driven strategies is even more natural in a mean field setting where a single agent's action does not impact the overall average considered by her peers. We thus provide a formulation for the corresponding mean field game via randomization over the investors' type vectors. The random type then represents the representative investor for which we can again solve the best response problem having fixed the competitive environment. An important intermediate step is the derivation of the geometric mean wealth in the mean field setting by conditioning on the common noise. For this, we carefully disentangle which part of the randomness is common to all investors and which part is idiosyncratic (and thus averaged out in the mean field limit).

Finally, we prove existence of a mean field equilibrium in two scenarios. First, we analyze the case of finitely many investor types under suitable assumptions, employing arguments analogous to those used in the multi-agent game. Secondly, when the mean field allows for infinitely many different types, we devise an alternative fixed-point map which permits us to prove existence of an equilibrium under the assumption of finitely many distinct common shock marks.

In numerical experiments we investigate how idiosyncratic investment signals affect investors with relative performance concerns. For concreteness we fully specify a mean field game with two types of investors and introduce a suitable certainty-equivalent measure for one of these types. By varying the other type's signal quality and frequency, its share in the population, and its risk preferences, we explore how much an investor cares about what is known by her peers. We confirm the basic intuition that a better-informed peer erodes an investor's welfare while a worse-informed one improves it. Our experiments qualify it in three ways. First, the effect is asymmetric: trailing a well-informed peer costs more than outpacing a poorly-informed one rewards. Second, the effect's magnitude is dictated by the prevalence of the better-informed type, growing from negligible when that type is rare to pronounced when it is common. Third, the two informational dimensions trade off against one another: a peer is genuinely threatening only when her signals are both frequent and precise. Moreover, faced with superior peers, an investor cannot trade her way back to parity, since her equilibrium strategy barely responds to what her peers know; her only effective recourse is to sharpen her own information. Here, quality dominates quantity and the advantage of signal quality only widens as investors grow more risk averse.

\bigskip

\textbf{Related literature.} Multi-agent and mean field games provide a powerful framework for the study of strategic interactions among a large number of agents, such as investors in financial markets. While multi-agent games capture the finite-population dynamics where individual decisions may affect others, mean field games approximate these interactions in large populations by considering the aggregate behavior, which simplifies analysis and computation. Mean field games, introduced independently by \cite{HuangMCMFG} and \cite{LLMFG}, have become influential in financial mathematics over the last two decades; for general treatments see \cite{bensoussan2013,Carda19,ProbabilisticTheoryOfMFG}. For works closest to our setting, see \cite{EspinosaTouzi2015,Bielagk,LZ19}. The former two consider optimal investment for exponential utility with relative performance concerns in Black--Scholes markets; \cite{EspinosaTouzi2015} establishes existence and uniqueness of Nash equilibria, and \cite{Bielagk} studies $n$-agent and mean field formulations via coupled quadratic BSDEs. The latter, \cite{LZ19}, treats CARA and CRRA preferences in a Black--Scholes market with common-noise Brownian motion and obtains unique time-independent constant equilibria; it has inspired a number of subsequent works including \cite{Bauerle23,Bo2022,BWY24,LackerSoretConsumption,Reis22,FBSDE, SZ24}.

None of these works, however, considers in detail the interplay between \emph{asymmetric, short-lived information about market jumps} and relative utility maximization. There is relatively little literature treating jumps as both part of common and idiosyncratic noise. In this direction, \cite{BH24} considers a mean-field portfolio game with idiosyncratic and common jump components modeled via integer-valued random measures; they characterize equilibria for exponential (CARA) utility via McKean--Vlasov FBSDEs with jumps and show existence and uniqueness.

We take a different path and use a Poisson random measure to place, at common jump times, both idiosyncratic and common jump marks. These jump marks account both for jumps in stock prices and for idiosyncratic signals which may carry imperfect information about the price jumps. In particular, our investors can be viewed as `small insiders' who do not affect prices and for whom any information advantage on imminent price shocks is short-lived. It is precisely this short-lived, instantaneous nature of the information advantage that makes the model tractable: because signals are revealed and acted on at the same moment as the price jump, strategic complications such as information leakage (what do my trades reveal about my private signal?) or bluffing (strategically misleading other players) are irrelevant as they would not give our investors an edge.

Finally, the heterogeneity in signals allows us in our numerical illustration to consider agents with different signal frequencies and signal qualities. So, when receiving signals, the agents will muse not only about what the signal means for the ultra-short-term price development, but also about how many of their peers have also received signals and what these entail for their investment choices. Under relative performance concerns, these considerations are reflected in their decision policies and allow us to answer the question \textit{How much should we care about what others know?}

\bigskip

The remainder of this paper is organized as follows. Section \ref{sec:multi_agent_game} introduces the multi-agent game, analyzes the single-agent best-response problem and proves existence of Nash equilibria. Section \ref{sec:mean_field_game} develops the mean field formulation and proves existence of mean field equilibria. Section \ref{sec:numerics} presents numerical experiments and discusses economic implications.

\section{The Multi-Agent Game}\label{sec:multi_agent_game}

This section introduces a multi-agent framework where investors navigate interactive utility maximization in the presence of jump signals. Subsection~\ref{sec:finanial_market_model} presents the market model, the investors' signal processes and their sets of admissible strategies. Subsection~\ref{sec:types} specifies the individual utility functions, accounting for relative performance concerns and heterogeneity in risk aversion and concern parameters, which shape the investors' best response problems. Subsection~\ref{sec:multi_agent_HJB} provides a best response map for a single investor interacting with investors who employ fixed signal-driven strategies. Subsection~\ref{sec:nash_equilibrium} provides a proof of existence of a Nash equilibrium. Finally, Subsection~\ref{sec:limit} establishes a convergence result for aggregate wealth as the number of investors $n$ goes to $\infty$, paving the way for the mean field game.

\subsection{The Financial Market Model with Jump Signals}\label{sec:finanial_market_model}

In the following, we introduce the market model and specify the investors’ admissible strategies.

Let $T>0$ be a finite time horizon and $n \in \mathbb{N}_{\geq 1}$. In our model, randomness is introduced via a probability space $(\Omega,\mathfrak{A},\mathbb{P})$ which supports $n+2$ independent standard Brownian motions $W^0,W^1,...,W^{n+1}$ and an independent Poisson random measure $N=N(dt,de)$. The Poisson measure is compensated by the intensity measure $dt \otimes \nu(de)$ for a positive finite measure $\nu$ on some Polish mark space $E$ equipped with the corresponding Borel-$\sigma$-field $\mathcal{E}=\mathcal{B}(E)$. We denote by $\mathcal{F}=(\mathcal{F}_t)_{t \in [0,T]}$ the right-continuous filtration induced by the completions of 
	\[ \sigma((W^0_s,W^i_s, N([0,s]\times A)): s \in [0,t], A \in \mathcal{E}, i=1,...,n+1), \quad t \in [0,T],\] 
and by $\mathcal{P}=\mathcal{P}(\mathcal{F})$ the predictable $\sigma$-field generated by all left-continuous $\mathcal{F}$-adapted processes on $\Omega \times \mathbb{R}_+$.\\

\textbf{Market Model.} We consider a finite number of $n+1$ investors\footnote{The choice of $n+1$ (rather than just $n$) investors underscores that we are considering at least two investors and allows us later to consider a geometric average over the performance of the $n$ other investors when specifying the preference structure of the given investor.} who trade between their individual risk-free bank accounts $S^{0i}$ with interest rate $r^i \in \mathbb{R}$ and individual risky stocks $S^i$. The asset price dynamics are given by
    \begin{equation}\label{eq:S_N_investors}
        \begin{split}
        S_0^{0i}=s^{0i}_0 >0, \quad  dS_t^{0i} &= r^i S_{t-}^{0i}dt,\quad t \in [0,T],\\
        S^i_0=s_0^i>0, \quad  \frac{dS^i_t}{S^i_{t-}} &=  \kappa^i dt + \sigma^i dW^i_t + \sigma^{0i} dW^0_t + \int_E \eta^i(e) N(dt,de) ,\quad t \in [0,T],
    \end{split}
    \end{equation}
for constants $\kappa^i \in \mathbb{R}$ and $\sigma^i,\sigma^{0i}\geq 0$ such that $\sigma^i +\sigma^{0i} > 0$ for $i=1,...,n+1$ and for measurable functions $\eta^i: E \to \mathbb{R}$ for $i=1,...,n+1$. We will assume
\begin{equation}\label{eq:nonnegative}
\nu(\{\eta^i < -1\})=0,
\end{equation}
to guarantee that each $S^i$ remains nonnegative at all times.
So, while the Brownian motions $W^i$ are idiosyncratic noise for respective investor $i = 1,...,n+1$, the Brownian motion $W^0$ accounts for common noise. The Poisson random measure $N$ introduces jumps by placing marks $e \in E$ at common times for all investors. These marks translate into jump sizes $\eta^i(e)$ which may differ between the investors $i=1,\dots,n+1$.

The Poisson random measure produces \emph{common} jump times with possibly stock-specific price shocks, so this parsimoniously models events that occur simultaneously across the market (e.g.\ macro announcements, sector shocks or common liquidity events) while allowing heterogeneous impacts on different assets or portfolios. In other words, the jump times represent \emph{common} events and the marks capture how those events translate into each investor's risky asset price. In particular, purely independent, idiosyncratic jumps would correspond to $\eta^i(e)\, \eta^j(e) = 0$ for $i\neq j$ for $\nu$-almost all $e \in E$.


The simplest setting, on which we will focus in our numerical experiments, is the single common stock case where all investors trade between the same common bond and the same common stock. This corresponds to the parameter choice where, for all $i=1,...,n+1$, $r^i =r$, $\sigma^i=0$, $\kappa^i = \kappa$, $\sigma^{0i} = \sigma^0$ and $\eta^i(e)=\eta(e)$ for some $r \in \mathbb{R}$, $\kappa \in \mathbb{R}$, $\sigma^0>0$ and some measurable function $\eta : E \to \mathbb{R}$. 
The possibility for idiosyncratic stocks does, however, not add any mathematical difficulties and is thus included here for sake of consistency with the literature; see, e.g., \cite{LZ19}.
It is merely for notational simplicity that we refrain from specifying a multivariate model with several stocks.
\\ 

\textbf{Wealth Dynamics.} Let $x^i_0 > 0$ be the $i$th investor's initial wealth. The $i$th investor's investment strategies are described by a stochastic process $\hat{\phi}^i= (\hat{\phi}^i_t)_{t \in [0,T]}$ specifying the fraction of wealth invested in the risky asset $S^i$ at times $t \in [0,T]$. So, the investor's wealth process $X^i:= (X^i_t)_{t \in [0,T]} := (X^{\hat{\phi}^i}_t)_{t \in [0,T]}$ will evolve accordingly as
    \begin{equation}\label{eq:wealth}
    \begin{split}
        X_0^{i} = x^i_0, \quad \frac{dX_t^{i}}{X^{i}_{t-}} = (1-\hat{\phi}^i_t)\frac{dS_t^{0i}}{S^{0i}_{t-}} + \hat{\phi}^i_t \frac{dS_t^i}{S^i_{t-}}, \quad t \in [0,T].
    \end{split}
    \end{equation}
This will be well-defined for any real-valued predictable $\hat{\phi}^i$ satisfying 
    \begin{equation*}
        \int_0^T (\hat{\phi}^i_t)^2 dt < \infty \text{ a.s. }
    \end{equation*}
In fact, since we assumed the intensity measure $\nu$ to be finite, there will be only finitely many jumps in asset prices in each scenario so that wealth dynamics will even be defined for optional $\hat{\phi}^i$ with the above square-integrability. We refer to \cite{BK22} for the additional integrability required in the case where $\nu$ allows infinitely many jumps and refrain from considering this possibility here for ease of exposition.\\

\textbf{Jump Signals via Meyer-$\sigma$-Fields.} As already mentioned, when deciding over the proportions invested in the risky asset, investors may, at times, have the possibility to recourse to some signal that gives them extra information on impending jumps in the stock price. 

The signal process of the $i$th investor is defined as
    \begin{equation}\label{eq:signal}
        Z^i_t = \int_{\{t\}\times E} \zeta^i(e) N(ds,de), \quad t \in [0,T],
    \end{equation}
for a measurable function $\zeta^i: E \to \mathbb{R}$. Observing a value $Z^i_t\not=0$ allows the investor to infer (typically imperfect) information about the mark $e \in E$ that has been set by the Poisson random measure. As a consequence, she will revise the probability of the corresponding asset prices jumps and even start musing who among her fellow investors will also have received a signal, what this signal may have been and what reaction it will trigger. Obviously, we have to allow our investor to account for the received signal in her investment choice and so we let her choose a strategy which can be written as
    \begin{equation}\label{eq:strategy_form}
        \hat{\phi}^i_t(\omega, Z^i_t(\omega)), \quad (\omega,t) \in \Omega \times [0,T],
    \end{equation}
for some $\mathcal{P}(\mathcal{F})\otimes \mathcal{B}(\mathbb{R})$-measurable field $\hat{\phi}^i:[0,T]\times \Omega \times \mathbb{R}\to \mathbb{R}$. Thus, investor $i$ has to decide in a predictable manner about all the respective positions $\hat{\phi}^i_t(\omega,z)$, $z \in \mathbb{R}$, she will take if in the next moment she was to receive a specific signal $z= Z^i_t(\omega)$. 
Mathematically, by \cite[Corollary 2.2]{BK22}, such strategies are in fact measurable with respect to the Meyer-$\sigma$-field $\Lambda^i := \mathcal{P} \vee \sigma(Z^i)$ generated by the signal process.
We refer to \cite{Len80}, \cite{Karoui} and \cite{BB19} for the detailed theory of Meyer-$\sigma$-fields.

\bigskip 
\textbf{Admissible strategies.} It will be crucial to understand what information each investor can derive from her signals. For instance, the kind of signals received (or not) at any one time will affect the investors' positions that keep their respective wealth nonnegative almost surely. 

We follow the disintegration argument in \cite[pp.1306-1307]{BK22}. That is, we write
    \begin{align*}
        \nu(de \cap \{ \zeta^i \neq 0\}) &=\int_{\zeta^i(E)\setminus \{0\}} K^i(z,de) \mu^i(dz) \quad \text{ for } \mu^i := \nu \circ (\zeta^i)^{-1} \text{ on } \mathbb{R}.
    \end{align*}
Note that this disintegration is indeed possible since $\nu$ is finite and positive, cf. \cite[II. 1.2, p. 65]{LimitProcesses}. The finite measure $\mu^i(dz)$ describes the frequency with which the $i$th investor receives a non-zero signal $z \in \zeta^i(E)\setminus \{0\}$, while $K^i(z,de)$ describes the a posteriori distribution of the marks $e \in E$ given a signal $z \neq 0$. Note that $K^i(z,\{\zeta^i = z\}) =K^i(z,E)=1$.

In moments when investor $i$ receives a signal $z \neq 0$, she can adopt any position $\phi^i_t(z)$ that guarantees nonnegative wealth after any  imminent price jump $\eta^i(e)$ that may come from any mark $e$ that is compatible with the observed signal in the sense that $\zeta^i(e)=z$. To formalize this, we recall \eqref{eq:nonnegative} and introduce tight jump bounds
    \begin{equation*}
        \underline{\eta}^i(z):=K^i(z,\cdot)\text{-ess inf } \eta^i \geq -1 \text{ and } \bar{\eta}^i(z):=K^i(z,\cdot)\text{-ess sup } \eta^i
    \end{equation*}
which allow us to describe the positions ensuring nonnegative wealth by the interval
    \begin{align*}
         \Phi^i(z) := [-1/ \bar{\eta}^i(z), -1 / \underline{\eta}^i(z)], \quad z \in \zeta^i(E)\setminus\{0\}.
     \end{align*}
Note that these intervals are compact under our henceforth standing no arbitrage condition    
	\begin{equation}\label{eq:noarbitrage}
        \underline{\eta}^i(z) < 0 < \bar{\eta}^i(z) \text{ for $\mu^i$-a.e. }  z \in \zeta^i(E) \setminus \{0\},
    \end{equation}
which rules out that investors learn for sure the direction of the next price jump given a signal $z\neq 0$. 

In case of no signal, i.e. when $Z^i_t=z=0$, the investor learns that (if any) a mark can only be set in $\{\zeta^i=0\}$. Thus, if there are no unsignaled shocks, i.e. if $\nu(\zeta^i = 0, \eta^i \neq 0 ) = 0$, she can choose any position in $\Phi^i(0) := \mathbb{R}$. If, by contrast, there may be unsignaled shocks, i.e. if $\nu(\zeta^i = 0, \eta^i \neq 0 ) > 0$, we ensure nonnegative wealth by restricting to positions in
    \begin{equation*}
        \Phi^i(0) := 
        \begin{cases} [-1/ \bar{\eta}^i(0), -1 / \underline{\eta}^i(0) ], \quad & \text{ if } \underline{\eta}^i(0) < 0 < \bar{\eta}^i(0),\\
        [-1/ \bar{\eta}^i(0), +\infty), \quad & \text{ if } 0 \leq \underline{\eta}^i(0) \leq \bar{\eta}^i(0), \quad \\
        (-\infty, -1 / \underline{\eta}^i(0) ], \quad & \text{ if } \underline{\eta}^i(0) \leq  \bar{\eta}^i(0) \leq 0,
        \end{cases}
    \end{equation*}
where the jump bounds in this case are
    \begin{align*}
        \underline{\eta}^i(0)&:= \nu(\cdot \cap\{\zeta^i =0, \eta^i \neq 0\}) \text{-ess inf } \eta^i \geq -1, \\
        \bar{\eta}^i(0)&:= \nu(\cdot \cap\{\zeta^i =0, \eta^i \neq 0\}) \text{-ess sup } \eta^i.
    \end{align*}
Notice that the condition $\sigma^i +\sigma^{0i} >  0$ ensures no arbitrage also when unsignaled shocks can only go in one direction since it ensures that at least one of the Brownian motions $W^i$ or $W^0$ keeps stock prices fluctuating suitably.

This allows us to fix the set of admissible strategies which keep our investors' wealth levels even strictly above zero as follows:

\begin{proposition}[cf. \cite{BK22}, Corollary 2.5]\label{prop:admissible_strategy}
    Suppose \eqref{eq:nonnegative} and \eqref{eq:noarbitrage} hold. Let, for all $z \in \zeta^i(E)\cup\{0\}$, $\bar{\Phi}^i(z)$ be a compact interval included in $\{0\} \cup \Interior\Phi^i(z)$. Then, any strategy $\hat{\phi}^i$ in the set of admissible strategies
        \begin{equation*}
            \mathcal{A}^i := \{ \hat{\phi}^i \text{ of form \eqref{eq:strategy_form} such that } \hat{\phi}^i_t(Z_t^i) \in \bar{\Phi}^i(Z_t^i) \text{ for } t \in [0,T] \text{ a.s.}\},
        \end{equation*}
    admits wealth dynamics $X^i = X^{\hat{\phi}^i}$ solving \eqref{eq:wealth} which remain strictly positive throughout $[0,T]$ almost surely.
\end{proposition}

Note that an investor only holding bonds will never end up with zero wealth and, so, $\bar{\Phi}^i(z)$ can always be chosen to include the possibility of a zero investment in stock. In general, our choice of admissible positions shields the investor from ruin and in fact ensures that she has strictly positive wealth throughout. This insistence on strictly positive wealth will become important later when investors are invited to interact with one another, because the individual optimization problems only remain meaningful if none of the other investors' wealth drops to zero; see the next Section~\ref{sec:types}. It should also be noted that compactness and convexity of $\bar{\Phi}^i(z)$ is essential for applying Schauder's fixed point theorem later.

With the financial market model, the signal and wealth processes of investors subject to admissible strategies at hand, we can move on to their individual optimal investment problems.

\subsection{Risk Preferences, Relative Performance Concerns and Investor Types}\label{sec:types}

We assume that each investor $i=1,...,n+1$ has an individual constant relative risk aversion $\alpha^i \in (0,1) \cup (1,\infty)$. She assesses her wealth $x$ not only by itself but also in view of her peers' average wealth $\bar{x}$. Her specific concern for the others' performance is measured by a parameter $\theta^i \in [0,1]$ which enters her utility function via
    \begin{equation}\label{eq:utility_function}
        u_i(x,\bar{x}):=u_{\alpha^i}(x \bar{x}^{-\theta^i}) := \frac{(x \bar{x}^{-\theta^i})^{1-\alpha^i}}{1-\alpha^i} \quad \text{ for } x,\bar{x}>0.
    \end{equation}
We furthermore let $u_i(0,0) := 0$ and, for $x,\bar{x} \in (0,\infty)$, we fix
    \begin{equation}\label{eq:utility_function_limits}
    	u_{i}(x,0) := \begin{cases}
    		+ \infty &\text{ if } \alpha^i \in (0,1),\\
    		0 &\text{ if } \alpha^i > 1,
    	\end{cases}  \quad \text{ and } \quad  u_{i}(0,\bar{x}) := \begin{cases}
    		0 &\text{ if } \alpha^i \in (0,1),\\
    		-\infty &\text{ if } \alpha^i > 1.
    	\end{cases}
    \end{equation}
For investor $i$, her peers' geometric average wealth at any time $t\in[0,T]$ is of course
    \begin{equation}\label{eq:geometric_mean}
        \bar{X}_t^{-i} := \left( \prod_{j \neq i} X^j_t\right)^{\frac{1}{n}},
    \end{equation}
and so, comparing the terminal average $\bar{X}^{-i}_T$ with her own terminal wealth $X^i_T$, she will assess this investment outcome by its utility
	\begin{equation}\label{eq:utlity_wealth}
		u_i(X^i_T,\bar{X}^{-i}_T) = u_{\alpha^i}(X^i_T / (\bar{X}_T^{-i})^{\theta^i}) = \frac{1}{1-\alpha^i}\bigg(\frac{X^i_T}{(\bar{X}_T^{-i})^{\theta^i}}\bigg)^{1-\alpha^i}.
	\end{equation}

Given this utility structure, a concern parameter $\theta^i$ close to $0$ corresponds to low (or even no) relative performance concerns while $\theta^i$ close to $1$ corresponds to high relative performance concerns of the $i$th investor. In other words, the higher the value of $\theta^i$, the more the investor's concern shifts from her net worth $X^i_T$ to its size compared to the average wealth $\bar{X}^{-i}_T$. Clearly, the geometric mean is chosen for tractability, as it allows us to exploit the homogeneity of the utility function. Economically, this specification captures a simple “keeping up with the Joneses” motive (cf. \cite{Joneses}): investors care not only about how much wealth they accumulate in absolute terms, but also about how well they perform in comparison to others. Note that degeneracies will occur when one of the investors has $X^j_T=0$ as this entails $\bar{X}^{-i}_T=0$, a singularity we will be able to avoid due to our notion of admissible strategies, see Proposition~\ref{prop:admissible_strategy}.

\begin{remark}\label{remark:utility}
    We model the interaction and relative performance concerns as in \cite[Section 3]{LZ19}. 
    The authors in \cite{LZ19} investigate both constant absolute risk aversion (CARA) and constant relative risk aversion (CRRA) utility separately. Since \cite{BK22} considers CRRA utility only, we focus here exclusively on such utility functions as well, making this paper correspond to the special case where $\theta^i=0$. Note also that we omit the case $\alpha^i=1$ corresponding to log-utility, because investors with log-utility are always indifferent to their peers, regardless of $\theta^i$. 
\end{remark}

The characteristics of each investor $i=1,...,n+1$ are now fully described and can be summarized by her corresponding \textit{type (vector)}
    \[ \mathtt{t}^i = (x_0^i,r^i,\kappa^i, \sigma^i, \sigma^{0i},\eta^i,\zeta^i,\alpha^i,\theta^i) \] 
from the \textit{type space}
	\[ \mathtt{T} := (0,\infty) \times \mathbb{R}^2 \times ([0,\infty)^2\setminus\{0,0\}) \times L^0(E) \times L^0(E) \times ((0,\infty)\setminus\{1\}) \times [0,1] \]
where $L^0(E)$ denotes the set of measurable functions mapping $E$ to $\mathbb{R}$. 

Noting that a collection of type vectors $(\mathtt{t}^1, \ldots, \mathtt{t}^{n+1})$ captures all relevant information about the investors involved, we henceforth assume types to be common knowledge, that is, all investors know their own type and the other investors' types.

It will be convenient for our analysis to make the following standing assumption:

\begin{assumption}\label{assumption:admissible_type}
    Our investors' types $(\mathtt{t}^1,...,\mathtt{t}^{n+1}) \in \mathtt{T}^{n+1}$ are such that
    \begin{enumerate}
        \item[$(\mathtt{t}_1)$]\namedlabel{itm:t1}{$(\mathtt{t}_1)$} the stock price remains nonnegative, i.e. $\nu(\{\eta^i < -1 \}) = 0$,
        \item[$(\mathtt{t}_2)$]\namedlabel{itm:t2}{$(\mathtt{t}_2)$} the no-arbitrage condition holds, i.e. $\underline{\eta}^i(z) < 0 < \bar{\eta}^i(z) \text{ for $\mu^i$-a.e. }  z \in \zeta^i(E) \setminus \{0\}$,
        \item[$(\mathtt{t}_3)$]\namedlabel{itm:t3}{$(\mathtt{t}_3)$} the signal map $e\mapsto \zeta^i(e)$ takes only finitely many values, and
        \item[$(\mathtt{t}_4)$]\namedlabel{itm:t4}{$(\mathtt{t}_4)$} jump sizes are sufficiently integrable, in the sense that
        \begin{equation*}
            \int_E (1+\eta^i(e))^p \nu(de) < \infty 
            \text{ for all } p>0.
        \end{equation*}
    \end{enumerate}
\end{assumption}

Here, \ref{itm:t3} means that, while there may be an infinite range of possible marks $e \in E$ placed by the Poisson random measure, each signal function $\zeta^i$, $i=1,\dots,n+1$ classifies this information into a range of only finitely many discrete categories. This is not unreasonable in a real-world context, as this classification helps reduce the complexity of decision making by providing a manageable amount of information that is still sufficiently useful. Mathematically, \ref{itm:t3} will be crucial to get the necessary continuity and compactness we need for our proof of existence of a Nash equilibrium. Condition~\ref{itm:t4} is a non-degeneracy assumption which ensures that expected utility functions remain finite.

Moreover, this assumption allows to show a very convenient property: 

\begin{lemma}\label{lemma:unif_bdd}
    Following an admissible strategy $\hat{\phi}^i \in \mathcal{A}^i$, each investor $i=1,\dots,n+1$ is uniformly shielded from ruin in the sense that her investment returns satisfy
        \[ 0 < \nu\text{-}\essinf (1+ (\hat{\phi}^i \circ \zeta^i)\eta^i). \]
\end{lemma}

\begin{proof}
    By \ref{itm:t3}, $\zeta^i(E)\cup \{0\}$ is a finite set. Thus, we can minimize over all possible signals $z$, positions $x$ and jumps $y$ to find that
        \begin{align*}
            1+ (\hat{\phi}^i_t \circ \zeta^i)\eta^i \geq \min_{z \in \zeta^i(E)\cup\{0\}}\min_{x \in \bar{\Phi}^i(z)} \min_{y \in [\underline{\eta}^i(z),\bar{\eta}^i(z)]} 1 + xy. 
        \end{align*}
    Further, our restriction to positions in $\bar{\Phi}^i(z)$ ensures that $-1 \notin \{ xy : x \in \bar{\Phi}^i(z), y \in [\underline{\eta}^i(z),\bar{\eta}^i(z)]\}$, for respective signals $z \in \zeta^i(E)\cup\{0\}$, and thus the right-hand side is strictly greater than zero.
\end{proof}

Now, given the other investors' strategies $\hat{\phi}^{-i}:=(\hat{\phi}^1,...,\hat{\phi}^{i-1},\hat{\phi}^{i+1},...,\hat{\phi}^{n+1})$ and their corresponding average wealth $\bar{X}^{-i}_T$ of \eqref{eq:geometric_mean}, the $i$th investor's best response problem is to find a strategy $\hat{\phi}^i$ which
    \begin{equation}\label{eq:investor_optimization_problem}
        \text{maximizes }J^{i}(\hat{\phi}^i;\hat{\phi}^{-i}) := \mathbb{E}\big[ u_i(X_T^i,\bar{X}_T^{-i}) \big] \text{ over } \hat{\phi}^i \in \mathcal{A}^i,
    \end{equation}
subject to her wealth dynamics
    \begin{equation}\label{eq:wealth_dynamics}
    \begin{split}
         X^i_0 &= x_0^i,\\
         \frac{dX^i_t}{X_{t-}^i} &=  (r^i+\hat{\phi}_t^i(0)(\kappa^i-r^i))dt + \sigma^i \hat{\phi}_t^i(0) dW^i_t + \sigma^{0i} \hat{\phi}_t^i(0) dW^0_t + \int_{E} \eta^i(e) \hat{\phi}_t^i(\zeta^i(e)) N(dt,de).
    \end{split}
    \end{equation}

Here, whenever there is no signal, the wealth process evolves depending on the agent's zero signal position $\hat{\phi}_t^i(0)$, while at times of a jump signal $\zeta^i(e)$, the wealth process changes with respect to the signal-dependent position $\hat{\phi}_t^i(\zeta^i(e))$ taken just for this moment.
Recognizing the interdependence of the investors' optimization problems through the geometric mean, we find that we are dealing with an $n$-dimensional stochastic differential game which is best understood in terms of its Nash equilibria:

\begin{definition}\label{defn:nashequilibrium}
    A vector of admissible strategies $\hat{\pi}^\star = (\hat{\pi}^{1,\star}, ....,\hat{\pi}^{n+1,\star}) \in \mathcal{A}^1 \times ... \times \mathcal{A}^{n+1}$ is called a \textit{Nash equilibrium}, if for each $i \in \{1,...,n+1\}$ we have
        \begin{equation}\label{eq:optimality_nash_equilibrium_thm}
            J^{i}(\hat{\pi}^{i},\hat{\pi}^{-i,\star}) \leq J^{i}(\hat{\pi}^{i,\star},\hat{\pi}^{-i,\star}) \text{ for all $\hat{\pi}^i \in \mathcal{A}^i$}. 
        \end{equation}
\end{definition}

So, as usual, a collection of strategies $\hat{\pi}^\star = (\hat{\pi}^{1,\star}, ....,\hat{\pi}^{n+1,\star}) \in \mathcal{A}^1 \times ... \times \mathcal{A}^{n+1}$ is called a Nash equilibrium if it is its own best response in the sense that each $\hat{\pi}^{i,\star}$, $i=1,\dots,n+1$ solves~\eqref{eq:investor_optimization_problem} for $\hat{\phi}^{-i}:=\hat{\pi}^{-i,\star}$. While finding such best responses for arbitrary peers' strategies $\hat{\phi}^{-i}=\hat{\pi}^{-i,\star}$ is a daunting task in general, it will turn out to be entirely feasible if they come from the following restricted class:

\begin{definition}\label{defn:signal-driven_strategy}
    Strategies $(\hat{\phi}^1,\dots,\hat{\phi}^{n+1}) \in \mathcal{A}^1 \times \dots \mathcal{A}^{n+1}$ are called \textit{(purely) signal-driven} if each $\hat{\phi}^i$, $i=1,\dots,n+1$, can be written in the form
    \begin{equation}\label{eq:strategyform}
        \hat{\phi}^i_t(\omega, Z^i_t(\omega))  = \phi^i( Z^i_t(\omega)), \quad (\omega,t) \in \Omega \times [0,T],
    \end{equation}
    for some $\mathcal{B}(\mathbb{R})$-measurable time-independent and deterministic function  $\phi^i: \mathbb{R} \to \mathbb{R}$. We denote by
        \[ \mathcal{A}_{\text{sig}}^i := \{ \hat{\phi}^i \in \mathcal{A}^i \text{ signal-driven in the sense of } \eqref{eq:strategyform} \}  \]
    the set of admissible signal-driven strategies for the $i$th investor.
\end{definition}

\begin{definition}\label{defn:signaldrivennashequilibrium}
    A Nash equilibrium $\pi^\star = (\pi^{1,\star}, ....,\pi^{n+1,\star })$ is called a \textit{signal-driven Nash equilibrium} if it is signal-driven in the sense of Definition \ref{defn:signal-driven_strategy}, i.e. if  $(\pi^{1,\star}, ....,\pi^{n+1,\star }) \in \mathcal{A}_{\text{sig}}^1 \times ... \times \mathcal{A}_{\text{sig}}^{n+1}$.
\end{definition}

Note that either notion of Nash equilibrium presented here is  of open-loop type, that is, the equilibrium strategies are best responses within a class of strategies written in terms of the system's driving noise.

\subsection{HJB equation and Best Response Map for the Multi-Agent Game}\label{sec:multi_agent_HJB}
The main result of this section is a verification theorem for best response controls for an investor who interacts with investors employing signal-driven strategies. As we will see, the best response in the set of \emph{all} admissible strategies $\mathcal{A}^i$ is itself signal-driven, i.e., it comes from the smaller set $\mathcal{A}^i_{\text{sig}}$. 

This result will be obtained by dynamic programming techniques similar to the case with one investor solved in \cite{BK22}. As in \cite{LZ19}, the key difference here is of course that each investor $i\in \{1,...,n+1\}$ has to take into account her peers' average mean wealth~\eqref{eq:geometric_mean}. When they pursue admissible, signal driven strategies $\pi^j \in \mathcal{A}^j_{\text{sig}}$, $j \in \{1,...,n+1\}\setminus\{i\}$, we can use It\^{o}'s formula to compute the dynamics  of this average, which, conveniently, turns out to be an exponential Lévy process:
\begin{equation}\label{eq:averageLevyStructure}
    \begin{split}
        \bar{X}^{-i}_t = \bar{x}_0 \exp\bigg( (\overline{\tau\pi})^{-i} t &+ \frac{1}{n} \sum_{j \neq i}  \sigma^j \pi^j(0) W^j_t + (\overline{\sigma^0\pi})^{-i} W^0_t \\
                &+  \int_{[0,t] \times E}  \log((\overline{1 + (\pi\circ \zeta) \eta})^{-i} (e) )  N(ds,de)  \bigg)
    \end{split}
    \end{equation}
with dynamics given by    
    \begin{equation}\label{eq:geom_wealth_dynamics}
    \begin{split}
        \bar{X}^{-i}_0 &= \bar{x}_0 := \bigg( \prod_{j\neq i} x^j_0 \bigg)^{\frac{1}{n}}\\
        \frac{d \bar{X}^{-i}_t}{\bar{X}^{-i}_{t-}} &=  (\overline{\tau\pi})^{-i} dt  + \frac{1}{2} [ ((\overline{\sigma^0\pi})^{-i})^2 + (\overline{\sigma^2\pi^2})^{-i}] dt + \frac{1}{n} \sum_{j \neq i}  \sigma^j \pi^j(0) dW^j_t \\
        & \qquad \qquad \qquad +(\overline{\sigma^0\pi})^{-i} dW^0_t  + \int_E \big[ (\overline{1 + (\pi\circ \zeta) \eta})^{-i} (e) -1  \big]  N(dt,de), \quad t \in [0,T],
    \end{split}
    \end{equation}
where
    \begin{align*}
        (\overline{\tau\pi})^{-i} &:= \frac{1}{n} \sum_{j \neq i} \big[r^j + \pi^j(0)(\kappa^j-r^j) - \frac{1}{2} ((\sigma^j)^2 + (\sigma^{0j})^2) (\pi^j(0))^2 \big] , \quad 
        (\overline{\sigma^0\pi})^{-i} :=\frac{1}{n} \sum_{j \neq i}\sigma^{0j} \pi^j(0) \\ (\overline{\sigma^2\pi^2})^{-i} &:= \frac{1}{n^2} \sum_{j\neq i}(\sigma^j)^2(\pi^j(0))^2 \text{ and } (\overline{1 + (\pi\circ \zeta) \eta})^{-i}(e) := \prod_{j \neq i} (1+ \pi^j(\zeta^j(e))\eta^j(e))^{\frac{1}{n}}.
    \end{align*}
Bringing together the perspectives from~\cite{LZ19} and~\cite[pp.1312-1314]{BK22}, we can now derive the HJB equation for the best response problem~\eqref{eq:investor_optimization_problem}. The corresponding value function is
    \begin{equation*}
        v^i(x^i_0,\bar{x}_0,T) :=\sup_{\hat{\phi}^i \in \mathcal{A}^i} \mathbb{E}\big[u_{i}(X^i_T, \bar{X}^{-i}_T)\big].
    \end{equation*}
Assuming $v^i$ to be smooth, we expect by It\^{o}'s formula that, for any admissible strategy $\hat{\phi}^i$, the value process $V_t^i := V_t^{\hat{\phi}^i,\pi^{-i}} := v^i(X^i_t,\bar{X}^{-i}_t,T-t)$, $t \in [0,T]$, to have the dynamics:
    \begin{align*}
    v^i(X^i_t,\bar{X}^{-i}_t,T-t) &= v^i(x^i_0, \overline{x}_0,T) - \int_0^t \partial_T v^i(X^{i}_{s-},\bar{X}^{-i}_{s-},T-s) ds \\
        & \quad + \int_0^t \partial_x v^i(X^{i}_{s-},\bar{X}^{-i}_{s-},T-s)dX_{s}^{i,c} + \int_0^t \partial_{\bar{x}} v^i(X^{i}_{s-},\bar{X}^{-i}_{s-},T-s) d\bar{X}^{-i,c}_s  \\
        &\quad + \int_0^t \partial_{x\bar{x}} v^i(X^i_{s-},\bar{X}^{-i}_{s-},T-s)d[X^i,\bar{X}^{-i}]^c_{s} + \frac{1}{2} \int_0^t \partial_{xx} v^i(X^{i}_{s-},\bar{X}^{-i}_{s-},T-s) d[X^i,X^i]_s^{\mathtt{c}}  \\
        &\quad + \frac{1}{2} \int_0^t \partial_{\bar{x}\bar{x}} v^i(X^{i}_{s-},\bar{X}^{-i}_{s-},T-s) d[\bar{X}^{-i},\bar{X}^{-i}]_s^{c}  \\
        & \quad + \sum_{0 < s \leq t} \left[ v^i(X^{i}_s,\bar{X}^{-i}_s,T-s) - v^i(X^i_{s-},\bar{X}^{-i}_{s-},T-s) \right],
    \end{align*}
where the superscript $(\cdot)^{c}$ denotes the continuous part of a process. Inserting the respective dynamics \eqref{eq:wealth_dynamics} and \eqref{eq:geom_wealth_dynamics}, we get
    \begin{equation*}
    \begin{split}
        &v^i(X^i_t,\bar{X}^{-i}_t,T-t) \\&= v^i(x^i_0, \bar{x}_0,T) - \int_0^t \partial_T v^i(X^{i}_{s-},\bar{X}^{-i}_{s-},T-s) ds \\
        & \quad + \int_0^t \partial_x v^i(X^{i}_{s-},\bar{X}^{-i}_{s-},T-s) X^i_{s-}\left[(r^i+\hat{\phi}^i_s(0)(\kappa^i-r^i))ds + \sigma^i \hat{\phi}^i_s(0) dW^i_s + \sigma^{0i} \hat{\phi}^i_s(0) dW^0_s\right] \\
        &\quad + \int_0^t \partial_{\bar{x}} v^i(X^{i}_{s-},\bar{X}^{-i}_{s-},T-s) \bar{X}^{-i}_{s-}\bigg[ (\overline{\tau\pi})^{-i} ds  + \frac{1}{2} [ ((\overline{\sigma^0\pi})^{-i})^2 + (\overline{\sigma^2\pi^2})^{-i}] ds \\
        &\qquad \qquad \qquad \qquad  \qquad \qquad \qquad \qquad  \qquad \qquad + \frac{1}{n} \sum_{j \neq i}  \sigma^j \pi^j(0) dW^j_s + (\overline{\sigma^0\pi})^{-i}dW^0_s \bigg]\\
        & \quad + \int_0^t \partial_{x\bar{x}} v^i(X^i_{s-},\bar{X}^{-i}_{s-},T-s) X^{i}_{s-} \bar{X}^{-i}_{s-} \sigma^{0i} \hat{\phi}^i_s(0) (\overline{\sigma^0\pi})^{-i} ds \\
        &\quad + \frac{1}{2} \int_0^t \partial_{xx} v^i(X^{i}_{s-},\bar{X}^{-i}_{s-},T-s)(X^i_{s-})^2 \left[ (\sigma^i)^2 + (\sigma^{0i})^2 \right] (\hat{\phi}^i_s(0))^2 ds \\
        &\quad + \frac{1}{2} \int_0^t \partial_{\bar{x}\bar{x}} v^i(X^{i}_{s-},\bar{X}^{-i}_{s-},T-s) (\bar{X}^{-i}_{s-})^2 [ ((\overline{\sigma^0\pi})^{-i})^2 + (\overline{\sigma^2\pi^2})^{-i}]  ds \\
        & \quad + \int_{[0,t] \times E} \bigg[ v\left(X^i_{s-}(1+\hat{\phi}^i_s(\zeta^i(e))\eta^i(e)),\bar{X}^{-i}_{s-}(\overline{1 + (\pi\circ \zeta) \eta})^{-i} (e),T-s\right) \\
        &\qquad \qquad \qquad \qquad  \qquad \qquad \qquad \qquad  \qquad \qquad \qquad  - v^i(X^i_{s-},\bar{X}^{-i}_{s-},T-s) \bigg] \big(\bar{N}(ds,de) + ds\otimes \nu(de)\big),
    \end{split}
    \end{equation*}
where, $\bar{N}$ is the compensated Poisson random measure $\bar{N}(ds,de) = N(ds,de) - ds\otimes \nu(de)$. 

Following the martingale optimality principle, we look for conditions such that the value process is a supermartingale for any admissible $\hat{\phi}^i$ and a martingale for some (then) optimal $\hat{\phi}^{i,\star}$. This leads us to the following HJB equation:

\begin{equation}\label{eq:investor_HJB}
    \begin{split}
         &-\partial_T v^i(x,\bar{x},T) + \partial_x v^i(x,\bar{x},T) x r^i  + \partial_{\bar{x}} v^i(x,\bar{x},T)\bar{x} [(\overline{\tau\pi})^{-i} + \tfrac{1}{2} [((\overline{\sigma^0\pi})^{-i})^2 +  (\overline{\sigma^2\pi^2})^{-i} ]] \\
         &\quad + \tfrac{1}{2} \partial_{\bar{x}\bar{x}} v^i(x,\bar{x},T) \bar{x}^2 [ ((\overline{\sigma^0\pi})^{-i})^2 + (\overline{\sigma^2\pi^2})^{-i}]   \\
        &\quad + \sup_{\varphi \in \bar{\Phi}^i(0)} \bigg\{ \partial_x v^i(x,\bar{x},T) x \varphi (\kappa^i-r^i)) + \partial_{x\bar{x}} v^i(x,\bar{x},T)x\bar{x} \sigma^{0i} \varphi (\overline{\sigma^0\pi})^{-i} \\
        & \quad \quad \quad \quad \quad \quad \quad \quad + \tfrac{1}{2} \partial_{xx} v^i(x,\bar{x},T)x^2 \big[ (\sigma^i)^2 + (\sigma^{0i})^2 \big] \varphi^2 \\
        &\quad \quad \quad \quad \quad \quad \quad \quad + \int_{\{ \zeta^i = 0\}}  \bigg[ v^i\big(x(1+\varphi\eta^i(e)),\bar{x} (\overline{1 + (\pi\circ \zeta) \eta})^{-i} (e),T\big) - v^i(x,\bar{x},T) \bigg] \nu(de) \bigg\}\\
        &\quad + \int_{\zeta^i(E) \setminus\{0\}}  \sup_{\varphi \in \bar{\Phi}^i(z)} \bigg\{ \int_{\{\zeta^i =z\}} \bigg[ v^i\big(x(1+\varphi \eta^i(e)),\bar{x} (\overline{1 + (\pi\circ \zeta) \eta})^{-i} (e),T\big) \\
        &\qquad \qquad \qquad \qquad \qquad \qquad \qquad \qquad \qquad \qquad \qquad \qquad \qquad \qquad - v^i(x,\bar{x},T) \bigg] K^i(z,de)  \bigg\}  \mu^i(dz) = 0
    \end{split}
    \end{equation}
    with boundary condition 
    \[ v^i(x,\bar{x},0) = u_{\alpha^i}(x \bar{x}^{-\theta^i}) = \frac{(x \bar{x}^{-\theta^i})^{1-\alpha^i}}{1-\alpha^i}.\]

To prepare the verification theorem for the HJB equation, we include the following auxiliary lemma.

\begin{lemma}\label{lemma:majorante}
    For all $p> 0$, there exist $L^p(\nu)$-integrable functions $f:E \to \mathbb{R}$ and $g:E \to \mathbb{R}$ such that for all $i \in \{1,...,n+1\}$ we have:
    
    For all $z \in \zeta^i(E) \cup \{0\}$, for all $\varphi \in \bar{\Phi}^i(z)$
        \[ (1+\varphi\eta^i(e)) \leq f(e), \]
    and, for all admissible signal-driven strategies $\pi^{-i}=(\pi^j)_{j\neq i}$,
		\begin{align*}
            (\overline{1 + (\pi\circ \zeta) \eta})^{-i} (e)\leq g(e).
        \end{align*}
\end{lemma}
\begin{proof}
	For the first statement, observe that
        \[ (1+\varphi\eta^i(e)) \leq (1+\max_{x \in \bar{\Phi}^i(z)}|x|)(1+|\eta^i(e)|), \]
    where the right hand-side is independent from $\varphi$ and its $p$th power is $\nu$-integrable by \ref{itm:t4}. Similarly, for the second statement
        \begin{align*}
            (\overline{1 + (\pi\circ \zeta) \eta})^{-i} (e) &= \prod_{j \neq i}  (1+ \pi^j(\zeta^j(e))\eta^j(e))^{\frac{1}{n}} \\
            & \leq \prod_{j \neq i} \big(\max_{z \in \zeta^j(E) \cup \{0\}} (1 + |\pi^j(z)|)^{\frac{1}{n}} \big)(1+|\eta^j(e)|)^{\frac{1}{n}} \\
            &\leq \max_{j\neq i} \max_{z \in \zeta^j(E) \cup \{0\}} \max_{x \in \bar{\Phi}^i(z)} (1+|x|)(1+|\eta^j(e)|)
        \end{align*}
    where the last line is independent from $\pi^{-i}$ and its $p$th power is $\nu$-integrable by \ref{itm:t4}.
\end{proof}

With this result in place, we present the verification theorem for the best response problem and its proof:

\begin{theorem}\label{thm:verification_n}
    Assume that our type assumption~\ref{assumption:admissible_type} holds and suppose investor $i$'s peers use admissible signal-driven strategies $\pi^{-i}=(\pi^j)_{j \in \{1,...,n+1\}\setminus \{i\}}$.
    
    Then the value function of the $i$th investor's best response problem \eqref{eq:investor_optimization_problem} is given by
    \begin{equation}\label{eq:value_function_n_investor}
            v^i(x^i_0,\bar{x}_0,T) = u_{\alpha^i}(x^i_0 (\bar{x}_0)^{-\theta^i})e^{T(1-\alpha^i)(r^i + M_{\pi^{-i}}^i)}
    \end{equation}
    where
    \begin{equation}\label{eq:M_investor}
    \begin{split}
        M^i_{\pi^{-i}} &:=  - \theta^i (\overline{\tau\pi})^{-i} + \tfrac{1}{2} (\theta^i)^2(1-\alpha^i) [ ((\overline{\sigma^0\pi})^{-i})^2 + (\overline{\sigma^2\pi^2})^{-i} ] \\
        & \quad + \sup_{\varphi \in \bar{\Phi}^i(0)} 
        \bigg\{ (\kappa^i-r^i) \varphi - \alpha^i\tfrac{1}{2}  [(\sigma^i)^2+(\sigma^{0i})^2] \varphi^2  - \theta^i (1-\alpha^i)\sigma^{0i} \varphi (\overline{\sigma^0\pi})^{-i} \\
        & \quad \quad \quad \quad \quad  \quad \quad   + \int_{\{ \zeta^i = 0\}} \bigg[ u_i\left(1+\varphi\eta^i(e), (\overline{1 + (\pi\circ \zeta) \eta})^{-i} (e)\right) - u_i(1,1) \bigg] \nu(de)  \bigg\}\\
        &\quad + \int_{\zeta^i(E) \setminus\{0\}}  \sup_{\varphi \in \bar{\Phi}^i(z)} \bigg\{ \int_{\{\zeta^i = z\}} \bigg[ u_i\left(1+\varphi \eta^i(e), (\overline{1 + (\pi\circ \zeta) \eta})^{-i} (e)\right)  - u_i(1,1) \bigg] K^i(z,de)  \bigg\} \mu^i(dz)
    \end{split}
    \end{equation}
    is a finite constant. The function $v^i$ is indeed a smooth solution of the HJB equation \eqref{eq:investor_HJB} and the $i$th investor's best response control is unique and it is again an admissible signal-driven strategy $\phi^{i,\star}:= \phi^{i,\star}_{\pi^{-i}} \in \mathcal{A}_{\text{sig}}^i$: when receiving no signal, investor $i$'s best response to $\pi^{-i}$ is holding in stock a fraction of total wealth
    \begin{equation}\label{eq:optimal_strategy_zero}
    \begin{split}
        \phi^{i,\star}_{\pi^{-i}}(0) &= \argmax_{\varphi \in \bar{\Phi}^i(0)} \bigg\{ (\kappa^i-r^i) \varphi - \alpha^i\tfrac{1}{2}  [(\sigma^i)^2+(\sigma^{0i})^2] \varphi^2  - \theta^i (1-\alpha^i)\sigma^{0i} \varphi (\overline{\sigma^0\pi})^{-i} \\
        & \quad \quad \quad \quad \quad  \quad \quad   + \int_{\{ \zeta^i = 0\}} \bigg[ u_i\left(1+\varphi\eta^i(e), (\overline{1 + (\pi\circ \zeta) \eta})^{-i} (e)\right) - u_i(1,1) \bigg] \nu(de)  \bigg\}
     \end{split}
     \end{equation}
     and, at moments when receiving a signal $z\neq 0$, changes this to a fraction of size
     \begin{equation}\label{eq:optimal_strategy_nonzero}
     	\phi^{i,\star}_{\pi^{-i}}(z) = \argmax_{\varphi \in \bar{\Phi}^i(z)} \bigg\{ \int_{\{\zeta^i = z\}} \bigg[ u_i\left(1+\varphi \eta^i(e), (\overline{1 + (\pi\circ \zeta) \eta})^{-i} (e)\right) - u_i(1,1) \bigg] K^i(z,de) \bigg\}.
    \end{equation}
\end{theorem}


\begin{proof}(cf. \cite[proof of Theorem 3.1]{BK22})
    Let us first argue that $M_{\pi^{-i}}^i \in \mathbb{R}$. For this, as by \ref{itm:t3} $\mu^i$ is a finite measure, we only need to show finiteness of the suprema in \eqref{eq:M_investor}, for which we take a closer look at its integrands. Observe that for all $z \in \zeta^i(E) \cup \{0\}$, for all $\varphi \in \bar{\Phi}^i(z)$ we have 
    \begin{equation}\label{eq:majorante}
        \left| u_i\left(1+\varphi \eta^i(e), (\overline{1 + (\pi\circ \zeta) \eta})^{-i} (e)\right) \right| \leq \frac{(1+\varphi \eta^i(e))^{1-\alpha^i}}{|1-\alpha^i|} \big((\overline{1 + (\pi\circ \zeta) \eta})^{-i}(e)\big)^{-\theta^i(1-\alpha^i)}.
    \end{equation}
    Thus, both for $\alpha^i \in (0,1)$ and $ \alpha^i \in (1,\infty)$, the right-hand side exhibits a denominator which is uniformly bounded away from zero by Lemma \ref{lemma:unif_bdd}, and a numerator which is bounded from above by a $\nu$-integrable function, which does not depend on $\varphi$, by Lemma \ref{lemma:majorante}. Hence, both suprema in \eqref{eq:M_investor} are taken for continuous functions over compacts and thus are attained and finite.
    
    Now, a simple calculation shows that $v^i(x,\bar{x},T)$ as defined in \eqref{eq:value_function_n_investor} is a $C^{2,2,1}((0,\infty)\times (0,\infty)\times[0,T])$-solution of \eqref{eq:investor_HJB} and meets the initial condition $v^i(x,\bar{x},0)=u_{\alpha^i}(x \bar{x}^{-\theta^i})$ for all $(x,\bar{x}) \in (0,\infty)^2$. Note that $v$  satisfies $\lim_{\varepsilon \downarrow 0} v^i(\varepsilon,\bar{x},T-t) = v^i(0+,\bar{x},T-t) = u_i(0,\bar{x})$ for all $ t\in [0,T]$, $\bar{x} \in (0,\infty)$, by our convention \eqref{eq:utility_function_limits}.

    Let us argue that, for $\varepsilon>0$ fixed, $V^{i,\varepsilon}_t:= v^i(X^i_t+\varepsilon,\bar{X}^{-i}_t,T-t)$, $t \in [0,T]$, is a supermartingale where $X^i$ solves the wealth dynamics \eqref{eq:wealth_dynamics} for some $\hat{\phi}^i \in \mathcal{A}^i$. For this, we show that  $V^{i,\varepsilon}$ is bounded from below, i.e. $V^{i,\varepsilon}\geq - C_{\varepsilon,\pi^{-i}}$ for some integrable $C_{\varepsilon,\pi^{-i}} > 0$. Indeed, for $\alpha^i \in (0,1)$, the utility is bounded from below by zero. For $\alpha^i >1$, however, we ensure a uniform bound by showing that
    \begin{equation}\label{eq:unif_bdd_mean_wealth}
        \mathbb{E}[\sup_{t \in [0,T]} |\bar{X}^{-i}_t|^p ]< \infty
    \end{equation}
    for all $p>0$. In fact, the mean wealth process $\bar{X}^{-i}$ is the product of independent stochastic exponentials and strictly positive and can be written in closed form as
    \begin{equation}
    \begin{split}
        \bar{X}^{-i}_t = \bar{x}_0 \exp\bigg( (\overline{\tau\pi})^{-i} t &+ \frac{1}{n} \sum_{j \neq i}  \sigma^j \pi^j(0) W^j_t + (\overline{\sigma^0\pi})^{-i} W^0_t \\
                &+  \int_{[0,t] \times E}  \log((\overline{1 + (\pi\circ \zeta) \eta})^{-i} (e) )  N(ds,de)  \bigg)
    \end{split}
    \end{equation}
    for $t \in [0,T]$. Let $p>0$ and let us argue for one component after another: First, note that 
        \begin{equation*}
            \sup_{t \in [0,T]}\exp\big( p (\overline{\tau\pi})^{-i} t \big)
        \end{equation*}
    does not involve any randomness and is trivially finite. Secondly, as the maximum Brownian motion has finite exponential moments, it is easy to see that
        \begin{align*}
            \mathbb{E}\bigg[\sup_{t \in [0,T]} \exp \bigg(p \cdot \frac{1}{n} \sum_{j \neq i}  \sigma^j \pi^j(0) W^j_t \bigg) \bigg] < \infty \quad \text{ and } \quad  \mathbb{E}\bigg[\sup_{t \in [0,T]} \exp \big( p \cdot (\overline{\sigma^0\pi})^{-i} W^0_t \big)\bigg] < \infty.
        \end{align*}
    Hence, we are left to observe that
        \begin{align*}
            \mathbb{E}\bigg[\sup_{t \in [0,T]} \exp\bigg\{ \int_{[0,t] \times E} &p \log((\overline{1 + (\pi\circ \zeta) \eta})^{-i} (e) )  N(ds,de)\bigg\}\bigg] \\
            &\leq \mathbb{E}\bigg[\sup_{t \in [0,T]} \int_{[0,t] \times E} ((\overline{1 + (\pi\circ \zeta) \eta})^{-i}) (e) )^p  N(ds,de)\bigg]  < \infty,
        \end{align*}
    but, since the integral against the Poisson random measure admits almost surely finitely many jumps on any finite time interval, we only need to check whether the jumps have finite expectation with respect to $\nu$ and this immediately follows from Lemma \ref{lemma:majorante}.
    
    Thus, \eqref{eq:unif_bdd_mean_wealth} holds and together with $V^{i,\varepsilon}$ having nonnegative drift, which follows from our arguments leading to the HJB equation, this implies that $V^{i,\varepsilon}$ is a supermartingale. In particular, this allows us to observe
    \begin{align*}
        \mathbb{E}[u_i(X^i_T,\bar{X}^{-i}_T) ] \leq \mathbb{E}[ V^{i,\varepsilon}_T ] \leq  \mathbb{E}[ V^{i,\varepsilon}_0 ] = v^i(x^i_0+\varepsilon,\bar{x}_0,T).
        \end{align*}
    By sending $\varepsilon \downarrow 0$, we even have
        \[ \mathbb{E}[u_i(X^i_T,\bar{X}^{-i}_T) ] \leq v^i(x^i_0,\bar{x}_0,T).
        \]
    So the function on the right-hand side of \eqref{eq:value_function_n_investor} yields an upper bound for the value of investor $i$'s best response problem \eqref{eq:investor_optimization_problem}.

    It remains to show that $V^{i,\star}_t := v^i(X^{i,\star}_t,\bar{X}^{-i}_t,T-t)$ is a martingale for the wealth process $X^{i,\star}:=X^{\phi^{i,\star}}$ given the strategy $\phi^{i,\star} = \phi^{i,\star}_{\pi^{-i}}$. Then, this strategy is optimal since we can conclude
        \begin{align*}
            v^i(x^i_0,\bar{x}_0,T) &= \mathbb{E}[u_i(X^{i,\star}_T,\bar{X}^{-i}_T) ]\leq \sup_\phi \mathbb{E}[u_i(X^{i}_T,\bar{X}^{-i}_T)] = v^i(x^i_0,\bar{x}_0,T).
        \end{align*}

    We note that $\phi^{i,\star}$ is signal-driven (and admissible) in the sense of Definition \ref{defn:signal-driven_strategy} and that the use of $=$ instead of $\in$ is justified since the $\argmax$ in both \eqref{eq:optimal_strategy_zero} and \eqref{eq:optimal_strategy_nonzero} are singletons. This follows immediately from the fact that the target functions are strictly concave in $\varphi$.

    Further, our arguments leading to the HJB equation \eqref{eq:investor_HJB}, show that $V^{i,\star}$ has zero drift:
        \begin{align*}
            dV^{i,\star}_t = V^{i,\star}_{t-}(1-\alpha^i)&\bigg[ \sigma^i \phi^{i,\star}(0) dW^i_t + \sigma^{0i} \phi^{i,\star}(0) dW^0_t + \frac{1}{n} \sum_{j \neq i}  \sigma^j \pi^j(0) dW^j_t +  (\overline{\sigma^0\pi})^{-i}dW^0_t \\
            & \quad + \int_E \bigg[ u_i\big(1+\phi^{i,\star}(\zeta^i(e))\eta^i(e), (\overline{1 + (\pi\circ \zeta) \eta})^{-i} (e)\big) - u_i(1,1) \bigg]\bar{N}(dt,de)  \bigg].
        \end{align*}
    In particular, $V^{i,\star}$ is the product of independent exponentials.
    The stochastic exponentials with respect to the Brownian motions are martingales, since $\sigma^i,\sigma^{0i},\sigma^j,\sigma^{0j}$, $(1-\alpha^i)$, $\pi^j(0)$ and $\phi^{i,\star}_\pi(0)$ are constants, and the exponential with respect to the compensated compound Poisson process is a martingale, because (cf. \cite[Theorem 33.2]{Sato})
        \begin{align*}
            &\int_E \big[ (1+\phi^{i,\star}(\zeta^i(e)) \eta^i(e))^{\frac{1-\alpha^i}{2}} ((\overline{1 + (\pi\circ \zeta) \eta})^{-i} (e))^{\frac{-\theta^i(1-\alpha^i)}{2n}} - 1 \big]^2 \nu(de)\\
            & \quad \quad \leq\int_E \big[ (1+\phi^{i,\star}(\zeta^i(e)) \eta^i(e))^{1-\alpha^i} ((\overline{1 + (\pi\circ \zeta) \eta})^{-i} (e))^{\frac{-\theta^i(1-\alpha^i)}{n}} + 1 \big] \nu(de)
        \end{align*}
    is ensured to be finite by $M_{\pi^{-i}}^i$ being finite. Thus, $V^{i,\star}$ is a martingale and the proof is complete.
\end{proof}

\subsection{Existence of a Nash Equilibrium}\label{sec:nash_equilibrium} Having found the best response for a single investor interacting with peers employing signal-driven strategies, we are now in a position to demonstrate the existence of an equilibrium.

\begin{theorem}\label{thm:nash_equilibrium}
    Under Assumption~\ref{assumption:admissible_type}, there exists a signal-driven Nash equilibrium $(\pi^{1,\star},....,\pi^{n+1,\star}) \in \mathcal{A}_{\text{sig}}^1\times ... \times \mathcal{A}_{\text{sig}}^{n+1}$.
\end{theorem}

The proof is based on Schauder's Fixed Point Theorem which requires a continuous self-mapping of a compact and convex domain. For our best response map above \eqref{eq:optimal_strategy_zero} and \eqref{eq:optimal_strategy_nonzero}, these requirements are difficult to meet in general, but our Assumption~\ref{itm:t3} makes this tractable: There are only finitely many distinct signals to account for and so the space of admissible signal-driven strategies $(\pi^{1},....,\pi^{n+1}) \in \mathcal{A}_{\text{sig}}^1\times ... \times \mathcal{A}_{\text{sig}}^{n+1}$ can be identified with a vector 
	\[ (\pi^1,...,\pi^{n+1}) \in \bigtimes_{i=1}^{n+1}\bigtimes_{z \in \zeta^i(E)\cup\{0\}} \bar{\Phi}^i(z) =: \bar{\Phi}.\]
The set $\bar{\Phi}$ is compact and convex by construction and thus serves as the domain and image set of the desired fixed point mapping. 

\begin{proof}
     A sufficient condition for a candidate $\pi = (\pi^1,...,\pi^{n+1}) \in \mathcal{A}_{\text{sig}}^1 \times ... \times \mathcal{A}_{\text{sig}}^{n+1}$ to be a signal-driven Nash equilibrium is to satisfy $\pi^i = \phi^{i,\star}_{\pi^{-i}}  \text{ for all } i\in \{1,...,n+1\}$, where $\phi^{i,\star}_{\pi^{-i}}$ is described in Theorem \ref{thm:verification_n}. 
     Thus, the equilibrium is characterized by a fixed point of
	\begin{align*}
		\gamma : \bar{\Phi} \to \bar{\Phi}, \quad \pi = (\pi^1,...,\pi^{n+1}) \mapsto \gamma(\pi) := (\gamma(\pi)^1,...,\gamma(\pi)^{n+1}) := ( \phi^{1,\star}_{\pi^{-1}},...,\phi^{n+1,\star}_{\pi^{-(n+1)}} ).
	\end{align*}
   Let us argue for continuity of $\gamma$. By Lemma \ref{lemma:majorante}, we can use dominated convergence to find that the target functions in both \eqref{eq:optimal_strategy_zero} and \eqref{eq:optimal_strategy_nonzero} are continuous in $(\varphi,\pi)$. Thus, by Berge's Maximum Theorem (cf. \cite[Theorem 17.31]{Aliprantis2006}), $\phi^{i,\star}_{\pi^{-i}}$ is upper-hemicontinuous in $\pi$. In fact, since by strict concavity of our utility functions the $\argmax$ operators in~\eqref{eq:optimal_strategy_zero} and~\eqref{eq:optimal_strategy_nonzero} always yield singletons, the maximizers are actually continuous functions. As a consequence, $\gamma$ is indeed continuous in $\pi$.

    Finally, since $\bar{\Phi}$ is non-empty, compact and convex and $\gamma: \bar{\Phi} \to \bar{\Phi}$ is continuous, we can use Schauder's Fixed Point Theorem, cf.\ \cite{Schauder1930}, to conclude that there exists $\pi^\star \in \bar{\Phi}$ with $\pi^\star=\gamma(\pi^\star)$.
\end{proof}

In summary, we have demonstrated the existence of a signal-driven Nash equilibrium, which also qualifies as a Nash equilibrium within the broader class of admissible (not necessarily signal-driven) strategies. Let us emphasize that we neither claim uniqueness of such equilibria nor can we rule out existence of Nash equilibria involving strategies which are not signal-driven.

\subsection{Common vs.\ Idiosyncratic Noise in Aggregate Wealth}\label{sec:limit}

An important question arises: how can this framework transition to a corresponding mean field game? This transition is particularly motivated by the fact that, in such a mean field setting a single individual’s actions do not impact the environment of others and, so, the signal-driven strategies that we are focussing on are an even more natural concept as the average wealth at the core of the interactions between agents cannot be influenced anymore by every single investor. One way to interpret the mean field setting is to consider it as the limit of the multi-agent game for the number of agents $n \to \infty$.

Before moving on to a rigorous formulation of the mean field game in Section \ref{sec:mean_field_game} below, we provide an appetizer discussion and deliver a result on the limit of the average wealth for future reference.\\

Let us assume that the types of our investors $j=1,2,\dots$ are obtained from a corresponding sequence $\mathtt{t}^1$, $\mathtt{t}^2$,... in $\mathtt{T}$ generated by independently and identically distributed samples from a probability $\mathbb{T} = \mathbb{T}(d\mathtt{t})$ which describes the distribution of types among the population of investors. We denote a generic sample from $\mathbb{T}$ by 
    \[ \mathtt{t} = (x_0^\mathtt{t},r^\mathtt{t},\kappa^\mathtt{t}, \sigma^\mathtt{t}, \sigma^{0\mathtt{t}},\eta^\mathtt{t},\zeta^\mathtt{t},\alpha^\mathtt{t},\theta^\mathtt{t}) \in \mathtt{T} \]
and with some abuse of notation consider
    \[ \mathtt{t}^j = (x_0^j,r^j,\kappa^j, \sigma^j, \sigma^{0j},\eta^j,\zeta^j,\alpha^j,\theta^j) = (x_0^{\mathtt{t}^j},r^{\mathtt{t}^j},\kappa^{\mathtt{t}^j}, \sigma^{\mathtt{t}^j}, \sigma^{0\mathtt{t}^j},\eta^{\mathtt{t}^j},\zeta^{\mathtt{t}^j},\alpha^{\mathtt{t}^j},\theta^{\mathtt{t}^j}) \in \mathtt{T} \quad \text{ for } j \in \mathbb{N}. \]
Each investor $j=1,2,\dots$ is equipped with her own independent idiosyncratic Brownian motion $W^j$; the Brownian motion $W^0$ specifies common Brownian noise shared by all agents.

The specification of common and idiosyncratic jump noise is less obvious. A particular challenge in this context is the specification of a framework which allows for idiosyncratic signals on (partially) common price shocks, a key feature of our model, which, to the best of our knowledge, has not been addressed in the literature so far. Indeed, while jump times of the Poisson random measure $N(dt,de)$ should be common to all investors, it is important to address the question how the marks set by $N$ represent common and  idiosyncratic noise. For instance, we want to allow for the possibility that each of our investors receives a signal on common impending price shocks which is affected by idiosyncratic noise. So, an investor may have a type $\mathtt{t} \in \mathtt{T}$ that affords her warnings about impending price shocks with some probability $p^\mathtt{t}$. While the price shock may be common to all investors and the warning probability $p^\mathtt{t}$ will be shared by her fellow investors of type $\mathtt{t}$, the actual heads-or-tails outcome of her idiosyncratic coin flip experiment for the shock warning may be independent from that of everyone else. The distinction between common and idiosyncratic noise becomes particularly crucial in the mean field framework. Here, idiosyncratic noise can be averaged out by a law of large numbers to yield a deterministic quantity while common noise persists as a random quantity affecting everyone. 

To allow for a rigorous formalization of such settings with idiosyncratic and common jump noise, let us consider the mark space as an infinite product space
    \[ E := E^{\mathtt{c}} \times E^{\otimes\mathtt{i}} \text{ for } E^{\otimes\mathtt{i}} := \bigtimes_{j \in \mathbb{N}} E^{\mathtt{i},j} \]
for Polish spaces $E^{\mathtt{c}}$ and identical $E^{\mathtt{i},1}=E^{\mathtt{i},2}=\dots=:E^{\mathtt{i}}$ with generic elements $$e=(e^{\mathtt{c}},e^{\mathtt{i}}) = (e^{\mathtt{c}},e^{\mathtt{i},1},e^{\mathtt{i},2},...) \in E.$$ The independent Poisson random measure $N=N(dt,de) =N(dt,de^{\mathtt{c}},de^{\mathtt{i}})$ with jump rate $\lambda \in (0,\infty)$ has intensity measure
        \[ \nu := \lambda (\nu^{\mathtt{c}} \otimes \nu^{\otimes\mathtt{i}}), \]
where $\nu^{\mathtt{c}}$ is a probability measure on $E^{\mathtt{c}}$ and
        \[ \nu^{\otimes\mathtt{i}}(de^{\mathtt{i}}) := \bigotimes_{j \in \mathbb{N}} \nu^{\mathtt{i}}(de^{\mathtt{i},j}) \]
generates independently for each investor identically distributed  idiosyncratic noise marks, each with law $\nu^{\mathtt{i}}$ on $E^{\mathtt{i}}$. In particular, this leads to jumps and signals with both common and idiosyncratic components for each investor $j=1,2,\dots$ via $\eta^{j}(e^{\mathtt{c}},e^{\mathtt{i}})= \eta^{\mathtt{t}^j}(e^{\mathtt{c}},e^{\mathtt{i}})=\eta^{\mathtt{t}^j}(e^{\mathtt{c}},e^{\mathtt{i},j})$ and $\zeta^{j}(e^{\mathtt{c}},e^{\mathtt{i}})=\zeta^{\mathtt{t}^j}(e^{\mathtt{c}},e^{\mathtt{i}})=\zeta^j(e^{\mathtt{c}},e^{\mathtt{i},j})$ where $\eta^j$ and $\zeta^j$ are the jump and signal maps on $E^{\mathtt{c}}\times E^{\mathtt{i}}$ for an investor of type $\mathtt{t}^j \in \mathtt{T}$.

Of course, the actual signals $Z_{t_k}^j=\zeta^{\mathtt{t}^j}(e^{\mathtt{c}}_k,e^{\mathtt{i},j}_k)$ that each individual investor $j=1,2,\dots$ may, at the various mark times $t_k$, receive are exposed to both the respective common mark $e^{\mathtt{c}}_k$ and the investor specific mark $e^{\mathtt{i},j}_k$.  As a result, even investors of the same type may make differing investment decisions and figuring out the limiting geometric average wealth when their number $n$ goes to infinity has to account for this:
\begin{lemma}\label{lemma:limit}
    Suppose each investor $j=1,2,\dots$ follows a signal-driven strategy $\pi^j=\pi^{\mathtt{t}^j} \in \mathcal{A}^j_{\text{sig}}$ determined solely by her respective type $\mathtt{t}^j$, determined as i.i.d.\ sample from the type distribution $\mathbb{T}(d\mathtt{t})$.
    Moreover, suppose the following integrability assumptions hold: 
    \begin{equation*}
        \begin{split}
            &\int_{\mathtt{T}} |\log (x_0^\mathtt{t} )| \mathbb{T}(d\mathtt{t}) < \infty, \\
            &\int_{\mathtt{T}}\big| r^\mathtt{t}+\pi^\mathtt{t}(0)(\kappa^\mathtt{t}-r^\mathtt{t}) - \frac{1}{2} ((\sigma^\mathtt{t})^2+(\sigma^{0\mathtt{t}})^2)(\pi^\mathtt{t}(0))^2 \big| \mathbb{T}(d\mathtt{t}) < \infty,\\
            &\int_{\mathtt{T}} |\sigma^{0 \mathtt{t}}\pi^\mathtt{t}(0)| \mathbb{T}(d\mathtt{t}) < \infty
        \end{split}
    \end{equation*}
    and for $e^\mathtt{c} \in E^\mathtt{c}$
    \begin{equation*}
		 \int_{\mathtt{T}\times E^{\mathtt{i}}} \log\big( 1+ \pi^{\mathtt{t}}(\zeta^{\mathtt{t}}(e^{\mathtt{c}},e^{\mathtt{i}})) \eta^{\mathtt{t}}(e^{\mathtt{c}},e^{\mathtt{i}}) \big) \nu^{\mathtt{i}}(de^{\mathtt{i}}) \otimes \mathbb{T}(d\mathtt{t})< \infty.
	\end{equation*}
    Then, as the number of investors $n$ converges to $\infty$, the average wealth $\bar{X}^{-i}_t$ of each investor's $i \in \mathbb{N}$ peers converges for all $t \in [0,T]$ almost surely to
        \begin{align*}
            \lim_{n \to \infty} \bar{X}^{-i}_t = \bar{x}_0 \exp &\bigg\{ \overline{\mathtt{t}\pi} \ t +\overline{\sigma^0 \pi} \  W^0_t  + \int_{[0,t] \times E^{\mathtt{c}}} \log(\overline{1+ (\pi \circ \zeta)\eta}(e^{\mathtt{c}})) N(ds,de^{\mathtt{c}},E^{\mathtt{i}})\bigg\}
        \end{align*}
    for
    \begin{equation}\label{eq:mean_components}
        \begin{split}
            \bar{x}_0 &:= \exp \int_{\mathtt{T}} \log (x_0^\mathtt{t} ) \mathbb{T}(d\mathtt{t}) \\
            \overline{\tau\pi} &:= \int_{\mathtt{T}}\left( r^\mathtt{t}+\pi^\mathtt{t}(0)(\kappa^\mathtt{t}-r^\mathtt{t}) - \frac{1}{2} ((\sigma^\mathtt{t})^2+(\sigma^{0\mathtt{t}})^2)(\pi^\mathtt{t}(0))^2 \right)\mathbb{T}(d\mathtt{t}) \\
            \overline{\sigma^0 \pi} &:= \int_{\mathtt{T}} \sigma^{0 \mathtt{t}}\pi^\mathtt{t}(0) \mathbb{T}(d\mathtt{t})
        \end{split}
        \end{equation}
    and
    \begin{equation}\label{eq:mean_jumps}
		\overline{1+ (\pi \circ \zeta)\eta}(e^{\mathtt{c}}) := \exp \left\{ \int_{\mathtt{T}\times E^{\mathtt{i}}} \log\big( 1+ \pi^{\mathtt{t}}(\zeta^{\mathtt{t}}(e^{\mathtt{c}},e^{\mathtt{i}})) \eta^{\mathtt{t}}(e^{\mathtt{c}},e^{\mathtt{i}}) \big) \nu^{\mathtt{i}}(de^{\mathtt{i}}) \otimes \mathbb{T}(d\mathtt{t}) \right\}, \quad  e^\mathtt{c} \in E^\mathtt{c}.
	\end{equation}
\end{lemma}

\begin{proof}
    For $n \in \mathbb{N}$, $i=1,\dots,n$, and $t \in [0,T]$ consider $\bar{X}^{-i}$ from \eqref{eq:averageLevyStructure} for the specified noise structure:  
    \begin{align*}
        \bar{X}^{-i}_t = \bigg( \prod_{j\neq i} x_0^j \bigg)^{\frac{1}{n}} \exp\bigg\{ t\frac{1}{n} \sum_{j \neq i} \big[r^j &+ \pi^j(0)(\kappa^j-r^j) - \frac{1}{2} ((\sigma^j)^2 + (\sigma^{0j})^2) (\pi^j(0))^2 \big] + \frac{1}{n} \sum_{j \neq i}  \sigma^j \pi^j(0) W^j_t \\
                 &+ \frac{1}{n} \sum_{j \neq i}\sigma^{0j} \pi^j(0) W^0_t + \sum_{t_k \leq t}  \frac{1}{n} \sum_{j \neq i} \log(1+ \pi^j(\zeta^j(e^{\mathtt{c}}_k, e^{\mathtt{i},j}_k))\eta^j(e^{\mathtt{c}}_k,e^{\mathtt{i},j}_k))  \bigg\}.
    \end{align*}
    By the law of large numbers, we immediately find
        \[ \lim_{n \to \infty} \bigg( \prod_{j\neq i} x_0^j \bigg)^{\frac{1}{n}} = \lim_{n \to \infty} \exp \bigg( \frac{1}{n} \sum_{j\neq i} \log(x^j_0) \bigg) = \exp \int_{\mathtt{T}} \log (x_0^\mathtt{t} ) \mathbb{T}(d\mathtt{t}).  \]
    Similarly, we find convergence of
    \begin{align*}
        \lim_{n \to \infty} \frac{1}{n} \sum_{j \neq i} \big[r^j & + \pi^j(0)(\kappa^j-r^j) - \frac{1}{2} ((\sigma^j)^2 + (\sigma^{0j})^2) (\pi^j(0))^2 \big]  \\
        &= \int_{\mathtt{T}} \left(r^\mathtt{t}+\pi^\mathtt{t}(0)(\kappa^\mathtt{t}-r^\mathtt{t}) - \frac{1}{2} ((\sigma^\mathtt{t})^2+(\sigma^{0\mathtt{t}})^2)(\pi^\mathtt{t}(0))^2\right) \mathbb{T}(d\mathtt{t}).
    \end{align*}
   For the term aggregating the independent idiosyncratic Brownian motions $W^j$, we find
    \begin{align*}
        \lim_{n \to \infty }\frac{1}{n} \sum_{j \neq i}  \sigma^j \pi^j(0) W^j_t =  0.
    \end{align*}
    For the term involving the common noise Brownian motion $W^0$, we get the average 
    \begin{align*}
        \lim_{n \to \infty } \frac{1}{n} \sum_{j \neq i}\sigma^{0j} \pi^j(0) W^0_t = W^0_t \int_{\mathtt{T}} \sigma^{0 \mathtt{t}}\pi^\mathtt{t}(0) \mathbb{T}(d\mathtt{t}).
    \end{align*}
    Lastly, for the jump term, we note that all investors have in common the mark times $t_k$ and the noise marks $e^{\mathtt{c}}_k$, while the $e^{\mathtt{i},j}_k$ are i.i.d.\ for $k,j \in \mathbb{N}$ with distribution $\nu^\mathtt{i}$. The law of large numbers thus yields the almost sure convergence 
    \begin{align*}
    	\lim_{n \to \infty} \sum_{t_k \leq t}  \frac{1}{n} \sum_{j \neq i} &\log(1+ \pi^j(\zeta^j(e^{\mathtt{c}}_k,e^{\mathtt{i},j}_k))\eta^j(e^{\mathtt{c}}_k,e^{\mathtt{i},j}_k)) \\
    	&= \sum_{t_k \leq t}  \int_{E^{\mathtt{i}}\times \mathtt{T}} \log\big( 1+ \pi^{\mathtt{t}}(\zeta^{\mathtt{t}}(e^{\mathtt{c}}_k,e^{\mathtt{i}})) \eta^{\mathtt{t}}(e^{\mathtt{c}}_k,e^{\mathtt{i}}) \big) \nu^{\mathtt{i}}(de^{\mathtt{i}}) \otimes \mathbb{T}(d\mathtt{t}) \\
    	&= \int_{[0,t] \times E^{\mathtt{c}}} \bigg( \int_{E^{\mathtt{i}}\times \mathtt{T}} \log\big( 1+ \pi^{\mathtt{t}}(\zeta^{\mathtt{t}}(e^{\mathtt{c}},e^{\mathtt{i}})) \eta^{\mathtt{t}}(e^{\mathtt{c}},e^{\mathtt{i}}) \big) \nu^{\mathtt{i}}(de^{\mathtt{i}}) \otimes \mathbb{T}(d\mathtt{t}) \bigg)N(ds,de^{\mathtt{c}},E^{\mathtt{i}}).
    \end{align*}
    Putting everything together, we find the desired result.
\end{proof}

It is an intriguing question whether, when adding more and more investors in an i.i.d.\ fashion, the corresponding Nash equilibria also converge to a mean field equilibrium, whose definition we provide and existence we establish in the following Section \ref{sec:mean_field_game}. Moreover, one could ask whether the average wealth of investors employing equilibrium strategies in a multi-agent game converges to the average wealth in the mean field game under its equilibrium strategy. We have to leave these questions for future research. The main reason is that we cannot even prove uniqueness of equilibria; see Remark \ref{remark:uniqueness} where we discuss the pertaining difficulties.

\section{The Mean Field Game}\label{sec:mean_field_game}

In this section we give a mean field formulation for interactive utility maximization with jump signals. Section~\ref{sec:representative_investor} provides a model for the representative investor which leads to a canonical notion of mean field equilibrium in signal-driven strategies. Section~\ref{sec:mean_field_HJB} provides the best response function which maps a proposed signal-driven strategy used in the mean field environment to the representative investor's best response. Section~\ref{sec:mean_field_existence} then proves existence of a mean field equilibrium, preparing the ground for the numerical case studies carried out in the next Section~\ref{sec:numerics}.

\subsection{The Representative Investor}\label{sec:representative_investor}

Guided by the considerations of the average wealth in Section~\ref{sec:limit}, let us assume that our representative investor's type is determined by a probability $\mathbb{T}=\mathbb{T}(d\mathtt{t})$ on the 
\textit{type space}
	\[ \mathtt{T} = (0,\infty) \times \mathbb{R}^2 \times [0,\infty) \times [0,\infty) \times L^0(E) \times L^0(E) \times ((0,\infty)\setminus\{1\}) \times [0,1] \]
with its Borel $\sigma$-field $\mathcal{T} := \mathcal{B}(\mathtt{T})$. As this type choice should be independent from the other sources of randomness in our system, we pass from our original probability space $(\Omega,\mathfrak{A},\mathbb{P})$ supporting independent Brownian motions $W$, $W^0$ along with an independent marked Poisson point process $N$ to
\[ \bar{\Omega} := \mathtt{T} \times \Omega,\; \bar{\mathfrak{A}} := \mathcal{T} \otimes \mathfrak{A} \text{ and } \bar{\mathbb{P}} := \mathbb{T} \otimes \mathbb{P}. \]
The processes $W$, $W^0$, $N$ are lifted to $\bar{\Omega}$ via the canonical projection $\bar{\Omega} \to \Omega$.

Similarly, we consider the projection $\tau(\bar{\omega}) = \tau(\mathtt{t},\omega) = \mathtt{t} \in \mathtt{T}$ defined for a generic $\bar{\omega} = (\mathtt{t},\omega) \in \mathtt{T} \times \Omega = \bar{\Omega}$. The components of this projection $\tau$ implicitly define the population's characteristics
	\[ \tau = (x_0,r,\kappa, \sigma, \sigma^{0},\eta,\zeta,\alpha,\theta)\]
	as measurable mappings on $\bar{\Omega}$ . To emphasize their dependence on the type, we sometimes will denote $x_0 = x_0^\tau $, i.e.\  $x_0^{\tau}(\bar{\omega}) = x_0^\mathtt{t}$, where the latter is the initial wealth component stored in the type vector $\mathtt{t} \in \mathtt{T}$; we proceed similarly for $r=r^\tau$, $\kappa = \kappa^\tau$, etc. 

Idiosyncratic and common jump noise can be introduced in analogy to our setting in Section~\ref{sec:limit} except that we only need idiosyncratic and common noise for one representative investor. We thus consider a mark space in product form    \[ E := E^{\mathtt{c}} \times E^{\mathtt{i}}\] with two Polish spaces $E^{\mathtt{c}}$, $E^{\mathtt{i}}$, equipped with their  respective Borel-$\sigma$-fields $\mathcal{E}^{\mathtt{c}} = \mathcal{B}(E^{\mathtt{c}})$ and $\mathcal{E}^i = \mathcal{B}(E^{\mathtt{i}})$. We write    \[ \nu:= \lambda (\nu^{\mathtt{c}} \otimes \nu^{\mathtt{i}}),\] as the product of two probability measures $\nu^{\mathtt{c}}$ and $\nu^{\mathtt{i}}$ scaled with the jump rate $\lambda \in (0,\infty)$. Elements of $E$ will be generically denoted by $e = (e^{\mathtt{c}},e^{\mathtt{i}}) \in E^{\mathtt{c}} \times E^{\mathtt{i}} = E$ and noise such as price shocks that commonly affect all investors will only depend on $e$ via its $e^{\mathtt{c}}$ component; idiosyncratic noise resulting from a mark $e$ such as shock warnings will also depend on $e^{\mathtt{i}}$, a source of randomness that we will use to allow for individual noise in signals that should be thought of as independent between agents. 

Let us also lift the information flows generated by the drivers. That is, denoting by $\mathcal{N}$ the set of $\bar{\mathbb{P}}$-negligible events, we have 
	\[ \bar{\mathcal{F}}^{W,W^0,N}_t := \sigma((W_s,W^0_s, N([0,s]\times A \times B)): s \in [0,t], A \in \mathcal{E}^{\mathtt{c}}, B \in \mathcal{E}^i)\vee \mathcal{N}.\] 
Supplemented by the information generated by the type $\tau$ this gives us the full information filtration
    \begin{align*}
        \bar{\mathcal{F}}_t := \bar{\mathcal{F}}^{W,W^0,N}_t \vee \sigma(\tau), \quad t \in [0,T].
    \end{align*}
Also, denote by $\mathcal{F}^{\text{0}}=(\mathcal{F}^0_t)_{t \in [0,T]}$ the common noise filtration
\begin{align*}
        \mathcal{F}^{0}_t := \sigma((W^0_s, N([0,s]\times A \times E^{\mathtt{i}})): s \in [0,t], A \in \mathcal{E}^{\mathtt{c}})\vee \mathcal{N}.
    \end{align*} 
An immediate, albeit fundamental consequence is that any  $\bar{\mathcal{F}}_0$-measurable random variable is actually $\sigma(\tau)$-measurable and independent of $N$, $W^0$ and $W$ under $\bar{\mathbb{P}}$. Additionally, 
	\[ \eta^{\tau}(\bar{\omega})(e) = \eta^{\mathtt{t}}(e) \text{ and } \zeta^{\tau}(\bar{\omega})(e) =  \zeta^{\mathtt{t}}(e) \]
are jointly measurable in their arguments.

As in the multi-agent game, investors will obtain information from their signal process. In the present mean field setting, we define the representative investor's signal process as
    \begin{equation*}
        Z_t = \int_{\{t\}\times E} \zeta(e) N(ds,de) \quad \text{ for } t \in [0,T].
    \end{equation*}
Observe that this signal process \emph{is} type-dependent as $$Z(\bar{\omega}) = Z^{\tau(\bar{\omega})}(\bar{\omega}) = Z^{\mathtt{t}}(\omega):=\left(\int_{\{t\}\times E} \zeta^{\mathtt{t}}(e) N(\omega)(ds,de)\right)_{t \in [0,T]} \text{ for } \bar{\omega} = (\mathtt{t},\omega) \in \mathtt{T} \times \Omega = \bar{\Omega}.$$

In the following, we rely on the type-dependent disintegration of the jump intensity measure: For all $\mathtt{t} \in \mathtt{T}$, we have
	\begin{equation}\label{eq:disintegration_mf}
        \nu(de \cap \{ \zeta^{\mathtt{t}} \neq 0\}) =\int_{\mathbb{R}\setminus \{0\}} K^{\mathtt{t}}(z,de) \mu^{\mathtt{t}}(dz), \quad \mu^{\mathtt{t}} :=\nu \circ (\zeta^{\mathtt{t}})^{-1},
    \end{equation}
where $\mathtt{t} \mapsto \mu^{\mathtt{t}}(\cdot)$ and $\mathtt{t} \mapsto K^{\mathtt{t}}(z,\cdot)$ are $\mathcal{T}$-measurable. This allows us to introduce
    \begin{align*}
 		&\underline{\eta}^{\mathtt{t}}(z):=K^{\mathtt{t}}(z,\cdot)\text{-ess inf } \eta^{\mathtt{t}},\\
        &\bar{\eta}^{\mathtt{t}}(z):=K^{\mathtt{t}}(z,\cdot)\text{-ess sup } \eta^{\mathtt{t}},\\
        &\underline{\eta}^{\mathtt{t}}(0):= \nu(\cdot \cap\{\zeta^{\mathtt{t}} =0, \eta^{\mathtt{t}} \neq 0\}) \text{-ess inf } \eta^{\mathtt{t}} \text{ and }\\
        &\bar{\eta}^{\mathtt{t}}(0):= \nu(\cdot \cap\{\zeta^{\mathtt{t}} =0, \eta^{\mathtt{t}} \neq 0\}) \text{-ess sup } \eta^{\mathtt{t}}
    \end{align*}
for all $\mathtt{t} \in \mathtt{T}$ and $z \in \mathbb{R}\setminus \{0\}$.

By the same reasoning as in the multi-agent game, we obtain both the no-arbitrage condition 
    \begin{equation}\label{eq:noarbitrage_mf}
        \underline{\eta}^{\mathtt{t}}(z) < 0 < \bar{\eta}^{\mathtt{t}}(z) \text{ for }  z \in \zeta^{\mathtt{t}}(E) \setminus \{0\},
    \end{equation}
and the admissible positions ensuring non-negative wealth: In case $\nu(\zeta^{\mathtt{t}} = 0, \eta^{\mathtt{t}} \neq 0) >0$ when unsignaled jumps are possible, the default positions without signal can be chosen from
    \begin{equation*}
        \Phi^{\mathtt{t}}(0) := 
        \begin{cases} [-1/ \bar{\eta}^{\mathtt{t}}(0), -1 / \underline{\eta}^{\mathtt{t}}(0) ], \quad & \text{ if } \underline{\eta}^{\mathtt{t}}(0) < 0 < \bar{\eta}^{\mathtt{t}}(0),\\
        [-1/ \bar{\eta}^{\mathtt{t}}(0), +\infty), \quad & \text{ if } 0 \leq \underline{\eta}^{\mathtt{t}}(0) \leq \bar{\eta}^{\mathtt{t}}(0), \quad \\
        (-\infty, -1 / \underline{\eta}^{\mathtt{t}}(0) ], \quad & \text{ if } \underline{\eta}^{\mathtt{t}}(0) \leq  \bar{\eta}^{\mathtt{t}}(0) \leq 0;
        \end{cases}
    \end{equation*}
if $\nu(\zeta^{\mathtt{t}} = 0, \eta^{\mathtt{t}} \neq 0) =0$, then all jumps are signaled and so $\Phi^{\mathtt{t}}(0) := \mathbb{R}$. For a non-zero signal we have
     \begin{align*}
         \Phi^{\mathtt{t}}(z) := [-1/ \bar{\eta}^{\mathtt{t}}(z), -1 / \underline{\eta}^{\mathtt{t}}(z)] \text{ for } z \in \mathbb{R}\setminus\{0\}.
     \end{align*}

We make the following standing assumption on the distribution of investor types.

\begin{assumption}\label{assumption:T} The type distribution $\mathbb{T}$ is such that
	 \begin{enumerate}
    	\item[$(\mathtt{T}_1)$]\namedlabel{itm:T1}{$(\mathtt{T}_1)$} the stock prices are nonnegative, i.e. $\nu(\{\eta^{\mathtt{t}} < -1 \}) = 0$ for $\mathbb{T}$-almost every $\mathtt{t} \in \mathtt{T}$, 
    	\item[$(\mathtt{T}_2)$]\namedlabel{itm:T2}{$(\mathtt{T}_2)$} the no-arbitrage condition, i.e. $\underline{\eta}^{\mathtt{t}}(z) < 0 < \bar{\eta}^{\mathtt{t}}(z) \text{ for $\mu^{\mathtt{t}}$-a.e. }  z \in \zeta^{\mathtt{t}}(E) \setminus \{0\}$, holds for $\mathbb{T}$-almost every $\mathtt{t} \in \mathtt{T}$,
    	\item[$(\mathtt{T}_3)$]\namedlabel{itm:T3}{$(\mathtt{T}_3)$} each investor type can receive only finitely many different signals, i.e. $|\zeta^{\mathtt{t}}(E)|<\infty$ for $\mathbb{T}$-almost every $\mathtt{t}\in\mathtt{T}$, and
    	\item[$(\mathtt{T}_4)$]\namedlabel{itm:T4}{$(\mathtt{T}_4)$} the market parameters and the risk aversion are bounded, i.e.             \[ \mathbb{T}(d\mathtt{t})\text{-}\esssup (|r^\mathtt{t}| +|\kappa^\mathtt{t}|+\sigma^\mathtt{t}+\sigma^{0\mathtt{t}}+\alpha^\mathtt{t}) < \infty\]
        and the initial wealth allows for a finite geometric average
        \[
         \log x^{\mathtt{t}}_0 \in L^1(\mathbb{T}(d\mathtt{t})).
        \]
    	\item[($\mathtt{T}_{5}$)]\namedlabel{itm:T5}{$(\mathtt{T}_{5})$} jump bounds are uniformly bounded away from zero: There exists $c_0>0$ such that for $\mathbb{T}$-almost every $\mathtt{t} \in \mathtt{T}$ 
			\begin{align*}
				 \bar{\eta}^{\mathtt{t}}(z) \geq c_0 > 0   \text{ and }  \underline{\eta}^{\mathtt{t}}(z) \leq -c_0 < 0 \quad \text{ for all } z \in \mathbb{R}
			\end{align*}
    \item[$(\mathtt{T}_{6})$]\namedlabel{itm:T6}{$(\mathtt{T}_6)$}   if only finitely many distinct types can occur (i.e.\ if $|\supp \mathbb{T}|<\infty$), jump sizes have finite moments in the sense that 
            \[ \int_{E} (1+\eta^{\mathtt{t}}(e))^p \nu(de) < \infty \text{ for all } p > 0, \ \mathtt{t} \in \supp \mathbb{T}; \]
    otherwise, we need uniform boundedness in the sense that
             \begin{align*}
				\mathbb{T}(d\mathtt{t})\otimes \nu(de)\text{-}\esssup \eta^\mathtt{t}(e) < \infty.
			\end{align*}
    \end{enumerate}
\end{assumption}

Let us review these assumptions. First of all, \ref{itm:T1}-\ref{itm:T3} are consistent with our assumptions in the multi-agent game, see Assumption \ref{assumption:admissible_type}. Further, \ref{itm:T4} will give us bounds on the geometric mean field wealth by ensuring that market parameters do not differ too much between different types of investors. It will hold immediately if $\mathbb{T}$ distinguishes only finitely many types.  Similarly, \ref{itm:T5} will ensure `nice' properties of the mean field jumps, as it yields uniform bounds for admissible positions of  all investor types, see \ref{itm:(i)} in Lemma~\ref{lemma:mf_tools} below. This will be particularly important in the construction of a suitable fixed point mapping in our proof for existence of mean field equilibria, see  Lemma~\ref{lemma:K}. Lastly, \ref{itm:T6} ensures sufficient integrability of stock price jumps, a condition also needed in the multi-agent game.\\

Let us next define the set of admissible strategies for the representative investor. For this, we will again consider compact intervals $\bar{\Phi}^{\mathtt{t}}(z)$ included in $\{0\} \cup \Interior\Phi^{\mathtt{t}}(z)$ for all $z \in \mathbb{R}$, $\mathtt{t} \in \mathtt{T}$. However, to ensure measurability and uniformly nonnegative wealth, we specify the intervals to be of the following form: If $\mathbb{T}$ is finitely supported, we let
	\[  \bar{\Phi}^{\mathtt{t}}(z) :=
	 \left[ - \frac{1-a^{\mathtt{t}}}{\overline{\eta}^{\mathtt{t}}(z)}, - \frac{1-b^{\mathtt{t}}}{\underline{\eta}^{\mathtt{t}}(z)} \right] \text{ for $z \in \mathbb{R}$, $\mathtt{t} \in \supp \mathbb{T}$,} \]
for {$\mathcal{T}$-measurable} $\mathtt{t} \mapsto a^{\mathtt{t}}$, $\mathtt{t} \mapsto b^{\mathtt{t}} \in (0,1]$. In case of infinitely many possible types, we consider
	\begin{align}\label{eq:positionboundsTinfty}
    \bar{\Phi}^{\mathtt{t}}(z) := \left[ - \frac{1-a}{\overline{\eta}^{\mathtt{t}}(z)} , - \frac{1-b}{\underline{\eta}^{\mathtt{t}}(z)}  \right] 
    \end{align}
	for all $\mathtt{t} \in \mathtt{T}$ and for type-independent $a,b \in (0,1]$.

To prepare the application of the Measurable Maximum Theorem (cf. \cite[Theorem 18.19]{Aliprantis2006}) let us record the following lemma:
\begin{lemma}\label{lemma:weak_mb_Phi}
     The map $\mathtt{t} \mapsto \bar{\Phi}^{\mathtt{t}}(z)$ is weakly $\mathcal{T}$-measurable in the sense of \cite[Def. 18.1]{Aliprantis2006}, that is, for all open subsets $B$ of $\mathbb{R}$, we have
			\[ \{ \mathtt{t} \in \mathtt{T} : \bar{\Phi}^{\mathtt{t}}(z) \cap B \neq \emptyset \} \in \mathcal{T}  \text{ for any } z \in \mathbb{R}. \]
\end{lemma}

\begin{proof} 
It suffices to show that for all $x,y \in \mathbb{R}$ we have
	\[ \{ \mathtt{t} \in \mathtt{T} : \bar{\Phi}^{\mathtt{t}}(z) \cap (x,y)   \neq \emptyset \} \in \mathcal{T} \]
for $z \in \mathbb{R}$. Thus, we get the desired measurability for both maps as soon as we have shown measurability of the interval bounds, that is, we need to show that each of the maps
	\begin{align*}
 		&\mathtt{t} \mapsto \underline{\eta}^{\mathtt{t}}(0)= \nu(\cdot \cap\{\zeta^{\mathtt{t}} =0, \eta^{\mathtt{t}} \neq 0\}) \text{-ess inf } \eta^{\mathtt{t}}, \\
        &\mathtt{t} \mapsto \overline{\eta}^{\mathtt{t}}(0)= \nu(\cdot \cap\{\zeta^{\mathtt{t}} =0, \eta^{\mathtt{t}} \neq 0\}) \text{-ess sup } \eta^{\mathtt{t}},\\
        &\mathtt{t} \mapsto \underline{\eta}^{\mathtt{t}}(z)=K^{\mathtt{t}}(z,\cdot)\text{-ess inf } \eta^{\mathtt{t}}, \\
        &\mathtt{t} \mapsto \overline{\eta}^{\mathtt{t}}(z)=K^{\mathtt{t}}(z,\cdot)\text{-ess sup } \eta^{\mathtt{t}}
    \end{align*}
is $\mathcal{T}$-measurable.
 
The argument for the other maps being similar, let us show that $\mathtt{t} \mapsto \overline{\eta}^{\mathtt{t}}(0)$ is $\mathcal{T}$-measurable. Noting that by Fubini-Tonelli
        \[ \mathtt{t} \mapsto \| \eta_+^{\mathtt{t}}\|_{L^p(\nu(\cdot \cap\{\zeta^{\mathtt{t}} =0, \eta^{\mathtt{t}} \neq 0\}))} = \int_E \mathds{1}_{\{\zeta^{\mathtt{t}}(e) =0, \eta^{\mathtt{t}}(e) \neq 0\}}|\eta_+^{\mathtt{t}}(e)|^p \nu(de) \footnote{For any function $f: \bar{\Omega} \times E \to \mathbb{R}$ we define $f_+(\mathtt{t},e) := \max \{f(\mathtt{t},e),0\}$.}  \]
is $\mathcal{T}$-measurable since the integrand is jointly $\mathcal{T} \otimes \mathcal{E}$-measurable and positive, we can write $\overline{\eta}^{\mathtt{t}}(0)$ as the pointwise limit of $\mathcal{T}$-measurable functions:
		\[ \overline{\eta}^{\mathtt{t}}(0)=\nu(\cdot \cap\{\zeta^{\mathtt{t}} =0, \eta^{\mathtt{t}} \neq 0\})\text{-}\esssup \eta^{\mathtt{t}}(\cdot) = \lim_{p\to \infty} \| \eta_+^{\mathtt{t}} \|_{L^p(\nu(\cdot \cap\{\zeta^{\mathtt{t}} =0, \eta^{\mathtt{t}} \neq 0\}))}. \]
Thus, $\mathtt{t} \mapsto \overline{\eta}^{\mathtt{t}}(0)$ is $\mathcal{T}$-measurable. 
\end{proof}

As before, admissible strategies are signal-dependent investment choices of the form
	\begin{equation}\label{eq:meyer_mb_strategy}
		\hat{\phi}_t(\bar{\omega},Z_t(\bar{\omega})) = \hat{\phi}^{\mathtt{t}}_t(\omega,Z^{\mathtt{t}}_t(\omega)), \quad (\bar{\omega},t)=(\mathtt{t},\omega,t)  \in \bar{\Omega} \times [0,T]
	\end{equation}
for some $\mathcal{P}(\bar{\mathcal{F}})\otimes \mathcal{B}(\mathbb{R})$-measurable field $\hat{\phi}:[0,T]\times \bar{\Omega} \times \mathbb{R}\to \mathbb{R}$. 
In the light of Meyer-$\sigma$-fields, we consider strategies measurable with respect to $\Lambda := \mathcal{P}(\bar{\mathcal{F}}) \vee \sigma(Z).$
The set of admissible strategies is        
	\begin{equation*}
		\bar{\mathcal{A}} := \{ \hat{\phi} \text{ is of form \eqref{eq:meyer_mb_strategy} and }  \hat{\phi}^{\mathtt{t}}_t(z) \in \bar{\Phi}^{\mathtt{t}}(z) \text{ for  all } z \in \mathbb{R},\ \mathtt{t} \in \mathtt{T},\ t \in [0,T] \}.
    \end{equation*}
Although $\bar{\mathcal{A}}$ is very similar to the definition in the multi-agent game, here we insist on $\bar{\mathcal{A}}$ being a set of positions for the representative investor, not for an investor of a specific type.

\begin{proposition}\label{prop:admissible_strategy_mf}
    Given an admissible strategy $\hat{\phi}_t \in \bar{\mathcal{A}}$, the representative investor's wealth process $X^{\hat{\phi}}$ has dynamics
    \begin{equation}\label{eq:mean_field_wealth}
    \begin{split}
        X^{\hat{\phi}}_0 &= x_0^\tau,\\
        \frac{dX^{\hat{\phi}}_t}{X^{\hat{\phi}}_{t-}} &= (r^\tau+\hat{\phi}_t(0)(\kappa^{\tau}-r^\tau))dt +  \sigma^{\tau} \hat{\phi}_t(0) dW_t + \sigma^{0\tau} \hat{\phi}_t(0) dW^0_t + \int_E \eta^{\tau}(e)\hat{\phi}_t(\zeta^{\tau}(e)) N(dt,de), 
    \end{split}
    \end{equation}
    for $t\in [0,T]$ and it remains uniformly bounded away from zero on $[0,T]$ $\bar{\mathbb{P}}$-almost surely.
\end{proposition}

The proof of this uniform boundedness is an immediate consequence of the following result:
\begin{lemma}\label{lemma:mf_lower_bound}
	Independently from the strategy applied, the returns from jumps are uniformly bounded from below: There exists a constant $c > 0$ such that for all $\hat{\phi} \in \bar{\mathcal{A}}$:
		\[ 0 < c \leq (\mathbb{T}(d\mathtt{t})\otimes \nu(de))\text{-}\essinf (1+ \hat{\phi}^{\mathtt{t}}(\zeta^{\mathtt{t}}(e))\eta^{\mathtt{t}}(e)). \]
\end{lemma}

\begin{proof}
	If $|\supp \mathbb{T}| < \infty$, we easily argue  as in Lemma \ref{lemma:unif_bdd}. If $|\supp \mathbb{T}| = \infty$, we have by our choice of admissible positions~\eqref{eq:positionboundsTinfty} in that case that 
		\begin{align*}
			1 + \hat{\phi}_t^{\mathtt{t}}(\zeta^{\mathtt{t}}(e))\eta^{\mathtt{t}}(e) \geq 1+ \min \left\{ - \frac{1-a}{\underline{\eta}^{\mathtt{t}}(\zeta^{\mathtt{t}}(e))} \cdot \underline{\eta}^{\mathtt{t}}(\zeta^{\mathtt{t}}(e)) , \  - \frac{1-b}{\overline{\eta}^{\mathtt{t}}(\zeta^{\mathtt{t}}(e)))} \cdot \bar{\eta}^{\mathtt{t}}(\zeta^{\mathtt{t}}(e)) \right\} =  \min \{a,b\} > 0
        \end{align*}
        for all $e \in E$ and all $\mathtt{t} \in \mathtt{T}$.
\end{proof}

As in the multi-agent game, we will prove existence of equilibria for the special class of signal-driven strategies:

\begin{definition}\label{defn:signal-driven_strategy_mf}
    An admissible strategy $\hat{\phi} \in \bar{\mathcal{A}}$ is called \textit{(purely) signal-driven} if $\hat{\phi}$ in \eqref{eq:meyer_mb_strategy} does not depend on $t$ and $\omega$, that is, if it can be written in the form	
	\begin{equation}\label{eq:strategyform_mf}
	   \hat{\phi}_t(\bar{\omega},z) = \phi^{\mathtt{t}}(z), \quad (\bar{\omega},t) \in \bar{\Omega} \times [0,T]
	\end{equation}
    for a $\mathcal{T}\otimes\mathcal{B}(\mathbb{R})$-measurable time-independent function  $\phi: \mathtt{T} \times \mathbb{R} \to \mathbb{R}, \ (\mathtt{t},z) \mapsto \phi^{\mathtt{t}}(z)$. We denote by
        \[ \bar{\mathcal{A}}_{\text{sig}} := \{ \hat{\phi} \in \bar{\mathcal{A}} \text{ of form } \eqref{eq:strategyform_mf} \}  \]
    the set of admissible signal-driven strategies.
\end{definition}
We will limit our search for (yet to be defined) mean field equilibria to signal-driven ones, following the same reasoning as in the multi-agent case. Also, we will, in preparation for our proof of existence of an equilibrium, solve a best response problem for the representative investor given a fixed environment. This environment is again described through the geometric mean wealth as induced by a signal-driven strategy $\pi=(\pi^{\mathtt{t}}) \in \bar{\mathcal{A}}_{\text{sig}}$ which specifies each type's $\mathtt{t}$ signal-driven strategy $\pi^{\mathtt{t}}$. Indeed, denoting by $\bar{\mathbb{E}}[\cdot]$ the expected value over the product space $\bar{\Omega}$ with respect to $\bar{\mathbb{P}}$, we define the average wealth in the mean field setting as
	\[ \bar{X}^{\pi}_t := \exp \bar{\mathbb{E}}[ \log X^{\pi}_t |\mathcal{F}^0_T]. \]
In light of our considerations for the limiting average wealth with infinitely many agents in Section~\ref{sec:limit}, our next result shows that this is consistent with our initial derivation of average wealth in Lemma \ref{lemma:limit}:

\begin{lemma}\label{lemma:mean_wealth_dynamics}
	For any $\pi \in \bar{\mathcal{A}}_\text{sig}$ and respective wealth process $X^{\pi}$ from Proposition \ref{prop:admissible_strategy_mf}, we find that $\bar{X}^\pi_t = \exp \bar{\mathbb{E}}[ \log X^{\pi}_t |\mathcal{F}^0_T]$ is given by
	\begin{equation}\label{eq:closed_from_y}
        \begin{split}
            \bar{X}^\pi_t = \bar{x}_0 \exp \bigg\{ \overline{\tau\pi} t + \overline{\sigma^0 \pi} W^0_t + \int_{[0,t] \times E^{\mathtt{c}}} \log(\overline{1+ (\pi \circ \zeta)\eta}(e^{\mathtt{c}}))  N(ds,de^{\mathtt{c}},E^{\mathtt{i}}) \bigg\}
        \end{split}
    \end{equation}
    for $\bar{x}_0$, $\overline{\tau\pi}$, $\overline{\sigma^0\pi}$ as defined in \eqref{eq:mean_components} and $\overline{1+ (\pi \circ \zeta)\eta}(e^{\mathtt{c}})$ as defined in \eqref{eq:mean_jumps}. In particular, the dynamics coincide with the limit from Lemma \ref{lemma:limit}.
\end{lemma}
\begin{proof}
	Let $\pi = \pi^\tau \in \bar{\mathcal{A}}_\text{sig}$. We note that
	 \[ \bar{X}^\pi_t = \exp \bar{\mathbb{E}}[\log X^\pi_t | \mathcal{F}^0_T]  = \exp \bar{\mathbb{E}}[\log X^\pi_t | \mathcal{F}^0_t], \]
where the second equality holds because $X^\pi_t$ is independent of $(N([t,s),de))_{s \in [t,T]}$ and $(W^0_s-W^0_t)_{s \in [t,T]}$. By It\^{o}'s formula, we easily find that
\begin{align*}
        \log X^{\pi}_t = \log(x_0^\tau) &+  \int_0^t \big[ (r+\pi^\tau(0)(\kappa-r)) - \tfrac{1}{2} (\sigma^2+(\sigma^0)^2)(\pi^\tau(0))^2\big]  ds +  \int_0^t\sigma \pi^\tau(0)dW_s   \\
            &+  \int_0^t \sigma^0 \pi^\tau(0)dW^0_s + \int_{[0,t] \times E} \log\big( 1+ \pi^\tau(\zeta^\tau(e)) \eta^\tau(e) \big) N(ds,de),
    \end{align*}
where, since $\pi^\tau$ is signal-driven, the strategies do not depend on $t$ and $\omega$ but on the signal. In particular, $\pi^\tau(z)$ is $\sigma(\tau)$-measurable for $z \in \mathbb{R}$. This is important, as in the following, we take conditional expectation with respect to the common noise $\mathcal{F}^0_t$ and then apply the exponential:

First, for the initial value, we find that
	\[ \bar{x}_0 := \bar{X}^\pi_0 = \exp \bar{\mathbb{E}}[\log x_0^\tau | \mathcal{F}^0_t] = \exp \int_{\mathtt{T}} \log (x_0^\mathtt{t} ) \mathbb{T}(d\mathtt{t}) ,\]
because $x_0^{\tau}$, being $\sigma(\tau)$-measurable, is independent from the common noise. Similarly, for the drift term, we exploit that the lifted common noise drivers $W^0$ and $N$ actually only depend on $\omega$, while the market parameters and $\pi^\tau(0)$ only depend on the type $\mathtt{t}$, thus, the $ds$-integrand averages to
    \begin{align*}
        \bar{\mathbb{E}}&\left[ r+\pi^\tau(0)(\kappa-r) - \frac{1}{2} (\sigma^2+(\sigma^0)^2)(\pi^\tau(0))^2 \big| \mathcal{F}^0_t \right] \\
        &\quad = \int_{\mathtt{T}} r^\mathtt{t}+\pi^\mathtt{t}(0)(\kappa^\mathtt{t}-r^\mathtt{t}) - \frac{1}{2} ((\sigma^\mathtt{t})^2+(\sigma^{0\mathtt{t}})^2)(\pi^\mathtt{t}(0))^2 \mathbb{T}(d\mathtt{t}) = \overline{\tau\pi}. 
    \end{align*}
Further, since $W$, $\sigma^0$ and $\pi^\tau(0)$ are independent from the common noise drivers, we find that when taking conditional expectation with respect to $\mathcal{F}^0_t$ the $dW$-term vanishes as a martingale with zero expectation. However, for the $dW^0$-term, randomness remains in the types of the integrand, that is
	\[ \bar{\mathbb{E}}\left[ \int_0^t \sigma^0 \pi^\tau(0)dW^0_s \big| \mathcal{F}^0_t \right] = \int_0^t \bar{\mathbb{E}}\left[ \sigma^0 \pi^\tau(0) | \mathcal{F}^0_t \right] dW^0_s = \int_0^t \overline{\sigma^0\pi} dW^0_s \]
	for 
		\[\bar{\mathbb{E}}[\sigma^0 \pi^\tau(0)| \mathcal{F}^0_t] = \int_{\mathtt{T}} \sigma^{0\mathtt{t}} \pi^{\mathtt{t}}(0) \mathbb{T}(d\mathtt{t}) =  \overline{\sigma^0\pi}. \]
Finally, for the Poisson integral, we have to average over the types and the idiosyncratic jump marks, that is
	\begin{align*}
		\bar{\mathbb{E}}\bigg[ \int_{[0,t] \times E} &\log\big( 1+ \pi^\tau(\zeta^\tau(e)) \eta^\tau(e) \big) N(ds,de)  \bigg| \mathcal{F}^0_t \bigg] = \bar{\mathbb{E}}\left[ \sum_{t_k \leq t} \log\big( 1+ \pi^{\tau}(\zeta^{\tau}(e^{\mathtt{c}}_k,e^{\mathtt{i}}_k)) \eta^{\tau}(e^{\mathtt{c}}_k,e^{\mathtt{i}}_k) \big) \bigg| (t_k)_k, (e^{\mathtt{c}}_k)_k \right]\\
		&= \sum_{t_k \leq t} \int_{E^{\mathtt{i}}\times \mathtt{T}} \log\big( 1+ \pi^{\mathtt{t}}(\zeta^{\mathtt{t}}(e^{\mathtt{c}}_k,e^{\mathtt{i}})) \eta^{\mathtt{t}}(e^{\mathtt{c}}_k,e^{\mathtt{i}}) \big) \nu^{\mathtt{i}}(de^{\mathtt{i}}) \otimes \mathbb{T}(d\mathtt{t})\\
		&= \int_{[0,t]\times E^{\mathtt{c}}} \bigg( \int_{E^{\mathtt{i}}\times \mathtt{T}} \log\big( 1+ \pi^{\mathtt{t}}(\zeta^{\mathtt{t}}(e^{\mathtt{c}},e^{\mathtt{i}})) \eta^{\mathtt{t}}(e^{\mathtt{c}},e^{\mathtt{i}}) \big) \nu^{\mathtt{i}}(de^{\mathtt{i}}) \otimes \mathbb{T}(d\mathtt{t}) \bigg) N(ds,de^{\mathtt{c}},E^{\mathtt{i}})
	\end{align*}
	for the mark times $t_k$ and common and idiosyncratic jump marks $(e^{\mathtt{c}}_k,e^{\mathtt{i}}_k) \in E^{\mathtt{c}}\times E^{\mathtt{i}}$, $k=1,2,\dots$, set by the Poisson random measure $N$. Together with \eqref{eq:mean_jumps} and taking the exponential, we find \eqref{eq:closed_from_y} which coincides with the limit from Lemma \ref{lemma:limit}.
\end{proof}

As we will see soon, the above geometric mean field wealth works as a sufficient statistic of the control's environment enabling us to bypass the so-called master equation when considering best responses.\\

Let us fix some $\pi \in \bar{\mathcal{A}}_{\text{sig}}$. The geometric mean wealth is defined as $\bar{X}^\pi_T = \exp \bar{\mathbb{E}}[ \log X_T^\pi | \mathcal{F}^0_T]$. Then, the representative investor's best response problem is to
\begin{equation*}
	\text{ maximize } \bar{\mathbb{E}}[u(X^{\hat{\phi}}_T,\bar{X}^\pi_T)] = \bar{\mathbb{E}}\bigg[ \frac{(X^{\hat{\phi}}_T (\bar{X}^\pi_T)^{-\theta})^{1-\alpha}}{1-\alpha}\bigg] \text{ over } \hat{\phi}_t \in \bar{\mathcal{A}}
\end{equation*}
subject to her wealth process $X^{\hat{\phi}}$ from \eqref{eq:mean_field_wealth}. Here, $u=u_{\tau}$ is the utility function dependent on the type $\tau$, which almost surely behaves as the utility defined in \eqref{eq:utility_function}, i.e. 
	\[ u(\bar{\omega},x,\bar{x}):=u_{\tau(\bar{\omega})}(x,\overline{x})=u_{\mathtt{t}}(x,\overline{x})=u_{\alpha^{\mathtt{t}}}(x\bar{x}^{-\theta^{\mathtt{t}}}) \text{ for } \bar{\omega}=(\mathtt{t},\omega) \in \mathtt{T} \times \Omega = \bar{\Omega} . \]

Ultimately, the consistency condition characterizing a mean-field equilibrium requires that the best response 
\[
\hat{\phi}^{\star}_{\hat{\pi}^\star} := \argmax_{\hat{\phi}} \bar{\mathbb{E}}\big[u(X^{\hat{\phi}}_T,\bar{X}^{\hat{\pi}^\star}_T)\big]
\]
to the geometric mean wealth $\bar{X}^{\hat{\pi}^\star}$ induced by an equilibrium strategy $\hat{\pi}^\star \in \bar{\mathcal{A}}$ satisfies
\[ \bar{X}^{\hat{\pi}^\star}_T = \exp \big(\bar{\mathbb{E}}[\log X^{\hat{\phi}^{\star}_{\hat{\pi}^\star}}_T \mid \mathcal{F}^0_T]\big).
\]
As we will see, our framework yields a unique best response to any given environment. Hence, this consistency requirement can be thought of as or is at least directly implied by the best response coinciding precisely with the environment-generating strategy almost everywhere, i.e., $\hat{\phi}^{\star}_{\hat{\pi}^\star} = \hat{\pi}^\star$. We will exploit this sufficiency condition to establish the existence of an equilibrium in the subsequent proofs.

\begin{definition}\label{defn:mf_equilibrium}
    Let $\hat{\pi}^{\star} \in \bar{\mathcal{A}}$ be an admissible strategy and consider $\bar{X}^{\hat{\pi}^\star}_T=\exp \bar{\mathbb{E}}[\log X^{\hat{\pi}^{\star}}_T| \mathcal{F}^0_T]$, where $X^{\hat{\pi}^{\star}}$ is the wealth process corresponding to $\hat{\pi}^{\star}$. We say that $\hat{\pi}^{\star} $ is a \textit{mean field equilibrium} if
    \begin{equation}\label{eq:optimality_mf}
        \bar{\mathbb{E}}[u(X^{\hat{\pi}}_T,\bar{X}_T^{\hat{\pi}^\star}) ] \leq \bar{\mathbb{E}}[u(X^{\hat{\pi}^{\star}}_T,\bar{X}^{\hat{\pi}^\star}_T) ]  \text{ for all } \hat{\pi} \in \bar{\mathcal{A}}. 
    \end{equation}
    A mean field equilibrium $\pi^{\star}$ is called \textit{signal-driven} if  $\pi^{\star} \in \bar{\mathcal{A}}_{\text{sig}}$ is signal-driven in the sense of Definition~\ref{defn:signal-driven_strategy_mf}.
\end{definition}

As mentioned, we first solve the best response problem of the representative investor.

\subsection{HJB equation and Best Response Map for the Mean Field Game}\label{sec:mean_field_HJB} Let $\pi \in \bar{\mathcal{A}}_{\text{sig}}$ describe the investment decisions made in the mean field and consider the resulting geometric mean wealth $\bar{X}^\pi_T = \exp \bar{\mathbb{E}}[ \log X_T^\pi | \mathcal{F}^0_T]$ from \eqref{eq:closed_from_y}.

The following lemma collects some technical results ensuring, in particular, that the mean wealth behaves `nicely'. 

\begin{lemma}\label{lemma:mf_tools}
	We find that
	\begin{enumerate}
        \item[(i)]\namedlabel{itm:(i)}{(\textit{i})} signal-driven investment positions are uniformly bounded: There exists a constant $C>0$ such that for all $\pi \in \bar{\mathcal{A}}_{\text{sig}}$ and for $\mathbb{T}$-almost every $\mathtt{t} \in \mathtt{T}$
            \[ \max_{z \in \mathbb{R}} |\pi^\mathtt{t}(z)| \leq C. \]
		\item[(ii)]\namedlabel{itm:(ii)}{(\textit{ii})} there exists a constant $\bar{c}>0$ such that $0 < \bar{c} \le \nu^{\mathtt{c}}(de^{\mathtt{c}})\text{-}\essinf \overline{1+ (\pi \circ \zeta)\eta}(e^{\mathtt{c}})$ for all $\pi \in \bar{\mathcal{A}}_{\text{sig}}$,
		\item[(iii)]\namedlabel{itm:(iii)}{(\textit{iii})} for all $p>0$ there exists a $L^p(\nu^c)$-integrable function $f:E^c \to \mathbb{R}$ such that for all $\pi \in \bar{\mathcal{A}}_{\text{sig}}$, we have
   	 		\[  \overline{1+ (\pi \circ \zeta)\eta}(e^{\mathtt{c}})  \leq f(e^c).\]
	\end{enumerate}
	Further, from \ref{itm:(ii)} and \ref{itm:(iii)} it follows that for $\mathbb{T}$-almost every $\mathtt{t} \in \mathtt{T}$, for all $z \in \mathbb{R}$ and any $\varphi \in \Phi^\mathtt{t}(z)$ there exists a $\nu$-integrable function $g:E \to \mathbb{R}$ such that for all $\pi \in \bar{\mathcal{A}}_{\text{sig}}$ we have
    \begin{equation}\tag{\textit{iv}}\namedlabel{eq:(iv)}{(\textit{iv})}
            \begin{split}
                 \left| u_{\mathtt{t}}(1+ \varphi \eta^\mathtt{t}(e^{\mathtt{c}},e^{\mathtt{i}}),\overline{1+ (\pi \circ \zeta)\eta}(e^{\mathtt{c}})) \right| \leq g(e).
            \end{split}
        \end{equation}
\end{lemma} 

Note that, by \ref{itm:(ii)}, the mean wealth dynamics in \eqref{eq:closed_from_y} remain strictly positive. Also, \ref{itm:(i)} together with \ref{itm:T4}, ensures that $\overline{\tau\pi}$ is finite and $\sigma^0\pi(0)$ is uniformly bounded with respect to $\mathbb{T}$ independent from the choice of $\pi$. Further, \ref{itm:(iii)} and \eqref{eq:(iv)} ensure finiteness of the suprema of the HJB equation below, see \eqref{eq:HJB_MF}, as well as sufficient continuity of the target functions of the best response maps, see \eqref{eq:mf_optimal_zero} and \eqref{eq:mf_optimal_nonzero}, and, later, also continuity of the fixed point maps in the proof of existence of mean field equilibria.

\begin{proof}
	Let $\pi \in \bar{\mathcal{A}}_{\text{sig}}$. 
    For \ref{itm:(i)}, it is straightforward that \ref{itm:T5} bounds signal-driven investment positions uniformly in types and signals independently from $\pi$ as
    \begin{equation*}
	   \max_{z \in \mathbb{R}} |\pi^\mathtt{t}(z)| \leq  \max_{z \in \mathbb{R}} \max\left\{ \bigg|\frac{1}{\bar{\eta}^\mathtt{t}(z)}\bigg|, \bigg|\frac{1}{\underline{\eta}^\mathtt{t}(z)}\bigg| \right\} \leq |(c_0)^{-1}|.
    \end{equation*}
    
	For \ref{itm:(ii)}, it suffices to show that there exists $\tilde{c} \in \mathbb{R}$ independent from $\pi$ such that for all $e^{\mathtt{c}} \in E^{\mathtt{c}}$, we have that
		\[ -\infty < \tilde{c} \leq (\nu^{\mathtt{i}}(de^{\mathtt{i}}) \otimes \mathbb{T}(d\mathtt{t}))\text{-}\essinf\big(\log( 1+ \pi^{\mathtt{t}}(\zeta^{\mathtt{t}}(e^{\mathtt{c}},e^{\mathtt{i}})) \eta^{\mathtt{t}}(e^{\mathtt{c}},e^{\mathtt{i}})) \big),  \]
	because then the exponential $\overline{1+ (\pi \circ \zeta)\eta}$ is uniformly bounded away from zero on $E^{\mathtt{c}}$; and indeed, this follows immediately from Lemma \ref{lemma:mf_lower_bound}.
	
	For \ref{itm:(iii)}, observe that for all $e^\mathtt{c} \in E^\mathtt{c}$, by \ref{itm:(i)},  we find
		\begin{align*}
			\overline{1+ (\pi \circ \zeta)\eta}(e^{\mathtt{c}}) &\le \int_{\mathtt{T}\times E^{\mathtt{i}}} \big( 1+ \pi^{\mathtt{t}}(\zeta^{\mathtt{t}}(e^{\mathtt{c}},e^{\mathtt{i}})) \eta^{\mathtt{t}}(e^{\mathtt{c}},e^{\mathtt{i}}) \big) \nu^{\mathtt{i}}(de^{\mathtt{i}}) \otimes \mathbb{T}(d\mathtt{t})\\
            &\le (1+|(c_0)^{-1}|) \int_{\mathtt{T}\times E^{\mathtt{i}}} (1+|\eta^{\mathtt{t}}(e^{\mathtt{c}},e^{\mathtt{i}})|) \nu^{\mathtt{i}}(de^{\mathtt{i}}) \otimes \mathbb{T}(d\mathtt{t})
		\end{align*}
	where the right-hand side is independent of $\pi$ and implies existence of $f$ as above by \ref{itm:T6}.
	
	Lastly, concerning \eqref{eq:(iv)}, we find that
		\begin{align*}
			\begin{split}
                 \left| u_{\mathtt{t}}(1+ \varphi \eta^\mathtt{t}(e^{\mathtt{c}},e^{\mathtt{i}}),\overline{1+ (\pi \circ \zeta)\eta}(e^{\mathtt{c}})) \right| \leq \frac{(1+\varphi \eta^\mathtt{t}(e))^{1-\alpha^\mathtt{t}}}{|1-\alpha^\mathtt{t}|}\big( \overline{1+ (\pi \circ \zeta)\eta}(e^{\mathtt{c}}) \big)^{-\theta^\mathtt{t}(1-\alpha^\mathtt{t})}
            \end{split}
		\end{align*}
    where both for $\alpha^\mathtt{t} \in (0,1)$ and $\alpha^\mathtt{t} \in (1,\infty)$ the right-hand side exhibits a denominator which is uniformly bounded away from zero by Lemma \ref{lemma:mf_lower_bound} and \ref{itm:(ii)}, and a numerator which, by \ref{itm:T6} together with \ref{itm:(i)} and \ref{itm:(iii)}, is bounded from above by a $\nu$-integrable function which is independent from $\pi$.
\end{proof}

Returning to our initial goal, for given $\pi \in \bar{\mathcal{A}}_{\text{sig}}$, we search the best response $\hat{\phi}^\star \in \bar{\mathcal{A}}$ determined by the problem to 
    \begin{equation}\label{eq:equivalent_optimization_problem_mf}
        \text{maximize } \bar{\mathbb{E}}[u(X^{\hat{\phi}}_T,\bar{X}^\pi_T)] \text{ over } \hat{\phi} \in \bar{\mathcal{A}}.
    \end{equation}
To make this problem amenable to dynamic programming we write
	\[ \bar{\mathbb{E}}[u(X^{\hat{\phi}}_T,\bar{X}^\pi_T)] = \int_{\mathtt{T}} \mathbb{E}[u_\mathtt{t}(X^{\hat{\phi}^\mathtt{t}}_T,\bar{X}^\pi_T)] \mathbb{T}(d\mathtt{t})\]
and introduce the value function for the representative investor's problem when she is of type $\mathtt{t} \in \mathtt{T}$:
	 \begin{equation}\label{eq:value_function_mf}
           v^{\mathtt{t}}(x,\bar{x},T) := \sup_{\hat{\phi} \in \bar{\mathcal{A}}^{\mathtt{t}}} \mathbb{E}[u_{\mathtt{t}}(X^{\hat{\phi}}_T,\bar{X}^\pi_T) ].
        \end{equation}
Here we use the type-dependent set of admissible strategies
	\begin{equation*}
		\bar{\mathcal{A}}^{\mathtt{t}} := \{ \hat{\phi}(Z^{\mathtt{t}}) \text{ for some $\mathcal{P}(\mathcal{F}) \otimes \mathcal{B}(\mathbb{R})$-measurable field $\hat{\phi}$ with } \hat{\phi}_t(z) \in \bar{\Phi}^{\mathtt{t}}(z) \text{ for  all } z \in \mathbb{R},\ t \in [0,T] \},
    \end{equation*}
where the filtration $\mathcal{F}=(\mathcal{F}_t)_{t \in [0,T]}$ on $\Omega$ is induced by the completions of \[ \sigma((W^0_s,W_s, N([0,s]\times A)): s \in [0,t], A \in \mathcal{E}), \quad t \in [0,T],\] 
and for the resulting wealth dynamics:
    \begin{equation}\label{eq:dynamics_x}
    \begin{split}
    	X^{\hat{\phi}}_0 &= x_0^{\mathtt{t}} \\
    	\frac{dX_t^{\hat{\phi}}}{X^{\hat{\phi}}_{t-}} &= (r^\mathtt{t}+\hat{\phi}_t(0)(\kappa^{\mathtt{t}}-r^\mathtt{t}))dt +  \sigma^{\mathtt{t}} \hat{\phi}_t(0) dW_t + \sigma^{0\mathtt{t}} \hat{\phi}_t(0) dW^0_t + \int_E \hat{\phi}_t(\zeta^{\mathtt{t}}(e))\eta^{\mathtt{t}}(e) N(dt,de).
    \end{split}
    \end{equation}
We will later also consider the type-dependent set of signal-driven admissible strategies which we define to be
	\begin{align*}
		\bar{\mathcal{A}}^{\mathtt{t}}_{\text{sig}} := \{ \hat{\phi} \in \bar{\mathcal{A}}^{\mathtt{t}} \text{ such that } \hat{\phi}_t(\omega,Z^{\mathtt{t}}_t(\omega)) = \phi(Z^{\mathtt{t}}_t(\omega)) \text{ for } (\omega,t) \in \Omega \times [0,T] &\\ 
        \text{ for some }   & \text{$\mathcal{B}(\mathbb{R})$-measurable } \phi : \mathbb{R} \to \mathbb{R} \}.
    \end{align*}

Using the wealth dynamics~\eqref{eq:closed_from_y} and~\eqref{eq:dynamics_x}, we can follow the same heuristics as in the multi-agent game to derive the corresponding HJB equation:
    \begin{equation}\label{eq:HJB_MF}
    \begin{split}
        &-\partial_T v^{\mathtt{t}}(x,\bar{x},T) + \partial_x v^{\mathtt{t}}(x,\bar{x},T)xr^\mathtt{t} + \partial_y v^{\mathtt{t}}(x,\bar{x},T)y \big( \overline{\tau\pi} + \tfrac{1}{2}(\overline{\sigma^0 \pi})^2 \big)  + \tfrac{1}{2} \partial_{yy} v^{\mathtt{t}}(x,\bar{x},T) y^2 (\overline{\sigma^0 \pi})^2\\
        &\quad + \sup_{\varphi \in \bar{\Phi}^{\mathtt{t}}(0)} \bigg\{ \partial_x v^{\mathtt{t}}(x,\bar{x},T)x \varphi(\kappa^{\mathtt{t}}-r^\mathtt{t}) + \tfrac{1}{2} \partial_{xx} v^{\mathtt{t}}(x,\bar{x},T) x^2 \varphi^2 ((\sigma^{\mathtt{t}})^2+(\sigma^{0\mathtt{t}})^2) +\partial_{xy} v^{\mathtt{t}}(x,\bar{x},T) x y \sigma^{0\mathtt{t}}\varphi \overline{\sigma^0 \pi}  \\
        &\quad \quad \quad \quad \quad \quad \quad \quad + \int_{\{\zeta^{\mathtt{t}} = 0 \}}  \bigg[ v^{\mathtt{t}}(x(1+\varphi \eta^{\mathtt{t}}(e)),\bar{x}(\overline{1+ (\pi \circ \zeta)\eta}(e)),T) - v^{\mathtt{t}}(x,\bar{x},T) \bigg] \nu(de)  \bigg\}  \\
        &\quad + \int_{E\setminus \{0\}} \sup_{\varphi \in \bar{\Phi}^{\mathtt{t}}(z)} \bigg\{ \int_{\{\zeta^{\mathtt{t}} = z \}}\bigg[ v^{\mathtt{t}}(x(1+\varphi \eta(e)),\bar{x}(\overline{1+ (\pi \circ \zeta)\eta}(e)),T) - v^{\mathtt{t}}(x,\bar{x},T) \bigg] K^{\mathtt{t}}(z,de) \bigg\} \mu^{\mathtt{t}}(dz) = 0.
    \end{split}
    \end{equation}
together with the boundary condition
    	\[ v^{\mathtt{t}}(x,\bar{x},0)= u_{\mathtt{t}}(x,\bar{x}) = u_{\alpha^{\mathtt{t}}}(x\bar{x}^{-\theta^{\mathtt{t}}}),\]
and where, for simplicity of notation, we write $\overline{1+ (\pi \circ \zeta)\eta}(e):=\overline{1+ (\pi \circ \zeta)\eta}(e^{\mathtt{c}})$ for $e=(e^{\mathtt{c}},e^{\mathtt{i}}) \in E$, cf. \eqref{eq:mean_jumps}. Moreover, we provide the following verification result:
  
\begin{lemma}\label{lemma:verification_mf}
	Fix $\pi \in \bar{\mathcal{A}}_{\text{sig}}$. Then, for $\mathbb{T}$-almost every $\mathtt{t}$, the value function $v^{\mathtt{t}}$ of \eqref{eq:value_function_mf} is given by
	\begin{equation}\label{eq:value_fct_mf}
        v^{\mathtt{t}}(x,\bar{x},T) = u_{\alpha^{\mathtt{t}}}(x\bar{x}^{-\theta^{\mathtt{t}}})e^{T(1-\alpha^{\mathtt{t}})(r^\mathtt{t}+M^{\mathtt{t}}_{\pi})} \quad \text{ for }(t,x,\bar{x}) \in [0,T]\times (0,\infty)\times (0,\infty)
    \end{equation}
   	for the finite constant
   	\begin{equation}\label{eq:M_mf}
	\begin{split}
        M^{\mathtt{t}}_{\pi} &:=  -\theta^{\mathtt{t}} \overline{\tau\pi} + \tfrac{1}{2}(\theta^{\mathtt{t}})^2(1-\alpha^{\mathtt{t}})(\overline{\sigma^0 \pi})^2  \\
        & \quad \quad + \sup_{\varphi \in \bar{\Phi}^{\mathtt{t}}(0)} \bigg\{ \varphi(\kappa^{\mathtt{t}}-r^\mathtt{t}) - \tfrac{1}{2} \alpha^{\mathtt{t}} \varphi^2 ((\sigma^{\mathtt{t}})^2+(\sigma^{0\mathtt{t}})^2)  - \theta^{\mathtt{t}}(1-\alpha^{\mathtt{t}}) \sigma^{0\mathtt{t}}\varphi \overline{\sigma^0\pi}\\
        & \qquad \qquad \qquad \qquad \qquad \qquad +\int_{\{\zeta^{\mathtt{t}} = 0 \}} \bigg[ u_{\mathtt{t}}(1+\varphi \eta^{\mathtt{t}}(e),\overline{1+ (\pi \circ \zeta)\eta}(e)) - u_{\mathtt{t}}(1,1) \bigg] \nu(de)  \bigg\}  \\
        &\quad \quad +\int_{\zeta^{\mathtt{t}}(E)\setminus \{0\}} \sup_{\varphi \in \bar{\Phi}^{\mathtt{t}}(z)} \bigg\{ \int_{\{\zeta^{\mathtt{t}} = z \}}\bigg[ u_{\mathtt{t}}(1+\varphi \eta^{\mathtt{t}}(e),\overline{1+ (\pi \circ \zeta)\eta}(e) )  - u_{\mathtt{t}}(1,1) \bigg] K^{\mathtt{t}}(z,de) \bigg\}\mu^{\mathtt{t}}(dz).
    \end{split}
    \end{equation}
    Moreover, the value function $v^\mathtt{t}$ is indeed a smooth solution of the HJB equation \eqref{eq:HJB_MF} and the best response control is again a signal-driven strategy $\phi_\pi^{\mathtt{t},\star} \in \bar{\mathcal{A}}^\mathtt{t}_{\text{sig}}$: at times $t$ without a signal, i.e., when $z=Z^{\mathtt{t}}_t=0$, it is given by the default investment 
    \begin{equation}\label{eq:mf_optimal_zero}
        \begin{split}
            \phi_\pi^{\mathtt{t},\star}(0) &= \argmax_{\varphi \in \bar{\Phi}^\mathtt{t}(0)} \bigg\{\varphi(\kappa^\mathtt{t}-r^\mathtt{t}) - \tfrac{1}{2} \alpha^\mathtt{t} \varphi^2 ((\sigma^\mathtt{t})^2+(\sigma^{0\mathtt{t}})^2)  - \theta^\mathtt{t}(1-\alpha^\mathtt{t}) \sigma^{0\mathtt{t}}\varphi \overline{\sigma^0\pi} \\
            & \qquad \qquad \qquad \qquad \qquad \qquad + \int_{\{\zeta^\mathtt{t} = 0 \}} \bigg[ u_\mathtt{t}(1+\varphi \eta^\mathtt{t}(e),\overline{1+ (\pi \circ \zeta)\eta}(e) ) - u_{\mathtt{t}}(1,1) \bigg] \nu(de)  \bigg\},
        \end{split}
    \end{equation}
    and, at moments $t$ when receiving a signal $z=Z^{\mathtt{t}}_t\neq 0$, it is best to invest the fraction
    \begin{equation}\label{eq:mf_optimal_nonzero}
    	\phi_\pi^{\mathtt{t},\star}(z) = \argmax_{\varphi \in \bar{\Phi}^\mathtt{t}(z)} \bigg\{ \int_{\{\zeta^\mathtt{t} = z \}}\bigg[ u_{\mathtt{t}}(1+\varphi \eta^\mathtt{t}(e),\overline{1+ (\pi \circ \zeta)\eta}(e)) - u_{\mathtt{t}}(1,1) \bigg] K^\mathtt{t}(z,de) \bigg\}.
    \end{equation}    
\end{lemma}

This lemma closely resembles the corresponding result, Theorem~\ref{thm:verification_n}, in the multi-agent game. Consequently, the proof requires only minor adaptations to account for the different mean wealth. For sake of brevity, we omit it here and confine ourselves to pointing the reader to our standing Assumption~\ref{assumption:T} and Lemma~\ref{lemma:mf_tools} which give us all the tools to argue exactly as in the proof of Theorem~\ref{thm:verification_n}.

Having thus found the best responses of each type of agent, we can now piece these together to find the best response control of our representative agent and thus solve her best response problem \eqref{eq:equivalent_optimization_problem_mf}:

\begin{theorem}\label{thm:mf_best_response}
	The best response control $\phi^\star_\pi \in \bar{\mathcal{A}}$ to $\pi \in \bar{\mathcal{A}}_{\text{sig}}$ in the sense of~\eqref{eq:equivalent_optimization_problem_mf} is an admissible signal-driven strategy $\phi^{\star}_\pi = \phi^{\tau,\star}_\pi \in \bar{\mathcal{A}}_{\text{sig}}$ with $(\phi^{\mathtt{t},\star}_\pi)_{\mathtt{t} \in \mathtt{T}}$ given by \eqref{eq:mf_optimal_zero} and \eqref{eq:mf_optimal_nonzero}.
\end{theorem}

\begin{proof}
	For given $\pi \in \bar{\mathcal{A}}_{\text{sig}}$, we need to show that the strategy $\phi^\star_\pi = (\phi^{\mathtt{t},\star}_\pi)_{\mathtt{t} \in \mathtt{T}}$ given by \eqref{eq:mf_optimal_zero} and \eqref{eq:mf_optimal_nonzero} maximizes
	\begin{equation*}
        \bar{\mathbb{E}}[u(X^{\hat{\phi}}_T,\bar{X}^\pi_T)]
    \end{equation*}
    over all $\hat{\phi}$ in $\bar{\mathcal{A}}$. By Lemma \ref{lemma:verification_mf}, for any $\hat{\phi} \in \bar{\mathcal{A}}$, it holds that
    \begin{equation}\label{eq:bound_control}
    	\bar{\mathbb{E}}[u(X^{\hat{\phi}}_T,\bar{X}^\pi_T)] = \int_{\mathtt{T}} \mathbb{E}[u_\mathtt{t}(X^{\hat{\phi}^\mathtt{t}}_T,\bar{X}^\pi_T)] \mathbb{T}(d\mathtt{t}) \leq \int_{\mathtt{T}} \sup_{\hat{\phi}^\mathtt{t} \in \bar{\mathcal{A}}^\mathtt{t} }\mathbb{E}[u_\mathtt{t}(X^{\hat{\phi}^\mathtt{t}}_T,\bar{X}^\pi_T)] \mathbb{T}(d\mathtt{t}) = \int_{\mathtt{T}} v^{\mathtt{t}}(x,\bar{x},0) \mathbb{T}(d\mathtt{t}).
    \end{equation}
	Defining $\phi^\star_\pi(\bar{\omega}) = \phi^{\mathtt{t},\star}_\pi(Z^{\mathtt{t}}(\omega))$ for $\bar{\omega}=(\mathtt{t},\omega) \in \bar{\Omega}$ with $\phi^{\mathtt{t},\star}_\pi(z)$ given by \eqref{eq:mf_optimal_zero} and \eqref{eq:mf_optimal_nonzero}, we find that $\phi^\star_\pi$ indeed attains the upper bound in \eqref{eq:bound_control} as soon as we have shown that $\phi^\star_\pi$ is a $\mathcal{T} \otimes \mathcal{B}(\mathbb{R})$-measurable field. Since, by \ref{itm:T3}, there are only finitely many signals for each type, it suffices to ensure measurable dependence of \eqref{eq:mf_optimal_zero} and \eqref{eq:mf_optimal_nonzero} on $\mathtt{t} \in \mathtt{T}$. For this, note that the target functions in \eqref{eq:mf_optimal_zero} and \eqref{eq:mf_optimal_nonzero} are Carathéodory functions: they are continuous in $\varphi$ by dominated convergence using Lemma \ref{lemma:mf_tools} \eqref{eq:(iv)}, and as functions of $\mathtt{t}$ they are $\mathcal{T}$-measurable. Together with Lemma~\ref{lemma:weak_mb_Phi} this allows us to apply \cite[Theorem 18.19]{Aliprantis2006} to show that the unique maximizers are indeed $\mathcal{T}$-measurable. In conclusion, $\phi^{\star}_\pi$ is an admissible signal-driven strategy and the best response control to $\pi$ in the sense of~\eqref{eq:equivalent_optimization_problem_mf}.
\end{proof}

\subsection{Existence of a Mean Field Equilibrium}\label{sec:mean_field_existence}

Similarly to the multi-agent game, the main task now is to define appropriate domain and image sets which ensure that Schauder's Fixed Point Theorem can be applied to a suitable fixed point mapping. However, here, instead of finitely many investors, we will rely on a restriction to finitely many types of investors, or on a restriction to finitely many different common noise marks, i.e. $|E^\mathtt{c}|< \infty$. Consequently, we will provide two existence results, the first of which treats the setting of finitely many types:

\begin{theorem}\label{thm:mean_field_equilibrium_finite}
	Under Assumption~\ref{assumption:T}, there exists a signal-driven mean field equilibrium $\pi^{\star}$ if $|\supp \mathbb{T}| < \infty$.
\end{theorem}

\begin{proof}
    We need to find $\pi^\star \in \bar{\mathcal{A}}_{\text{sig}}$ which coincides with its best response, i.e. which satisfies $\phi^{\star}_{\pi^\star}=\pi^{\star}$, where $\phi^{\star}_{\pi^\star}$ is the optimal best control from Theorem \ref{thm:mf_best_response} corresponding to $\pi^\star$. In fact, we, then, have that the consistency condition for the mean field equilibrium
		\[ \bar{X}^{\pi^\star}_T= \exp \bar{\mathbb{E}}[ \log X_T^{\pi^\star}| \mathcal{F}^0_T]= \exp \bar{\mathbb{E}}[\log X_T^{\phi^{\star}_{\pi^\star}}|\mathcal{F}_T^0],\]
	holds. 
    
    This fixed point problem can be reformulated as follows. Let us recall that $\supp \mathbb{T} =: \bar{\mathtt{T}} \subset \mathtt{T}$ is finite by assumption and that each type can receive only finitely many distinct signals due to assumption~\ref{itm:T3}. Therefore, the set of admissible signal-driven strategies $\bar{\mathcal{A}}_{\text{sig}}$ can be identified with the compact, convex and non-empty set
	\begin{equation*}
		\bar{\Phi} := \bigtimes_{\mathtt{t} \in \bar{\mathtt{T}}} \bigtimes_{z \in \zeta^{\mathtt{t}}(E)\cup\{0\}} \bar{\Phi}^{\mathtt{t}}(z).
	\end{equation*}
	This is the domain and image set of the fixed-point mapping
		\begin{align*}
    		\gamma_{\text{MF}} :  \bar{\Phi} \to \bar{\Phi}, \quad \pi = (\pi^{\mathtt{t}})_{\mathtt{t} \in \bar{\mathtt{T}}} \mapsto \phi^{\star}_{\pi} = (\phi^{\mathtt{t},\star}_{\pi})_{\mathtt{t} \in \bar{\mathtt{T}}}.
    	\end{align*}
    Let us show continuity of $\gamma_{\text{MF}}$. First, we observe that, by Lemma~\ref{lemma:mf_tools}\ref{itm:(i)} together with~\ref{itm:T4} and Lemma~\ref{lemma:mf_tools}\ref{itm:(iii)}, both $\pi \mapsto \overline{\sigma^0\pi}$ and $\pi \mapsto  \overline{1+(\pi \circ \zeta)}(e^{\mathtt{c}})$ are continuous on $\bar{\Phi}$, for each $e^{\mathtt{c}} \in E^{\mathtt{c}}$. Secondly, Lemma~\ref{lemma:mf_tools} \eqref{eq:(iv)} allows us by dominated convergence to deduce that the target functions of the best response control $\phi^{\mathtt{t},\star}_{\pi}$ in \eqref{eq:mf_optimal_zero} and \eqref{eq:mf_optimal_nonzero} are continuous in $(\varphi,\pi)$. Thus, by Berge's Maximum Theorem (cf. \cite[Theorem 17.31]{Aliprantis2006}) the maximizers \eqref{eq:mf_optimal_zero} and \eqref{eq:mf_optimal_nonzero} are upper-hemicontinuous in $\pi$. Since these $\argmax$-sets are actually singletons, the maximizers are continuously depending on $\pi$. As a consequence, $\gamma_{\text{MF}}$ is indeed continuous on $\bar{\Phi}$.
    
    Finally, since $\bar{\Phi}$ is non-empty, compact and convex and $\gamma_{\text{MF}}$ is a continuous function, there exists $\pi^{\star} \in \bar{\Phi}$ such that $\pi^{\star}=\gamma_{\text{MF}}(\pi^{\star})$ by Schauder's Fixed Point Theorem, cf. \cite{Schauder1930}.
\end{proof}

Note that the previous proof uses the assumption $|\supp \mathbb{T}| < \infty$ in an essential way to conclude continuity of the fixed-point mapping on a compact set. Indeed, when $|\supp \mathbb{T}| = \infty$, defining an appropriate domain and image set for such a fixed-point mapping becomes a more delicate task as we can no longer rely on a finite-dimensional representation of the representative investor's strategy. In the following we propose a way to circumvent this technical issue by assuming that there are only finitely many common noise marks. Indeed, under this assumption, the next lemma provides existence of a suitable fixed point mapping. 

\begin{lemma}\label{lemma:K}
    Suppose $|\supp \mathbb{T}| = \infty$ and $|E^{\mathtt{c}}|< \infty$. Then, there exists a constant $\bar{C} >0$ such that for all $\pi \in \bar{\mathcal{A}}_{\text{sig}}$, $\mathbb{T}$-almost every $\mathtt{t}$ and $\nu$-almost every $e =(e^{\mathtt{c}},e^{\mathtt{i}}) \in E$, we have
	\begin{equation}\label{eq:bdd_log}
		|\log(1+\pi^{\mathtt{t}}(\zeta^{\mathtt{t}}(e^{\mathtt{c}},e^{\mathtt{i}}))\eta^\mathtt{t}(e^{\mathtt{c}},e^{\mathtt{i}}))| \leq \bar{C}. 
	\end{equation}
    Moreover, there exist a non-empty convex and compact $K \subset \mathbb{R} \times \mathbb{R}^{|E^c|}$ with the following properties:
    \begin{itemize}
        \item[(i)] For all $\pi \in \bar{\mathcal{A}}_{\text{sig}}$,  we have 
        \begin{align*}
		\mathtt{m}_\pi:=(\mathtt{m}_\pi(0),(\mathtt{m}_{\pi}(e^{\mathtt{c}}))_{e^{\mathtt{c}}\in E^{\mathtt{c}}}):= (\overline{\sigma^0\pi},( \overline{1+ (\pi \circ \zeta)\eta}(e^{\mathtt{c}}))_{e^{\mathtt{c}} \in E^{\mathtt{c}}})\in K.
	\end{align*}
    \item[(ii)] Every $m=(m(0),(m(e^{\mathtt{c}}))_{e^{\mathtt{c}} \in E^{\mathtt{c}}}) \in K$, yields an admissible  signal-driven strategy $\pi^\star_m \in \bar{\mathcal{A}}_{\text{sig}}$ where $\mathbb{T}$-a.e.\ type $\mathtt{t}$ will invest a fraction
    \begin{equation}\label{eq:selfmap_pi}
        \begin{split}
            &\pi^{\mathtt{t},\star}_m(z) =  
       \begin{cases}
       		\argmax\limits_{\varphi \in \bar{\Phi}^\mathtt{t}(0)} \bigg\{ \varphi(\kappa^\mathtt{t}-r^\mathtt{t}) - \tfrac{1}{2} \alpha^\mathtt{t} \varphi^2 ((\sigma^\mathtt{t})^2+(\sigma^{0\mathtt{t}})^2)  - \theta^\mathtt{t} (1-\alpha^\mathtt{t}) \sigma^{0\mathtt{t}}\varphi m(0) &\\
            \qquad\qquad\qquad\qquad+ \displaystyle\int_{\{\zeta^\mathtt{t} = 0 \}} \bigg[ u_\mathtt{t}(1+\varphi \eta^\mathtt{t}(e), m(e^{\mathtt{c}}) ) - u_\mathtt{t}(1,1) \bigg] \nu(de)  \bigg\} &\text{ if } z = 0, \\
        	\argmax\limits_{\varphi \in \bar{\Phi}^\mathtt{t}(z)} \bigg\{ \displaystyle\int_{\{\zeta^\mathtt{t} = z \}}\bigg[ u_\mathtt{t}(1+\varphi \eta^\mathtt{t}(e),m_\pi(e^{\mathtt{c}})) - u_\mathtt{t}(1,1) \bigg] K^\mathtt{t}(z,de) \bigg\} &\text{ if } z \neq 0,
        \end{cases}
        \end{split}
    \end{equation}
    of wealth in the stock when observing a signal $z$. 
    \end{itemize}
    In particular, 
    \begin{equation}\label{eq:selfmap}
        \gamma_{\text{\rm MF}}(m):=\mathtt{m}_{\pi^{\star}_m}
    \end{equation}
    yields a mapping $\gamma_{\text{\rm MF}}:K\to K$.
\end{lemma}

\begin{proof}
    First, we show \eqref{eq:bdd_log}. For this, let $\pi \in \bar{\mathcal{A}}_{\text{sig}}$, $\mathtt{t} \in \mathtt{T}$ and $e =(e^{\mathtt{c}},e^{\mathtt{i}}) \in E$. We find that, on $\{ 0 < 1 + \pi^{\mathtt{t}}(\zeta^{\mathtt{t}}(e^{\mathtt{c}},e^{\mathtt{i}}))\eta^\mathtt{t}(e^{\mathtt{c}},e^{\mathtt{i}}) < 1 \}$, the $\log$-term is negative but uniformly bounded from below by Lemma~\ref{lemma:mf_lower_bound}. For the complementary case, i.e., on $\{ 1 < 1 + \pi^{\mathtt{t}}(\zeta^{\mathtt{t}}(e^{\mathtt{c}},e^{\mathtt{i}}))\eta^\mathtt{t}(e^{\mathtt{c}},e^{\mathtt{i}}) \}$, the $\log$-term is positive and we find
    \begin{align*}
    	|\log(1+\pi^{\mathtt{t}}&(\zeta^{\mathtt{t}}(e^{\mathtt{c}},e^{\mathtt{i}}))\eta^\mathtt{t}(e^{\mathtt{c}},e^{\mathtt{i}}))|  \leq  \big(\max_{z \in \zeta^\mathtt{t}(E)\cup\{0\}} |\pi^\mathtt{t}(z)|\big) (1+\eta^\mathtt{t}(e^{\mathtt{c}},e^{\mathtt{i}})),
    \end{align*}
    where both factors on the right-hand side are uniformly bounded by Lemma \ref{lemma:mf_tools} \ref{itm:(i)} and  \ref{itm:T6} (read for the present case $|\supp \mathbb{T}| = \infty$).

    For the existence of $K$, we have that $\sigma^0\pi(0)$ is bounded on $\mathtt{T}$ independently from the choice of $\pi \in \bar{\mathcal{A}}_{\text{sig}}$ by Lemma \ref{lemma:mf_tools} \ref{itm:(i)} together with \ref{itm:T4}.
	Also, for the mean jumps, Lemma~\ref{lemma:mf_tools}\ref{itm:(ii)} provides the lower and \eqref{eq:bdd_log} the upper bound, that is, we immediately, find that $\bar{c} \leq m_\pi(e^\mathtt{c}) \leq \bar{C}$ for all $\pi \in \bar{\mathcal{A}}_{\text{sig}}$ for all $e^{\mathtt{c}} \in E^{\mathtt{c}}$. Putting everything together, this ensures the existence of $K$.

    Also, following the same arguments as in the proof of Theorem~\ref{thm:mf_best_response}, it is straightforward to show that $\pi^\star_m$ as defined in \eqref{eq:selfmap_pi} is admissible and signal-driven.
\end{proof}

Now, we are ready to prove the existence of a mean field equilibrium when there are not necessarily finitely many types.

\begin{theorem}\label{thm:mean_field_equilibrium} Under         Assumption~\ref{assumption:T}, there exists a signal-driven mean field equilibrium $\pi^{\star}$ if 
	$|\supp \mathbb{T}| = \infty$ and $|E^{\mathtt{c}}|< \infty$. 
\end{theorem}

\begin{proof}
	Let us show that the self-map $\gamma_{\text{MF}}$ from \eqref{eq:selfmap} is continuous on $K$. First, we show continuity of
        \[ (m(0),(m(e^{\mathtt{c}}))_{e^{\mathtt{c}} \in E^{\mathtt{c}}}) \mapsto \pi^{\mathtt{t},\star}(z) \]
    for all $z \in \zeta^\mathtt{t}(E) \cup\{0\}$, for $\mathbb{T}$-almost every $\mathtt{t}$. By Lemma \ref{lemma:mf_tools} \eqref{eq:(iv)} and dominated convergence, the target functions of the argmax-operators of $\pi^{\mathtt{t},\star}(z)$ from \eqref{eq:selfmap_pi} are continuous in $(\varphi,m(0),m(e^{\mathtt{c}}))_{e^{\mathtt{c}} \in E^{\mathtt{c}}}$. Then, by Berge's Maximum Theorem (cf. \cite[Theorem 17.31]{Aliprantis2006}) the $\argmax$-operators are upper-hemicontinuous in $(m(0),(m(e^{\mathtt{c}}))_{e^{\mathtt{c}} \in E^{\mathtt{c}}})$. As these $\argmax$-sets are actually singletons, the maximizers $\pi^{\mathtt{t},\star}_m(z)$ in \eqref{eq:selfmap_pi} are given by functions which are continuous in $(m(0),(m(e^{\mathtt{c}}))_{e^{\mathtt{c}} \in E^{\mathtt{c}}})$, for all $z \in \zeta^\mathtt{t}(E) \cup\{0\}$ and for $\mathbb{T}$-almost every $\mathtt{t}$.
    
    Secondly, recalling that
    \begin{equation}\label{eq:m_star}
        \begin{split}
		\mathtt{m}_{\pi^\star_m}(0) &= \overline{\sigma^0\pi^\star_m} = \int_{\mathtt{T}} \sigma^{0\mathtt{t}}\pi^{\mathtt{t},\star}_m(0) \mathbb{T}(d\mathtt{t}), \\
		\mathtt{m}_{\pi^\star_m}(e^{\mathtt{c}}) &= \overline{1+(\pi^\star_m \circ \zeta)}(e^{\mathtt{c}}) = \exp \left\{\int_{\mathtt{T}\times E^{\mathtt{i}}} \log(1+\pi^{\mathtt{t},\star}_m(\zeta^{\mathtt{t}}(e^{\mathtt{c}},e^{\mathtt{i}}))\eta^\mathtt{t}(e^{\mathtt{c}},e^{\mathtt{i}})) \nu^{\mathtt{i}}(de^{\mathtt{i}}) \otimes \mathbb{T}(d\mathtt{t}) \right\}.
        \end{split}
	\end{equation}
    for all $e^{\mathtt{c}} \in E^{\mathtt{c}}$, we find that the continuous dependence on $m=(m(0),(m(e^{\mathtt{c}}))_{e^{\mathtt{c}} \in E^{\mathtt{c}}})$ carries over because both integrands are uniformly bounded by Lemma~\ref{lemma:mf_tools}\ref{itm:(i)}, \ref{itm:T4} and \eqref{eq:bdd_log}, allowing us to apply dominated convergence. Thus, $\gamma_{\text{MF}}$ is continuous on $K$.
    
   As a consequence, $\gamma_{\text{MF}}$ is a continuous self-map on $K$. Since $K \subset \mathbb{R}\times \mathbb{R}^{|E^{\mathtt{c}}|}$ is a non-empty compact and convex set, Schauder's Fixed Point Theorem \cite{Schauder1930} gives existence of a fixed point $m^\star \in K$ such that $m^\star = \gamma_{\text{MF}}(m^\star)$. 
    
   Finally, we define the strategy $\pi^\star := \pi^\star_{m^\star}$. By the definition of $\gamma_{\text{MF}}$ in \eqref{eq:selfmap}, the fixed point property yields $m^\star = \mathtt{m}_{\pi^\star_{m^\star}} = \mathtt{m}_{\pi^\star}$, meaning the environment generated by this strategy $\pi^\star$ is precisely $m^\star$. By construction \eqref{eq:selfmap_pi}, $\pi^\star = \pi^\star_{m^\star}$ is the unique best response to the environment $m^\star$, and thus $\pi^\star$ is the best response to its own generated environment, i.e., $\phi^\star_{\pi^\star} = \pi^\star$. As discussed previously, this sufficiency relation ensures that the required consistency condition holds for $\pi^\star$, completing the proof of existence.
\end{proof}

\begin{remark}[Uniqueness of Equilibria]\label{remark:uniqueness}

    The uniqueness of the found signal-driven equilibria in the above multi-agent and mean-field games remains an open question. In \cite{LZ19}, the authors prove the uniqueness within the class of constant equilibria for which they can rely on explicit best response controls in this very small class. We, however, consider merely signal-wise constant equilibrium strategies, for which the best response controls are given by one-dimensional concave optimization problems. Here, even in case of submodular utility, i.e.\ with $\alpha \in (0,1)$, uniqueness of the found equilibria seems far from straightforward to prove, at least to us. 
    
    This lack of uniqueness in equilibrium strategies prevents us from making any convergence statements based on these strategies and their aggregate wealth.
    Nevertheless, we note that our numerical experiments suggest uniqueness of the identified equilibrium. Indeed, in all our simulations, the fixed-point iteration always converged to the same equilibrium, irrespective of the randomly chosen strategies by which we started the iteration. This observation provides numerical evidence for uniqueness inside the restricted signal-driven class. It does not reveal anything about the existence or non-existence of equilibria involving non-signal driven strategies. 
\end{remark}

\section{Case Studies and Financial-Economic Discussion}\label{sec:numerics}

Equilibrium strategies are not available in explicit form, making equilibrium behavior difficult to characterize. Therefore, we turn to case studies and numerical methods to analyze how strategic interactions shape investor welfare and risk perception and to address questions like how much we should care about what others know.

In the following case studies, we will examine scenarios where investors differ in signal quantity, signal quality, and competitiveness. By systematically varying these parameters, we will identify how they impact the certainty equivalent assessing the investors' happiness and perceived risk. Our findings reveal when and why investors prefer certain competitive dynamics, whether they benefit from informational advantages or even favor better informed peers. 

Our discussion focuses exclusively on mean field games for three reasons. First, we prefer scenarios where the influence of an individual on relative wealth is negligible and interactions are between the different types of players within the population. Secondly, and consequently, when individual decisions do not affect others, mean field games yield a more compelling framework than multi-agent games for our focus on (comparably simple) time-independent, signal-driven strategies. Finally, mean field games significantly reduce dimensionalty, simplifying the analysis and interpretation of investor behavior in large populations.

\subsection{A Single Stock Model with Two Investor Types Receiving Different Signals}\label{sec:numerics1}

For simplicity, we consider mean field games with two different types $\mathtt{t} \in \{\mathtt{A},\mathtt{B}\}=:\mathtt{T}$ of investors trading in a single common stock subject to log-normal price shocks, that is, the parameters $\sigma=0$ and $r,\kappa,\sigma^{0},\eta$ do not vary with the type but are fixed. The proportions of types are given by $\mathbb{T}(\{\mathtt{A}\})=p^{\mathtt{A}}$, $\mathbb{T}(\{\mathtt{B}\})=p^{\mathtt{B}} \in [0,1]$ with $p^{\mathtt{A}}+p^{\mathtt{B}}=1$. We put $E^{\mathtt{c}} := \mathbb{R}$ and $E^{\mathtt{i}} := \mathbb{R} \times [0,1]$, and denote by $e^{\mathtt{i}} := (e^{\mathtt{i}}_1,e^{\mathtt{i}}_2) \in E^{\mathtt{i}}$ a generic element in $E^{\mathtt{i}}$ which accounts for the independent idiosyncratic noise. Then, we specify the jump intensity measures
	\begin{equation*}
        \nu^{\mathtt{c}}(de^{\mathtt{c}}) = N_{0,1}(de^{\mathtt{c}}) \text{ and } \nu^{\mathtt{i}}(de^{\mathtt{i}}) = N_{0,1}(de_1^{\mathtt{i}}) \otimes \mathcal{U}_{[0,1]}(de_2^{\mathtt{i}}),
    \end{equation*}
	where $N_{0,1}$ denotes the standard normal distribution and $\mathcal{U}_{[0,1]}$ the uniform distribution on $[0,1]$. 
We define the common log-normal jumps as
		\[ \eta(e) := \eta(e^{\mathtt{c}}) = \exp(\hat{\sigma} e^{\mathtt{c}} + \hat{\kappa} - \tfrac{1}{2} \hat{\sigma}^2) -1 \quad \text{ for } e = (e^{\mathtt{c}},e^{\mathtt{i}}) \in E := E^{\mathtt{c}} \times E^{\mathtt{i}},  \]
	for additional common market parameters $\hat{\kappa} \in \mathbb{R}$, $\hat{\sigma} >0$. Note that $\nu(\{\eta < -1\}) = 0$.

The idiosyncratic noise enters through the individual signals our agents receive. Each type of investor $\mathtt{t} \in \{\mathtt{A},\mathtt{B}\}$ receives, in case of an imminent jump in the stock price, a non-zero signal with an individual probability denoted by $p_s^\mathtt{t} \in [0,1]$. Such a non-zero signal will be a categorical indication about the sign and the size of a idiosyncratically noisy version of the impending common price shock $\eta(e)=\eta(e^{\mathtt{c}})$. 

Specifically, we consider a perturbed version of this price shock
	\[ z^\mathtt{t}(e) := z^\mathtt{t}(e^\mathtt{c},e^\mathtt{i}_1) :=  \rho^{\mathtt{t}} e^{\mathtt{c}} + \sqrt{1-(\rho^{\mathtt{t}})^2} e_1^{\mathtt{i}}, \]
	where $|\rho^{\mathtt{t}}| < 1$ determines the signal quality for an investor of type $\mathtt{t} \in \{\mathtt{A},\mathtt{B}\}$. 
Since our existence theorems only work with finitely many distinct signals, we cannot use this $z^{\mathtt{t}}$ to define Type $\mathtt{t}$'s signal map $\zeta^{\mathtt{t}}$. Instead, we ensure $|\zeta^{\mathtt{t}}(E)| < \infty$ for $\mathtt{t} \in \{\mathtt{A},\mathtt{B}\}$ by categorizing the perturbed jump mark $z^{\mathtt{t}}(e)$ as ``small'', ``medium'', or ``large'' in absolute size and give an indication of the direction. Specifically, we put:
    \begin{align*}
        \zeta^{\mathtt{t}}(e) = \zeta^{\mathtt{t}}(e^\mathtt{c},e^\mathtt{i}_1,e^\mathtt{i}_2) = \mathds{1}_{[0,p_s^\mathtt{t}]}(e^{\mathtt{i}}_2) \cdot \mathrm{sign}(z^{\mathtt{t}}(e)) \cdot 
        \begin{cases}
           0.5 \quad & \text{ if } \quad  |z^{\mathtt{t}}(e)| \leq 0.5,\\
            1 \quad & \text{ if  } \quad  0.5 <|z^{\mathtt{t}}(e)| \leq 1,\\
             +\infty \quad & \text{ if } \quad |z^{\mathtt{t}}(e)| >1,
        \end{cases} 
    \end{align*} 

Thus, for both types the signal range is $\zeta^{\mathtt{t}}(E) \cup \{0\}=\{-\infty,-1,-0.5,0,0.5,1,\infty\}$. Still, signals may differ both in their frequency as parametrized by the probability $p^\mathtt{t}_s$ of signal reception and in their quality as determined by the perturbation parameter $\rho^\mathtt{t}$. \autoref{fig:signal1} and \autoref{fig:signal2} illustrate this choice of jump signals.

\begin{figure}[H]
  \centering
  \includegraphics[width=0.8\linewidth]{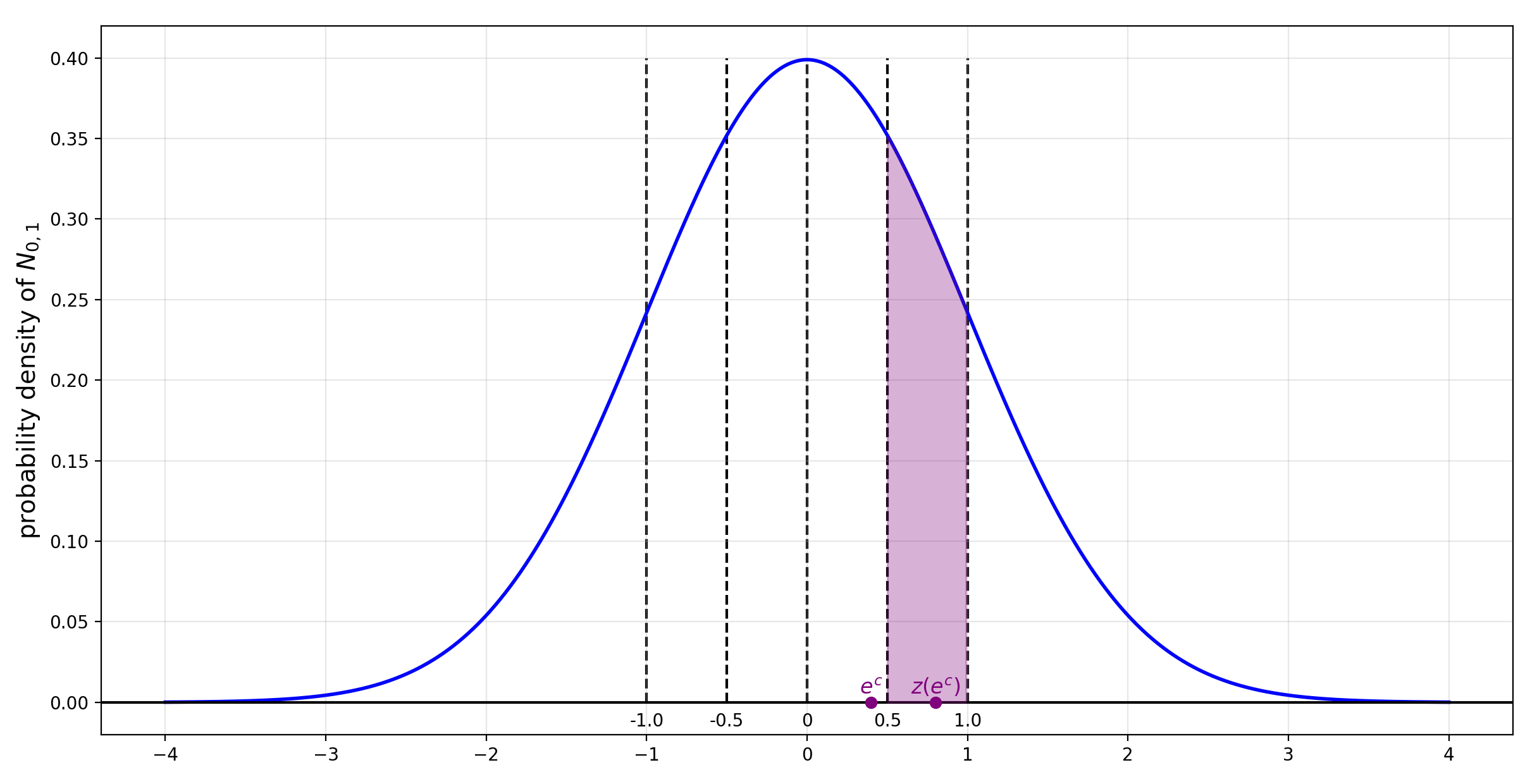}
\caption{For the original jump mark $e^{\mathtt{c}} = 0.4$ with perturbed jump mark $z^\mathtt{t}(e) = 0.8$ the signal $\zeta^\mathtt{t}(e)$ is $1$ indicating that the perturbed jump mark lies in the interval $(0.5,1]$.}
\label{fig:signal1}
\end{figure}

\begin{figure}[H]
  \centering
  \includegraphics[width=0.8\linewidth]{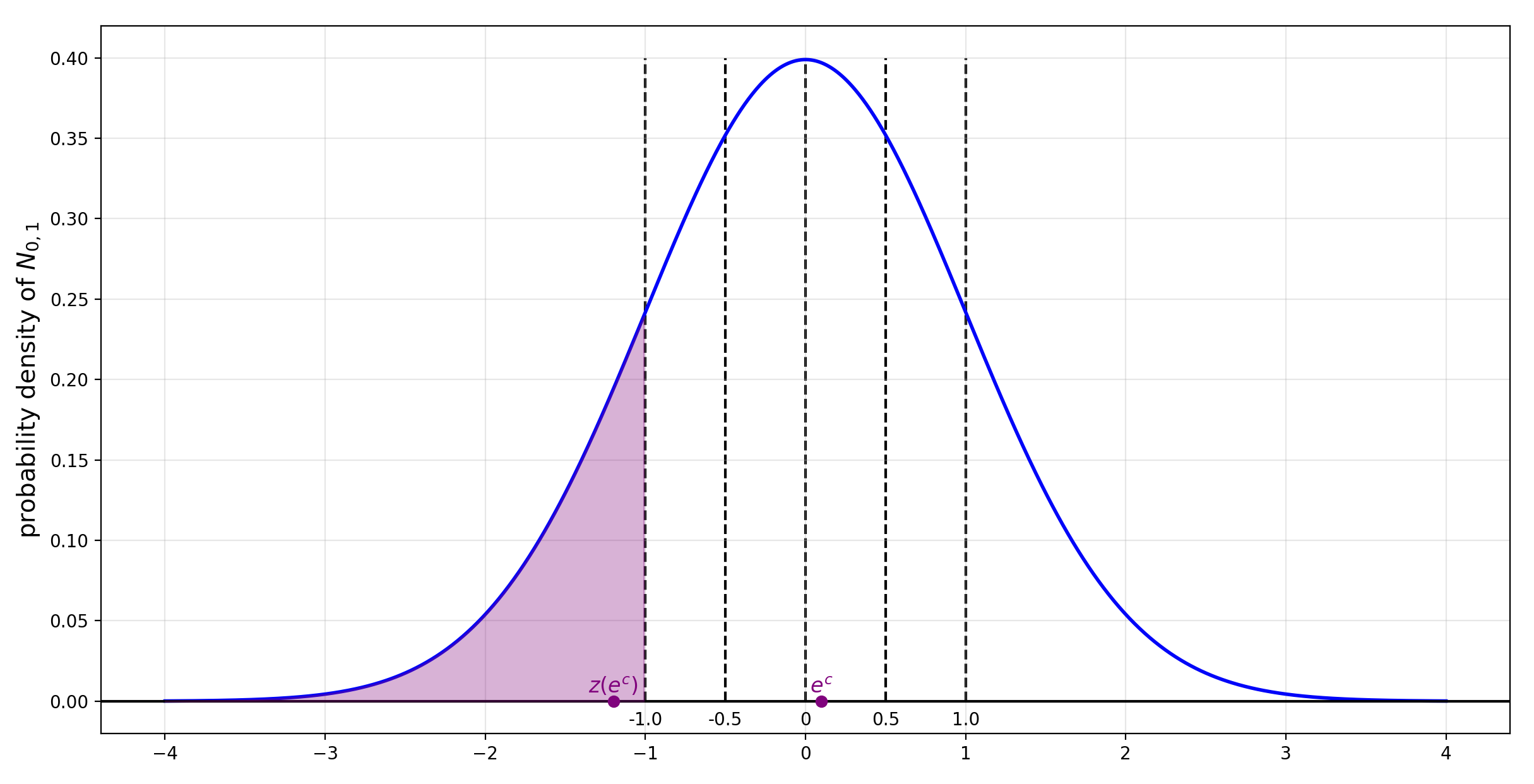}
\caption{For the original jump mark $e^{\mathtt{c}} = 0.1$ with perturbed jump mark $z^\mathtt{t}(e) = -1.2$ the signal $\zeta^\mathtt{t}(e)$ is $-\infty$ indicating that the perturbed jump mark lies in the interval $(-\infty,-1)$.}
\label{fig:signal2}
\end{figure}

Specifying the following intervals of integration given a signal $z \in \zeta(E)$
    \begin{align*}
        I(z) := \begin{cases}
            (1,+\infty) \quad & \text{ if } z=+\infty,\\
            (0.5,1] \quad & \text{ if } z=1,\\
            (0,0.5] \quad & \text{ if } z=0.5,\\
            [-0.5,0) \quad & \text{ if } z=-0.5,\\
            [-1,0.5) \quad & \text{ if } z=-1,\\
            (-\infty,-1) \quad & \text{ if } z=-\infty,\\
        \end{cases}
    \end{align*}
we find the disintegration of $\nu(\cdot \cap \{\zeta^{\mathtt{t}} \neq 0\})$ for type $\mathtt{t} \in \{\mathtt{A},\mathtt{B}\}$ to result in the signal frequency
   \begin{align*}
        \mu^{\mathtt{t}}(\{z\}) &= \lambda p^\mathtt{t}_s N_{0,1}(I(z)) \text{ for } z \in \zeta^{\mathtt{t}}(E) \setminus \{0\},
    \end{align*}
and, for a signal $z \in \zeta^{\mathtt{t}}(E) \setminus \{0\}$, we find the a posteriori distribution for the jump marks to be
	\begin{align*}
		K^\mathtt{t}(z,de) &= \frac{ \mathds{1}_{I(z)}(z^{\mathtt{t}}(e^{\mathtt{c}},e^{\mathtt{i}}_1)) }{N_{0,1}(I(z))}
        N_{0,1}(de^{\mathtt{c}}) \otimes N_{0,1}(de^{\mathtt{i}}_1).
	\end{align*}
    
We impose that the probability of receiving a signal in case of a jump satisfies $p_s^\mathtt{t} < 1$ and, hence, find that $\Phi^\mathtt{t}(z) = [0,1]$ for all $z \in \zeta(E) \cup \{0\}$.\\

In the remaining part of this paragraph, we explain how to efficiently compute the optimal investment positions for an investor of type $\mathtt{t} \in \{\mathtt{A},\mathtt{B}\}$. Readers may skip this section at first reading, and instead proceed directly to the numerical results and their discussion in Subsection~\ref{sec:numerics2}.

For a signal-based strategy $\pi \in \bar{A}_{\text{sig}}$, we recall that the representative investor's best response will be to invest a fraction
\begin{equation*}
        \begin{split}
            \phi_\pi^{\mathtt{t},\star}(0) &= \argmax_{\varphi \in \bar{\Phi}^\mathtt{t}(0)} \bigg\{\varphi(\kappa-r) - \tfrac{1}{2} \alpha^{\mathtt{t}} \varphi^2 (\sigma^2+(\sigma^0)^2)  - \theta^{\mathtt{t}}(1-\alpha^{\mathtt{t}}) \sigma^0\varphi \overline{\sigma^0\pi}  \\
            & \qquad \qquad \qquad + \frac{\lambda(1-p_s^\mathtt{t})}{1-\alpha^{\mathtt{t}}} \bigg[ \int_{\mathbb{R}} \big[ (1+\varphi \eta(e^{\mathtt{c}}))^{1-\alpha^{\mathtt{t}}} (\overline{1+ (\pi \circ \zeta)\eta}(e^{\mathtt{c}}) )^{-\theta^{\mathtt{t}}(1-\alpha^{\mathtt{t}})} - 1 \big] N_{0,1}(de^{\mathtt{c}}) \bigg] \bigg\}
        \end{split}
    \end{equation*}
when not receiving a signal.
We can explicitly compute
	\begin{equation*}
		\overline{\sigma^0\pi} = \bar{\mathbb{E}}[\sigma^0\pi(0)] = \sigma^0 (p^\mathtt{A} \pi^{\mathtt{A}}(0) + p^\mathtt{B} \pi^\mathtt{B}(0))
	\end{equation*}
	and
    \begin{equation}\label{eq:compute_mean}
    \begin{split}
        &\overline{1+ (\pi \circ \zeta)\eta}(e^{\mathtt{c}})\\ 
        &\quad = \exp\bigg\{ \sum_{\mathtt{t} \in \{ \mathtt{A},\mathtt{B}\}} p^{\mathtt{t}} \int_{E^{\mathtt{i}}} \log(1+\pi^{\mathtt{t}}(\zeta^{\mathtt{t}}(e^{\mathtt{c}},e^{\mathtt{i}}))\eta^{\mathtt{t}}(e^{\mathtt{c}},e^{\mathtt{i}})) \nu_{i}(de^{\mathtt{i}}) \bigg\}\\
            &\quad = \exp\bigg\{ p^{\mathtt{A}} \bigg( (1-p^{\mathtt{A}}_s) \log(1+\pi^{\mathtt{A}}(0)\eta(e^{\mathtt{c}}))  + p^{\mathtt{A}}_s \int_{\mathbb{R}} \log(1+\pi^{\mathtt{A}}(\zeta^{\mathtt{A}}(e^{\mathtt{c}},e^{\mathtt{i}}_1))\eta(e^{\mathtt{c}})) N_{0,1}(de_1^{\mathtt{i}}) \bigg) \\
            & \qquad \qquad \quad   + p^{\mathtt{B}} \bigg( (1-p^{\mathtt{B}}_s) \log(1+\pi^{\mathtt{B}}(0)\eta(e^{\mathtt{c}}))  + p^{\mathtt{B}}_s \int_{\mathbb{R}} \log(1+\pi^{\mathtt{B}}(\zeta^{\mathtt{B}}(e^{\mathtt{c}},e^{\mathtt{i}}_1))\eta(e^{\mathtt{c}})) N_{0,1}(de_1^{\mathtt{i}}) \bigg) \bigg\}.        
    \end{split}
    \end{equation}
Note that, given $e^{\mathtt{c}} \in E^\mathtt{c}$,  a non-zero signal $0 \neq z = \zeta^{\mathtt{t}}(e^{\mathtt{c}},e^{\mathtt{i}})$ allows one to deduce that the mark part $e^{\mathtt{i}}$ has to be in the interval
\begin{align*}
    I(z,e^c) := \begin{cases}
            \left( \frac{1 - \rho^{\mathtt{t}} e^{\mathtt{c}}}{\sqrt{1-(\rho^{\mathtt{t}})^2}}, + \infty \right) \quad & \text{ if } z=+\infty,\\
            \left( \frac{z - 0.5- \rho^{\mathtt{t}} e^{\mathtt{c}}}{\sqrt{1-(\rho^{\mathtt{t}})^2}}, \frac{z - \rho^{\mathtt{t}} e^{\mathtt{c}}}{\sqrt{1-(\rho^{\mathtt{t}})^2}} \right) \quad & \text{ if } z=0.5 \text{ or } z=1,\\
            \left( \frac{z - \rho^{\mathtt{t}} e^{\mathtt{c}}}{\sqrt{1-(\rho^{\mathtt{t}})^2}} , \frac{z + 0.5 - \rho^{\mathtt{t}} e^{\mathtt{c}}}{\sqrt{1-(\rho^{\mathtt{t}})^2}} \right) \quad & \text{ if } z=-0.5 \text{ or } z=-1,\\
            \left( -\infty, \frac{-1 - \rho^{\mathtt{t}} e^{\mathtt{c}}}{\sqrt{1-(\rho^{\mathtt{t}})^2}} \right) \quad & \text{ if } z=-\infty,\\
        \end{cases}
\end{align*}
Hence, the integrals in \eqref{eq:compute_mean} simplify:
    \begin{align*}
        \int_{\mathbb{R}} \log(1+\pi^{\mathtt{t}}(\zeta^{\mathtt{t}}(e^{\mathtt{c}},e^{\mathtt{i}}_1))\eta(e^{\mathtt{c}})) N_{0,1}(de_1^{\mathtt{i}}) = \sum_{z \in \zeta^{\mathtt{t}}(E) \setminus \{0\}} \log(1+\pi^{\mathtt{t}}(z)\eta(e^{\mathtt{c}})) N_{0,1}(I(z,e^{\mathtt{c}}))
    \end{align*}
which is easily computed using the cumulative distribution function of the standard normal distribution.

Similarly, we can compute the components of the best response to $\pi$ when receiving a non-zero signal $z\neq 0$:
    \begin{align*}
    	\phi_\pi^{\mathtt{t},\star}(z) &= \argmax_{\varphi \in \bar{\Phi}^\mathtt{t}(z)} \bigg\{ \int_{\mathbb{R}} \bigg[ \frac{(1+\varphi \eta(e^{\mathtt{c}}))^{1-\alpha^{\mathtt{t}}} (\overline{1+ (\pi \circ \zeta)\eta}(e^{\mathtt{c}}) )^{-\theta^{\mathtt{t}}(1-\alpha^{\mathtt{t}})} - 1}{(1-\alpha^{\mathtt{t}})} \bigg]  
        \frac{N_{0,1}(I(z,e^{\mathtt{c}}))}{N_{0,1}(I(z))}
        N_{0,1}(de^{\mathtt{c}}) \bigg\}.
    \end{align*}

\bigskip

\subsection{Numerical experiments}\label{sec:numerics2}
For our experiments,  we choose the market parameters to be
    \[ r= 0, \ \kappa = 0.08, \ \hat{\kappa} = 0, \ \sigma^0 = 0.3, \ \hat{\sigma} = 0.1, \ \lambda = 10, \]
and the time horizon for investments is $T=1$.\\

We investigate how the presence of investment signals impacts investors with relative performance concerns. For this, we consider investors of type $\mathtt{A}$ as a reference point and investigate how they assess their prospects when investors of type $\mathtt{B}$ receive more/less accurate or more/less frequent signals. To compare the outcomes of the different equilibria, we introduce a suitable certainty equivalent. The reference point for this certainty equivalent is in each case an equilibrium for the fully homogeneous reference type environment 
    \begin{align*}
        \mathtt{T}^{\text{ref}} &=  \{ (x_0^{\mathtt{A},\text{ref}}, p^{\mathtt{A},\text{ref}}, p^{\mathtt{A},\text{ref}}_s, \rho^{\mathtt{A},\text{ref}}, \theta^{\mathtt{A},\text{ref}},\alpha^{\mathtt{A},\text{ref}}), (x_0^{\mathtt{B},\text{ref}}, p^{\mathtt{B},\text{ref}}, p^{\mathtt{B},\text{ref}}_s, \rho^{\mathtt{B},\text{ref}}, \theta^{\mathtt{B},\text{ref}},\alpha^{\mathtt{B},\text{ref}}) \} \\
        &=  \{ (1, 0.5, 0.5, 0.5, 1.0 ,2.0), (1, 0.5, 0.5, 0.5, 1.0, 2.0) \}.
    \end{align*}
    For our primary analysis, we deliberately fix the concern parameters at their maximum, $\theta^{\mathtt{A},\text{ref}} = \theta^{\mathtt{B},\text{ref}} = 1.0$, and the relative risk aversion parameters at a moderate $\alpha^{\mathtt{A},\text{ref}} = \alpha^{\mathtt{B},\text{ref}} = 2.0$ to isolate the informational effects under conditions of maximal peer pressure. Unless stated otherwise, these parameters serve as our standard reference setting. 
    
    We then let the characteristics of the investors of type $\mathtt{B}$ deviate from those in the reference scenario; that is, we consider an alternative type environment 
    \begin{align*}
        \mathtt{T}^\text{alt} &= \{ (x_0^{\mathtt{A},\text{ref}}, 1 - p^{\mathtt{B},\text{alt}}, p^{\mathtt{A},\text{ref}}_s, \rho^{\mathtt{A},\text{ref}}, \theta^{\mathtt{A},\text{ref}},\alpha^{\mathtt{A},\text{ref}}), (x_0^{\mathtt{B},\text{alt}}, p^{\mathtt{B},\text{alt}}, p^{\mathtt{B},\text{alt}}_s, \rho^{\mathtt{B},\text{alt}}, \theta^{\mathtt{B},\text{alt}},\alpha^{\mathtt{B},\text{alt}}) \} 
    \end{align*}
for various choices of proportion $p^{\mathtt{B},\text{alt}}$ of type $\mathtt{B}$ investors, their signal frequency $p^{\mathtt{B},\text{alt}}_s$, and signal quality $\rho^{\mathtt{B},\text{alt}}$. Except for the final section where we explore the interaction between relative performance pressure and risk aversion, we keep the concern parameter $\theta^{\mathtt{B},\text{alt}} =1.0$ and relative risk aversion $\alpha^{\mathtt{B},\text{alt}}=2.0$ fixed at their reference levels.

For both the reference and alternative parameters, we compute the equilibrium strategies via numerical fixed-point iterations. Specifically, we iteratively evaluate the best response controls from Section~\ref{sec:numerics1} using numerical integration and optimization methods until the maximum absolute deviation between successive policy iterations falls below a predetermined tolerance. The expected utilities for investors of type $\mathtt{A}$ will typically differ between these equilibria. We then determine the certainty equivalent as the initial capital $x_0^\mathtt{A}$ that an investor of type $\mathtt{A}$ would require in the reference setting $\mathtt{T}^{\text{ref}}$ to achieve the same expected utility as in the alternative setting $\mathtt{T}^{\text{alt}}$: 
	\begin{equation*}
		x_0^\mathtt{A} = \exp \left(M^{\mathtt{A},\text{alt}}_{\pi^{\star,\text{alt}}} - M^{\mathtt{A},\text{ref}}_{\pi^{\star,\text{ref}}}\right)
	\end{equation*}
	for $M_{\pi^{\star,\text{ref}}}^{\mathtt{A},\text{ref}}$ and $M_{\pi^{\star,\text{alt}}}^{\mathtt{A},\text{alt}}$ as in \eqref{eq:M_mf} for corresponding type spaces $\mathtt{T}^\text{ref}$ and $\mathtt{T}^\text{alt}$ and equilibrium strategies $\pi^{\star,\text{ref}}$ and $\pi^{\star,\text{alt}}$, respectively.\\
	
For ease of presentation, we will sometimes refer to investors of type $\mathtt{A}$ as Type $\mathtt{A}$ and to investors of type $\mathtt{B}$ as Type $\mathtt{B}$. Also, whenever it is clear, we will drop the superscripts `ref' and `alt', and as we always consider equilibrium strategies we drop the $\star$-superscript for the strategies.\\

Below we first address the overarching question \textit{How much should we care about what others know?} by investigating the impact of Type $\mathtt{B}$'s parameter variations on the certainty equivalent $x_0^\mathtt{A}$. Second, we analyze Type $\mathtt{A}$'s and Type $\mathtt{B}$'s equilibrium strategy for changes in Type $\mathtt{B}$'s signal quality. Finally, we evaluate the efficiency of potential countermeasures, contrasting higher signal quality versus frequency under varying degrees of relative performance concerns and risk aversion.

\bigskip

\textbf{How much should we care about what others know?} \autoref{fig:plot_threat_power} illustrates how the certainty equivalent $x_0^{\mathtt{A}}$ of Type $\mathtt{A}$ responds to variations in the Type $\mathtt{B}$'s population share $p^{\mathtt{B},\text{alt}}$ and signal quality $\rho^{\mathtt{B},\text{alt}}$. The vertical $1.0$ contour line at $\rho^{\mathtt{B},\text{alt}} = 0.5$ refers back to Type $\mathtt{A}$'s utility in the symmetric reference setting and confirms a common intuition: any informational advantage in signal quality held by Type $\mathtt{B}$ ($\rho^{\mathtt{B},\text{alt}} > \rho^{\mathtt{B},\text{ref}}  = 0.5$) reduces Type $\mathtt{A}$'s utility ($x_0^{\mathtt{A}} < 1.0$), while any disadvantage ($\rho^{\mathtt{B},\text{alt}} < \rho^{\mathtt{B},\text{ref}} = 0.5$) benefits Type $\mathtt{A}$ ($x_0^{\mathtt{A}} > 1.0$), regardless of their share in the population.

\begin{figure}[h!]
     \centering
     \begin{subfigure}[b]{0.48\textwidth}
        \centering
        \includegraphics[width=\linewidth]{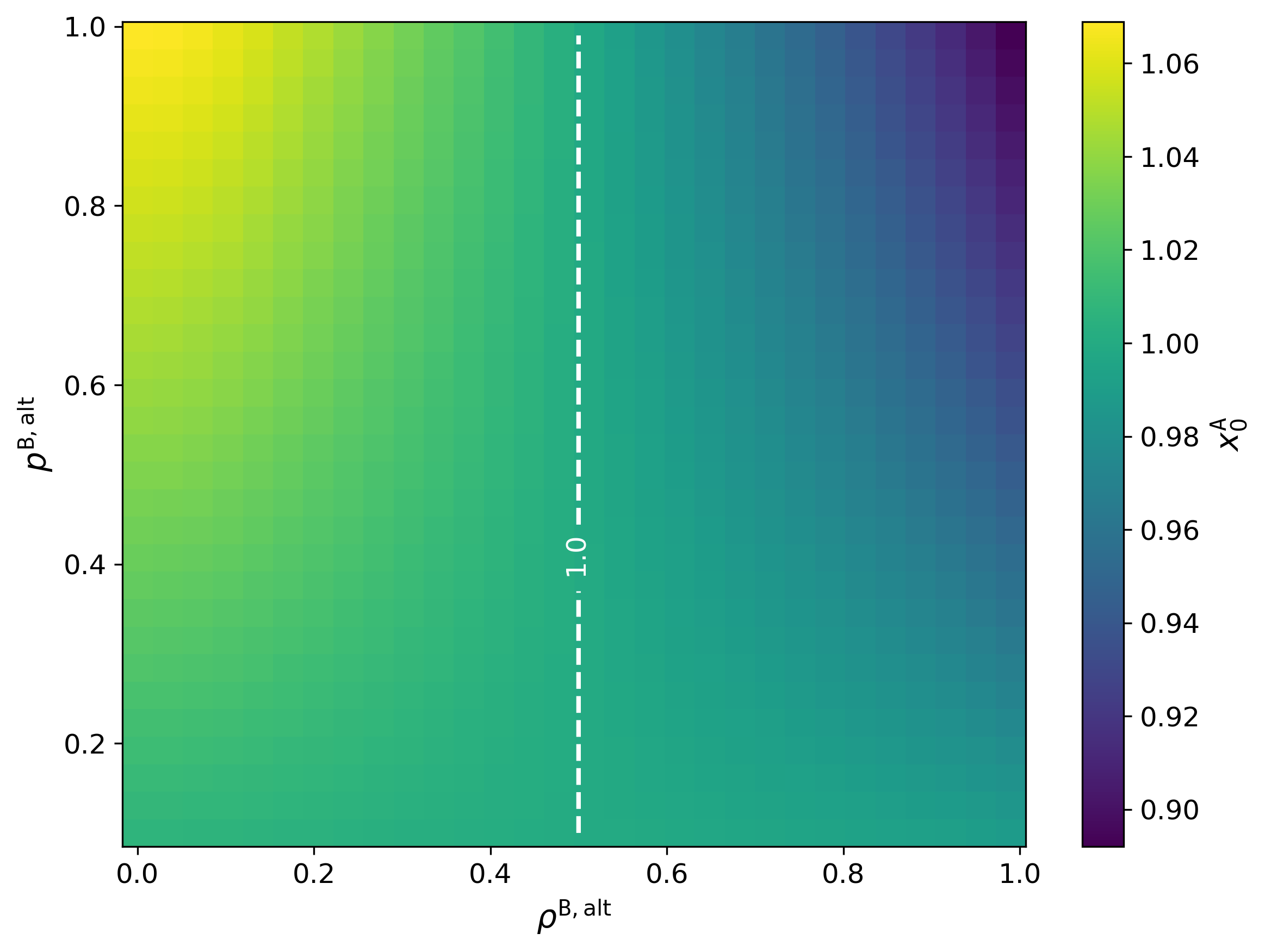}
        \caption{Impact of population share and signal quality.}
        \label{fig:plot_threat_power}
     \end{subfigure}
     \hfill
     \begin{subfigure}[b]{0.48\textwidth}
        \centering
        \includegraphics[width=\linewidth]{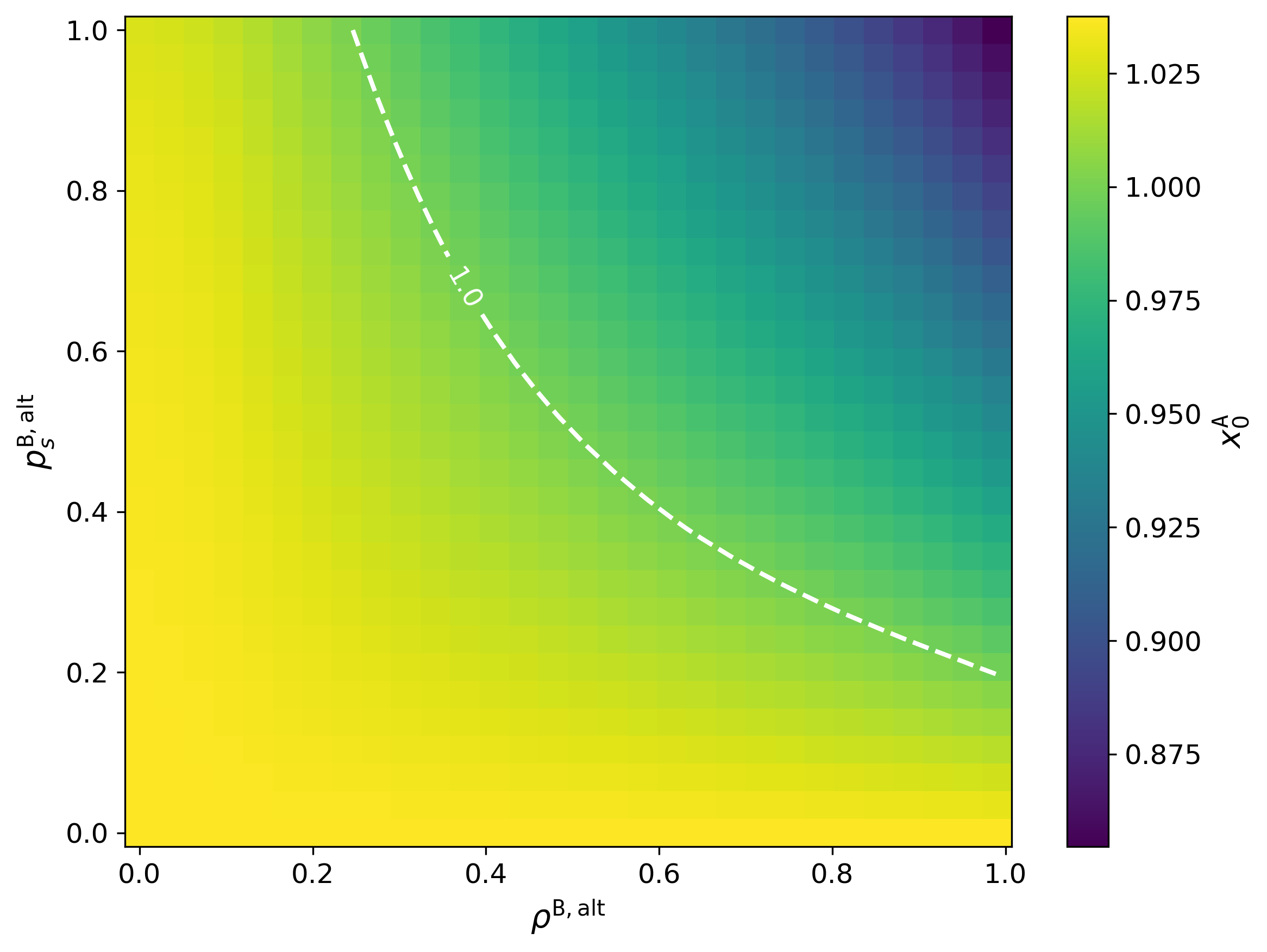}
        \caption{Impact of signal frequency and signal quality.}
        \label{fig:plot_threat_geometry}
     \end{subfigure}
\caption{Certainty equivalent $x_0^{\mathtt{A}}$ of Type $\mathtt{A}$ under varying parameters of Type $\mathtt{B}$. (A) illustrates the dependency on Type $\mathtt{B}$'s population share $p^{\mathtt{B},\text{alt}}$ and signal quality $\rho^{\mathtt{B},\text{alt}}$. (B) shows the interaction between Type $\mathtt{B}$'s signal frequency $p_s^{\mathtt{B},\text{alt}}$ and signal quality $\rho^{\mathtt{B},\text{alt}}$. The dashed white line represents the $x_0^{\mathtt{A}} = 1.0$ contour.}
\end{figure}

However, through the relative performance criterion, the magnitude of this utility shift is dictated by the population share of Type $\mathtt{B}$. On the one hand, when Type $\mathtt{B}$'s population share is small, the utility decay caused by their superior signal quality is marginal, as is the utility benefit when Type $\mathtt{B}$ holds inferior signal quality. On the other hand, when the population share of Type $\mathtt{B}$ is large, the changes to the certainty equivalent for Type $\mathtt{A}$ become substantially more pronounced. A structurally similar amplification effect is observed if we vary Type $\mathtt{B}$'s signal quantity instead of signal quality.

Furthermore, \autoref{fig:plot_threat_geometry} varies both Type $\mathtt{B}$'s alternative signal quantity $p_s^{\mathtt{B},\text{alt}}$ and their alternative signal quality $\rho^{\mathtt{B},\text{alt}}$. Here, the $1.0$ contour delineates the frontier where the combination of Type $\mathtt{B}$'s information parameters yields the same relative performance for Type $\mathtt{A}$ as in the reference equilibrium. The convexity of this frontier illustrates a trade-off: an advantage for Type $\mathtt{B}$ in one informational dimension only harms Type $\mathtt{A}$ if it is not offset by a sufficient disadvantage in the other relative to the reference setting. Consequently, both frequent but noisy signals, as well as rare but precise signals, actively benefit Type $\mathtt{A}$, pushing $x_0^{\mathtt{A}}$ above the $1.0$ baseline. Therefore, reductions in the certainty equivalent are confined to the top-right regime, where Type $\mathtt{B}$ possesses simultaneously high signal frequency and signal quality.

Both \autoref{fig:plot_threat_power} and \autoref{fig:plot_threat_geometry} reveal a striking asymmetry in the certainty equivalent. For instance, when Type $\mathtt{B}$ gains a complete informational advantage ($p_s^{\mathtt{B},\text{alt}}, \rho^{\mathtt{B},\text{alt}} \to 1$), Type $\mathtt{A}$'s certainty equivalent drops significantly (e.g., nearing $0.85$, representing a $-15\%$ downside). Conversely, when Type $\mathtt{B}$ is completely uninformed, the upside is capped much lower (e.g., near $1.05$, yielding only a $+5\%$ upside). This asymmetry, which structurally persists even at lower concern levels, such as $\theta=0.5$, demonstrates that in a relative performance setting, trailing a well-informed peer is harsher than the reward for beating a poorly-informed one. \\ 

Having established that information asymmetry poses a substantial threat to utility, we next investigate how the optimal equilibrium strategies of both types adapt to these informational shifts.\\

\begin{figure}[h]
\centering
 \includegraphics[width=1\linewidth]{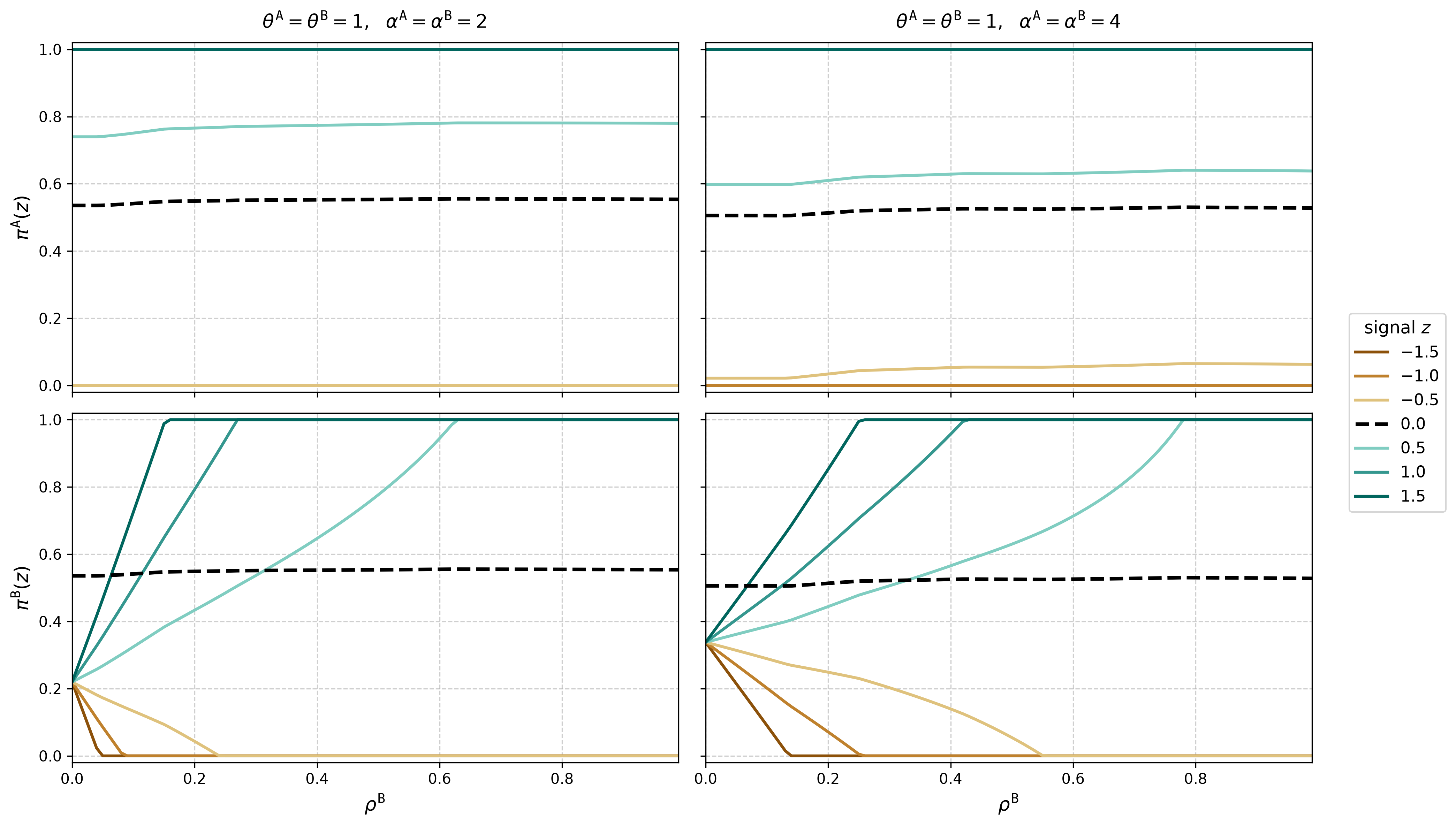}
\caption{Signal-dependent equilibrium strategies $\pi^{\mathtt{A}}(z)$ (top row) and $\pi^{\mathtt{B}}(z)$ (bottom row) as functions of Type $\mathtt{B}$'s signal quality $\rho^{\mathtt{B}}$, at maximal concern $\theta^{\mathtt{A}}=\theta^{\mathtt{B}}=1$ and two levels of risk aversion: $\alpha^{\mathtt{A}}=\alpha^{\mathtt{B}}=2$ (left) and $\alpha^{\mathtt{A}}=\alpha^{\mathtt{B}}=4$ (right). Each coloured curve corresponds to a signal value $z$; the dashed line is the no-signal strategy ($z=0.0$). Note that $\pi^{\mathtt{A}}(-1.5) \equiv \pi^{\mathtt{A}}(-1.0) \equiv 0$ in both environments.}
\label{fig:plot_flight}
\end{figure}

\textbf{How do investment decisions change?} \autoref{fig:plot_flight} displays both types' equilibrium strategies for each signal $z \in \{-1.5,-1.0,-0.5,0.0,0.5,1.0,1.5\}$ as Type $\mathtt{B}$'s signal quality $\rho^{\mathtt{B}}$ increases from $0$ to $1$, at maximal concern $\theta^{\mathtt{A}}=\theta^{\mathtt{B}}=1$ and two different levels of risk aversion ($\alpha^{\mathtt{A}}=\alpha^{\mathtt{B}}=2$ and $\alpha^{\mathtt{A}}=\alpha^{\mathtt{B}}=4$). The contrast between the types is stark, and it holds in both environments: while Type $\mathtt{B}$ reshapes its portfolio in response to better information, Type $\mathtt{A}$'s policy remains remarkably rigid, its zero-signal allocation changing by a negligible margin across the entire domain of the competitor's signal quality. This structural rigidity reveals that  Type $\mathtt{A}$ cannot easily trade its way out of a relative disadvantage.

This stoic response of Type $\mathtt{A}$ to changes in Type $\mathtt{B}$'s signal quality persists across may concern parameter and risk-aversion levels. 
On the flip side, this implies that to catch-up to better-informed peers, Type $\mathtt{A}$ must seek to improve its own information flow.

\bigskip

\textbf{Which adaptation is more efficient: increasing signal quality or quantity?} Since Type $\mathtt{A}$ cannot trade its way out of an informational disadvantage, its only recourse is to improve its own information which can be done in two ways: by increasing \emph{signal quality (precision)} $\rho^{\mathtt{A}}$ or \textit{signal quantity (frequency)} $p_s^{\mathtt{A}}$. \autoref{fig:plot_eff_defense} compares the effects of eitehr choice: in each economic environment it traces the two parity contours $x_0^{\mathtt{A}}=1.0$ that Type $\mathtt{A}$ must reach to offset Type $\mathtt{B}$'s rising signal quality $\rho^{\mathtt{B},\text{alt}}$, defending either by raising its own quality (red, $\rho^{\mathtt{A},\text{alt}}$) or its own frequency (blue, $p_s^{\mathtt{A},\text{alt}}$). Both curves are evaluated against a common baseline: the closer a contour remains to the horizontal 0.5 line, the less Type $\mathtt{A}$ must deviate from the reference to restore parity. Viewed this way, \autoref{fig:plot_eff_defense} shows that, for all three environments, restoring parity requires a smaller adjustment in signal quality than in frequency. Further experiments, not reported here, show that this qualitative trend also holds for parity contours responding to changes in Type $\mathtt{B}$'s signal quantity. Consequently, improving signal precision proves to be the more efficient lever.

\begin{figure}[h!]
     \centering
     \includegraphics[width=\linewidth]{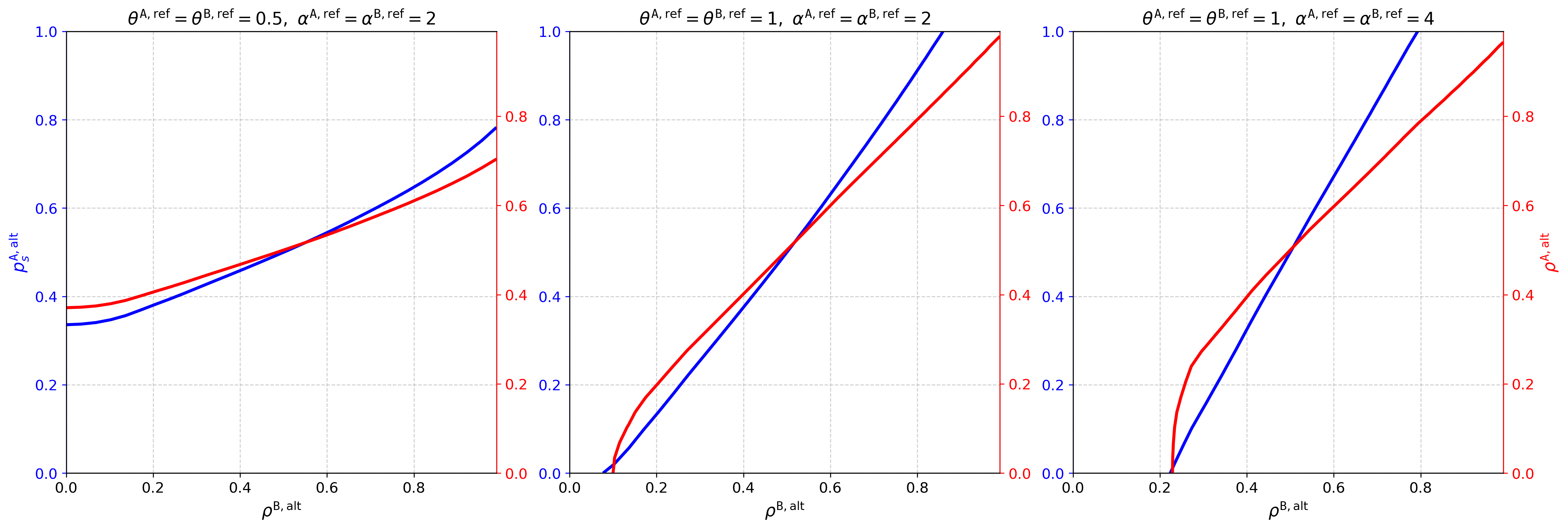}
\caption{Comparative efficiency of adapting signal quality $\rho^{\mathtt{A}}$ (red) versus signal frequency $p_s^{\mathtt{A}}$ (blue) to maintain parity ($x_0^{\mathtt{A}}=1.0$) against Type $\mathtt{B}$'s increasing signal quality $\rho^{\mathtt{B}}$. The panels illustrate how required parameter adaptations shift across different reference levels of relative performance concern $\theta$ and risk aversion $\alpha$.}
\label{fig:plot_eff_defense}
\end{figure}

In a disadvantaged situation ($\rho^{\mathtt{B},\text{alt}} > 0.5$), the reason signal quality is the more effective lever lies in the informational structure. What hurts Type $\mathtt{A}$ is the asymmetry in signal quality: a higher $\rho^{\mathtt{B},\text{alt}}$ lets Type $\mathtt{B}$ read the direction of an impending shock more accurately. At maximal concern ($\theta^{\mathtt{A},\text{ref}}=\theta^{\mathtt{B},\text{ref}}=1$), where everyone is graded purely on relative performance, the common signal quality restores parity close to the diagonal $\rho^{\mathtt{A},\text{alt}}\approx\rho^{\mathtt{B},\text{alt}}$ where the two types adopt near-identical signal responses and the mean field benchmark stays balanced. Under weaker concern ($\theta^{\mathtt{A},\text{ref}}=\theta^{\mathtt{B},\text{ref}}=0.5$) investors put emphasis on their individual wealth. Thus, in order to match the performance in the reference setting, Type $\mathtt{A}$ only requires modest increases in signal quality or quantity; so the contour is flatter.
By contrast, increasing signal quantity can only ensure parity as long as the signal quality $\rho^{\mathtt{B},\text{alt}}$ has not grown too much: beyond $\rho^{\mathtt{B},\text{alt}}\approx 0.9$ even being warned about every single price shock ($p_s^{\mathtt{A},\text{alt}}=1$) does not make up for the signal quality disadvantage $\rho^{\mathtt{A},\text{ref}} = 0.5$.

Conversely, the advantageous situation ($\rho^{\mathtt{B},\text{alt}} < 0.5$) is more surprising: as $\rho^{\mathtt{B},\text{alt}}\to0$, even for pure noise signals ($\rho^{\mathtt{A},\text{alt}}=0$) or without any signals ($p_s^{\mathtt{A},\text{alt}}=0$), Type $\mathtt{A}$ is better off than in the reference setting. Here, due to the  high relative performance concerns $\theta^{\mathtt{A},\text{ref}}=\theta^{\mathtt{B},\text{ref}}=1$, investors seek to track their competitors' average wealth. This tracking becomes easier as the competitors' signal quality deteriorates since their investment choices narrow to a universal position in the limit; cf. \autoref{fig:plot_flight}. For higher risk-averion this narrowing of investment choices is more pronounced. As a result the cut-off level in the right panel is about twice as high as in the middle pannel of \label{fig:plot_eff_defense}.

\bibliographystyle{alpha}
\bibliography{references}

\end{document}